\documentclass[11pt]{amsart}
\usepackage{hyperref}
\usepackage{graphicx}
\usepackage{amsmath}
\usepackage{amssymb}
\usepackage{color}
\usepackage{mathdots}
\usepackage{bbm}
\usepackage{amsfonts}
\usepackage{graphicx}
\usepackage{mathrsfs}
\usepackage{mathtools}

\newcommand{\powerset}{\raisebox{.15\baselineskip}{\Large\ensuremath{\wp}}}
\makeatletter
\def\@tocline#1#2#3#4#5#6#7{\relax
  \ifnum #1>\c@tocdepth 
  \else
    \par \addpenalty\@secpenalty\addvspace{#2}%
    \begingroup \hyphenpenalty\@M
    \@ifempty{#4}{%
      \@tempdima\csname r@tocindent\number#1\endcsname\relax
    }{%
      \@tempdima#4\relax
    }%
    \parindent\z@ \leftskip#3\relax \advance\leftskip\@tempdima\relax
    \rightskip\@pnumwidth plus4em \parfillskip-\@pnumwidth
    #5\leavevmode\hskip-\@tempdima
      \ifcase #1
       \or\or \hskip 1em \or \hskip 2em \else \hskip 3em \fi%
      #6\nobreak\relax
    \dotfill\hbox to\@pnumwidth{\@tocpagenum{#7}}\par
    \nobreak
    \endgroup
  \fi}
\makeatother
\newtheorem{thm}{Theorem}[subsection]
\newtheorem{cor}[thm]{Corollary}

\newtheorem{lem}[thm]{Lemma}

\theoremstyle{definition}
\newtheorem{defn}[thm]{Definition}
\theoremstyle{remark}
\newtheorem{rem}[thm]{Remark}

\numberwithin{equation}{section}

\newcommand{\gF}{\mathfrak{F}}
\newcommand{\gS}{\mathfrak{S}}
\newcommand{\gB}{\mathfrak{B}}
\newcommand{\cI}{\mathcal{I}}

\newcommand{\gG}{\mathfrak{G}}
\newcommand{\gI}{\mathfrak{I}}
\newcommand{\gJ}{\mathfrak{J}}

\newcommand{\gP}{\mathfrak{P}}
\newcommand{\sB}{\mathscr{B}}

\newcommand{\sL}{\mathscr{L}}
\newcommand{\cX}{\mathscr{X}}
\newcommand{\sA}{\mathscr{A}}
\newcommand{\sC}{\mathscr{C}}

\newcommand{\spin}{{\rm spin}}

\newcommand{\NN}{\mathbb{N}}
\newcommand{\II}{\mathbb{I}}

\newcommand{\cS}{\mathcal{S}}

\newcommand{\supp}{{\rm supp}\,}

\newcommand{\cB}{\mathcal{B}}

\newcommand{\cE}{\mathcal{E}}

\newcommand{\cH}{\mathcal{H}}
\newcommand{\cP}{\mathcal{P}}
\newcommand{\cG}{\mathcal{G}}
\newcommand{\cQ}{\mathcal{Q}}

\newcommand{\RR}{\mathbb{R}}

\newcommand{\CC}{\mathbb{C}}
\newcommand{\DD}{\mathbb{D}}
\newcommand{\cD}{\mathcal{D}}
\newcommand{\cL}{\mathcal{L}}

\newcommand{\TT}{\mathbb{T}}

\newcommand{\wt}{\widetilde}
\newcommand{\Ran}{{\rm Ran}\,}
\newcommand{\Ker}{{\rm Ker}\,}

\begin{document}
\title[Spectral theorem for normal operators on a Clifford module]{The spectral theorem for normal operators on a Clifford module}%

\author[F. Colombo]{Fabrizio Colombo}
\address{(FC) Dipartmento di Matematica \\
Politecnico di Milano \\
Via E. Bonardi, 9 \\
I-20133 Milano, Italy }
\email{fabrizio.colombo@polimi.it}

\author[D. P. Kimsey]{David P. Kimsey}
\address{(DPK) School of Mathematics, Statistics and Physics \\
Newcastle University\\
Newcastle upon Tyne NE1 7RU UK}
\email{david.kimsey@ncl.ac.uk}

\date{\today}

\begin{abstract}
In this paper, using the recently discovered notion of the $S$-spectrum, we prove the spectral theorem for a bounded or unbounded normal operator on a Clifford module (i.e., a two-sided Hilbert module over a Clifford algebra based on units that all square to be $-1$). Moreover, we establish the existence of a Borel functional calculus for bounded or unbounded normal operators on  a Clifford module. Towards this end, we have developed many results on functional analysis, operator theory, integration theory and measure theory in a Clifford setting which may be of an independent interest. Our spectral theory is the natural spectral theory for the Dirac operator on manifolds in the non-self adjoint case. Moreover, our results provide a new notion of spectral theory and a Borel functional calculus for a class of $n$-tuples of commuting or non-commuting operators on a real or complex Hilbert space. Moreover, for a special class of $n$-tuples of operators on a Hilbert space our results provide a complementary functional calculus to the functional calculus of J. L. Taylor.

 \end{abstract}

\maketitle

\medskip
\noindent AMS Classification 47A10, 47A60.


\tableofcontents

\section{Introduction}
\label{sec:INTRO}

The spectral theorem for a normal operator on a complex Hilbert space is an incredibly elegant result which lies at the heart of operator theory, harmonic analysis and mathematical physics. In this paper we endeavour to generalise the spectral theorem to a noncommutative setting where a complex Hilbert space is replaced by a Hilbert module over a Clifford algebra $\cH_n := \cH \otimes \RR_n$, where $\cH$ is a real Hilbert space and $\RR_n := \RR_{0,n}$ is a Clifford algebra which is generated by units $e_0 = 1$ and $e_1, \ldots, e_n$, where $e_i^2 = -1$ and $e_ie_j = - e_j e_i$ for $i,j=1,\ldots, n$ with $n > 0$ (we will favour the term Clifford module for brevity) and the notion of spectrum is replaced by the recently discovered notion of $S$-spectrum. We wish to stress that given the well-known classification of Clifford algebras $\RR_n$ for $n=1,2,\ldots$ (see, e.g., \cite{SpinGeometry}), our results can easily be translated to handle spectral theory for a linear operator on any Hilbert module over a finitely generated unital ring.

The $S$-spectrum can be characterised by the invertibility of a second order operator and it is defined to be a subset of the set of paravectors in $\RR_n$ (where one can instil a natural complex structure corresponding to any paravector). More precisely, corresponding to a right linear operator $T$ on a Clifford module $\cH_n$ over $\RR_n$, we define
\begin{equation}
\label{eq:Sspectrumintro}
\sigma_S(T) := \{ s = \sum_{j=0}^n s_j e_j: ( T^2 - 2s_0 T + (s_0^2+\ldots + s_n^2) I )^{-1} \notin \cB(\cH_n) \},
\end{equation}
where $\cB(\cH_n)$ denotes the set of bounded linear operators on $\cH_n$. The $S$-spectrum was discovered by the first author and I. Sabadini in the context of arbitrary operators on a quaternionic Banach space (i.e., a Clifford module over $\RR_2$) and paravector operators $T$ on a Banach module over a Clifford algebra $\cX_n = \cX \otimes \RR_n$, where $\cX$ is a real Banach space and  $\RR_n$ is as above with $n > 2$,  i.e., right linear operators of the form
$$T = T_0 + \sum_{j=1}^n e_j T_j ,$$
where $T_0,\ldots, T_n$ are linear operators on the real Banach space $\cX$ (see Appendix \ref{app:Spectrum} for more details).

Recently we realised that we can dispense with the restriction that $T$ is a paravector operator, i.e., the $S$-spectrum can be defined for an arbitrary right linear operator on $\cX_n$.  This observation will turn out to be critical for this manuscript. Moreover,  the consequences of this observation on the $S$-functional calculus and the slice hyperholomorphic function theory have been investigated in \cite{CGKSfunc} and \cite{CKPS}, respectively.

It turns out that Clifford numbers of the form $s = \sum_{j=0}^n s_j e_j$ embed naturally into a complex plane and hence
$$\sigma_S(T) \subseteq \bigcup_{\gI \, \in \, \mathbb{S}}\CC_{\gI},$$
where $\CC_{\gI} = \{ \lambda_1 + \lambda_2 \gI: \lambda_1,\lambda_2 \in \RR \}$ and $\mathbb{S} := \{ s = \sum_{j=1}^n s_j e_j: s^2 = -1 \}$. A moment's consideration of \eqref{eq:Sspectrumintro} will reveal that the upper half plane $\sigma_S(T) \cap \CC_{\gI}^+$ can identified in a natural way with $\sigma_S(T) \cap \CC_{\gJ}^+$ for any choice of $\gI, \gJ \in \mathbb{S}$. Thus, one may think of the $S$-spectrum as a complex notion of spectrum with an elegant symmetry. For a concise background of the history of the $S$-spectrum, related function theory and known results in the quaternionic case see Appendix \ref{app:Spectrum}.

We shall see that corresponding to every densely defined normal operator $T$ on a Clifford module $\cH_n$, there exists an imaginary operator $J$ (i.e., $J^*J = I$ and $J^* = -J$) and a uniquely determined spectral measure $E := E_{\gI}$ such that
\begin{equation}
\label{eq:SPECTRALBEGIN}
T = \int_{\CC_{\gI}^+} {\rm Re}(\lambda) dE(\lambda) + \int_{\CC_{\gI}^+} {\rm Im}(\lambda) dE(\lambda) \, J \quad \quad {\rm for} \quad \gI \in \mathbb{S}.
\end{equation}
We shall also see that
\begin{equation}
\label{eq:SPECTRALSUPPORT}
\supp \, E = \sigma_S(T) \cap \CC_{\gI}^+.
\end{equation}
Moreover, we establish a full analogue of spectral integrals corresponding to a spectral measure $E$, the Borel functional calculus associated with $T$ and the spectral mapping theorem in the bounded and unbounded case. The fact that $J$ is an imaginary operator on $\cH_n$ allows one to think of $J$ as an operator-valued analogue of an imaginary unit. It is worth pointing out that in the present setting the spectral measure $E$ gives rise to a family of $\RR_n$-valued measures which are positive as elements of the Clifford algebra $\RR_n$. This creates a significant technical difficulty when building the requisite machinery to prove \eqref{eq:SPECTRALBEGIN} and the Borel functional calculus.

\newpage

\noindent {\bf Motivation}

\bigskip

\noindent We mention some of the main motivations to study the spectral theorem on a Clifford module.

\medskip
\medskip

\noindent (I) {\em Spectral theory for vector operators}.

\medskip

\noindent The spectral theory on the $S$-spectrum is a very natural tool for studying vector operators
that come from vector analysis
such as the gradient operators with nonconstant coefficients
\begin{equation}\label{FOUR}
T=\sum_{j=1}^n e_j a_j(x)\partial_{x_j}, \ \ \ x=(x_1,...,x_n)\in \mathbb{R}^n
\end{equation}
where $e_j$ is an orthonormal basis in $\mathbb{R}^n$ and
$a_j:\Omega\to \mathbb{R}$, $j=1,...,n$
 are given functions with suitable regularity and $\Omega \subseteq \mathbb{R}^n $ is an open set.
 With our spectral theory we can define, for example, the fractional powers of operators of the form (\ref{FOUR}) which can be used to represent fractional Fourier's law for the propagation of the heat in nonhomogeneous materials
contained in $\Omega$.
In the quaternionic setting, fractional Fourier's law,
 the has been considered in various papers, see for example \cite{frc1,frc2},  and the references therein.

\medskip
\medskip

\noindent (II) {\em Dirac operator on manifolds}.

\medskip

\noindent In a great preponderance of the papers on spectral theory for the Dirac operators the self-adjoint case is considered. This is most likely the cause due to the associated difficulties of defining in an appropriate way the spectrum of the non self-adjoint Dirac operator. However, the Dirac operator is just
 a particular case of a Clifford operator, and on manifolds, it has in general non-constant coefficients, so its natural notion of spectrum is the $S$-spectrum. The $S$-spectrum of a self-adjoint Clifford operator (in particular the Dirac operator) is real.
Let us explain with more details the above motivation.
Assume that $g:U\to \mathbb{R}^{n \times n}$, given by
$$
g(x)=(g_{ij}(x))_{i,j=1}^n,
$$
is a smooth matrix-valued function defined on the open set $U$ in $\mathbb{R}^n$
where $g(x)$ will always be taken to be positive-definite and symmetric.
Then
\begin{equation}\label{1.2i}
g_x(\xi,\eta)=\sum_{i,j=1}^ng_{ij}(x)\xi_i\eta_j, \ \ \xi,\ \eta\in \mathbb{R}^n
\end{equation}
is a positive-definite inner product on $\mathbb{R}^n$ and
\begin{equation}\label{1.2ii}
g_x(X,Y)=\sum_{i,j=1}^ng_{ij}(x)a_ib_j,
\ \ {\rm where}\ \
X=\sum_{i=1}^n a_{i}(x)\partial_{x_i},\ \ \ Y=\sum_{j=1}^n b_{j}(x)\partial_{x_j},
\end{equation}
defines a positive-definite inner product space on the tangent space $\mathcal{T}_x(U)$ to $U$ at $x$.
So $U$ can be seen as a coordinate neighbourhood for the Riemannian manifold $M$ taking $x=(x_1,...,x_n)$ as coordinates and  (\ref{1.2ii}) as inner product, making it possible to study the Dirac operator on $M$, by introducing it as a nonconstant coefficients nonhomogeneous first-order systems of differential operators on $\mathcal{C}^\infty(U,\mathcal{H})$.
To give a precise definition of the spectrum of the Dirac operator on $M$, (i..e, the $S$-spectrum), we need to perform the following steps.

(a)
We need to give the definition of the Dirac operator $D$ on $M$.

(b)
Next, we need to write $D^2$ explicitly in terms of a second-order Laplacian and a curvature operator,
via the Bochner-Weitzenb\"ock theorem.

The precise  expression of the operator $D^2$ is of crucial importance in order to define the $S$-spectrum because it is associated with the operator
$$
\mathcal{Q}_s(D):=D^2-2s_0D+|s|^2 \, I,
$$
here $s$ is a paravector in the Clifford algebra $\mathbb{R}_n$.
Before we can define the Dirac operator, we require some additional notions. Let $g:  \RR^{n \times n} \to U$ be an invertible function on $\RR^{n \times n}$ and write
\begin{equation}\label{1.3i}
g^{-1}:U\to \mathbb{R}^{n\times n}, \ \ g^{-1}(x)=(g^{ij}(x))_{i,j=1}^n.
\end{equation}
Next, let
$$
\gamma^{}(x)=(\gamma_{ij}(x))_{i,j=1}^n,\ \ \gamma^{-1}(x)=(\gamma^{ij}(x))_{i,j=1}^n :U\to \mathbb{R}^{n\times n},
$$
be the unique square roots of $g$ and $g^{-1}$, respectively.
{\em
Let $e_1,...,e_n$ be the skew-adjoint operators
satisfying the Clifford relations
$e_je_k+e_ke_j=-2\delta_{jk}$, set
$$
e_i(x)=\sum_{j=1}^n \gamma^{ij}(x)e_j,\ \ \ x\in U,
$$
for every $i=1,...,n.$
}
By definition we have that
\begin{equation}\label{1.6}
e_j(x)e_k(x)+e_k(x)e_j(x)=-2g^{jk}(x),\ \ \ x\in U.
\end{equation}
Let $d\tau :\spin(n)\to \mathcal{U}(\mathcal{H})$ be the
representation of the Lie algebra of $\spin(n)$
derived from the representation $\tau: \spin(n)\to \mathcal{U}(\mathcal{H})$.
For more details see \cite{6DIRAC4}.

{\em
As a differential operator on $\mathcal{C}^\infty(U,\mathcal{H})$ the standard Dirac operator $D$ on $(U,g)$ is defined by
$$
D=\sum_{i=1}^ne_i(x)\Big(\partial_{x_i}+d\tau(\omega_i(x))\Big)
$$
 where $\omega_1,...,\omega_n:U\to \spin(n)$
are smooth functions uniquely determined by $g$.
}

Some more work is necessary to define $D^2$, but clearly the spectral theory on the $S$-spectrum allows one to consider non-self adjoint Dirac operators.
We recall that the scalar curvature of $(U,g)$ is defined by
\begin{equation}
\kappa(x)=-\sum_{i,j=1}^n R_{ijij}(x)
\end{equation}
where $R_{ijk\ell}(x)$ is the Riemann curvature tensor.
As an example we recall Spinor Laplacian, that is we
 assume $\mathcal{H}=\Theta_n$ and $\tau$ is the Spin representation of $S$ of $\spin(n)$ on $\Theta_n$, $n=2m$, (see \cite{6DIRAC4}) the second order operator
$D^2$ is often called spinor Laplacian is given by the Lichnerowicz's formula
$$
D^2=-\Delta_S+\frac{1}{4}\kappa(x).
$$
So in similar cases the $S$-spectrum can be written explicitly.
Several results, such as \cite{Bulla,bookTF,Kawai,KimFRI,Leitner} and also
\cite{songerg,Homma,Kirchberg,Moroianu,Nakad,Nakad2}
in spectral theory for the Dirac operators can now be seen in a new
light using the spectral theory on the $S$-spectrum for the Dirac operator,
also using the $S$-functional calculus \cite{6newresol,6hinfty,jfaStruppa,6JONAME} and the function theory
\cite{ABCS1,CSSf,CSSd,CSSe} on which this calculus is based on for $n$-tuples of operators.

\medskip
\medskip

\noindent (III) {\em Complementarity to the Taylor spectrum for a  class of tuples of commuting operators or non-commuting operators and harmonic analysis}.

\medskip

\noindent Given an $n$-tuple of bounded or unbounded operators $(T_1, …, T_n)$ on a real Hilbert space $\cH$, we can form the right linear operator
$$
T = \sum_{j=1}^n e_{\alpha_j} T_{j} ,
$$
where $\alpha_1, \ldots, \alpha_n \in \powerset(\{ 1, \ldots, n \})$. Now if there exists a configuration of units such that $T$ is normal, then one may define the spectrum of $(T_1, \ldots, T_n)$ to be the $S$-spectrum of $T$, i.e., $\sigma_S(T)$. Moreover, one has a Borel functional calculus at hand for a reasonably large class of functions of $T$, which is helpful in problems in harmonic analysis and partial differential equations. It is worth mentioning that one can find the relations between the monogenic functional calculus, Taylor functional calculus  and the Weyl functional calculus (see,
 \cite{6JM,6JMP,6MQ,6MP,6qian1}).
In harmonic analysis in higher dimensions, singular integrals and in the study of the Fourier transform
one can find various connections with Clifford analysis in the recent book \cite{BOOKTAO}. Boundary value problems treated with quaternionic techniques can be found in \cite{6GURLYSPROSS}. Clifford wavelets, singular integrals, and Hardy spaces are studied in \cite{Mitrea}.

\medskip

\noindent (IV) {\em Spectral theory for linear operators on a Hilbert module over a finitely generated unital ring}.

\medskip
\medskip

We first note that we may use the classification of Clifford modules (see, e.g., \cite{SpinGeometry}) and the fact that every finitely generated unital ring is isomorphic to $\RR^{m \times m}$ to see that there is a natural isomorphism between any Hilbert module over a finitely generated unital ring $R$ and a Clifford module $\cH_n$ for an appropriate choice of $n$.  Thus, given a normal right linear operator $T$ on a Hilbert module over a finitely generated unital ring $R$, we may view $T$ as a being a right linear operator on a Clifford module $\cH_n$ for a suitable choice of $n$. We can then utilise the spectral theory and functional calculi in the Clifford setting and translate back to the Hilbert module setting in a straight forward manner.

\medskip

 For the function theory of slice hyperholomorphic functions the
 main books are \cite{6ACSBOOK,bookSCE,6EntireBook,CSS,6GSBook,GSSBOOK}, while
 for the spectral theory on the $S$-spectrum we mention \cite{6COFBook,6CCKS,6CG,6CKG,CSS}.
 For the Fueter and monogenic function theory and related topics
 see the books \cite{6DIRAC1,6DIRAC2,6DIRAC3,6DIRAC4,6DIRAC5,6jefferies,6DIRAC6}.

\medskip

\noindent {\bf Strategy}

\medskip

We wish to summarise our strategy for proving the spectral theorem for a normal operator on a Clifford module.

\noindent (I)
It turns out that the $S$-spectrum and $S$-resolvent set
  can be defined for all bounded or unbounded operators (not necessarily paravector operators). This is the first crucial intuition for the decomposition of normal operators and
   we  show that the $S$-spectrum of a bounded operator is a non-empty compact subset of $\{ s \in \RR_n: s = s_0 + \sum_{j=1}^n s_j e_j \}$.

\noindent (II)
We define a spectral measure on a Clifford module and spectral integrals $\II(f)$,
where $f$ belongs to a suitable class of functions (see Section 4).

Moreover, we point out that an absolutely key result is Theorem \ref{thm:IMPORTANT} which connects
$\sigma_S(\II(f))$ to $\supp E$. Dealing with the $S$-spectrum requires to overcome substantially
different difficulties with respect to the classical complex or quaternionic Hilbert space case (where the spectral measure $E$ gives rise to a family of positive measures in the usual sense).

\noindent (III) We  prove a spectral theorem for a bounded self-adjoint operator
(see Theorem \ref{thm:BSA}) where the important result given in Theorem \ref{thm:IMPORTANT}
 is being used to show that the spectral measure $E$ has the property that  $\supp E = \sigma_S(T)$.

\noindent (IV) We prove a polar decomposition for a bounded Clifford operator $T$
   and specialise the result to the case when $T$ is normal (see Theorem \ref{thm:POLARDECOMP}).
   This then enables one to prove that {\it every} bounded normal operator can be written as
   $T = A+ BJ$, where $A = A^* \in \cB(\cH_n)$, $B \in \cB(\cH_n)$ is positive, $JJ^* =  I$ and $J^* = -J$ and $A$, $B$ and $J$ all mutually commute (see Theorem \ref{thm:AJB}).

\noindent (V)
To prove the spectral theorem for bounded normal operators
 (see Theorem \ref{thm:BN}), one needs to apply a technical result (see Theorem \ref{thm:COMMUTINGspectral}) to manufacture a uniquely determined spectral measure $E$ which lives on $\sigma_S(A) \times \sigma_S(B)$ which will be identified with the complex plane $\CC_{\gI}$, where $\gI \in \mathbb{S}$.

One applies the key result Theorem \ref{thm:IMPORTANT} with the identity function $f(\lambda) = \lambda$ to see that $\sigma_S(T) \cap \CC_{\gI} = \supp E \cap \CC_{\gI}$. One can get a suitable integral representation for $T$ which resembles the quaternionic case.

\noindent (VI)
We use the bounded case of the spectral theorem (see Theorem \ref{thm:BN}) to prove the unbounded case (see Theorem \ref{thm:UNBOUNDEDNORMAL}).

\medskip

For the convenience of the reader we have compiled a list of commonly used notation that appears throughout this manuscript.

\medskip

\noindent {\bf Notation}

\medskip

\noindent $\RR_n:=\RR_{0,n}$ will denote the Cifford algebra generated by the units $e_0 := 1$ and $e_1, \ldots, e_n$, where $e_i e_j = -e_j e_i$ for $i \neq j$. An element $a \in \RR_n$ can be written as
$$a = \sum_{\alpha } a_{\alpha} e_{\alpha} := \sum_{\alpha \in \powerset(\{ 1, \ldots, n \})} a_{\alpha} e_{\alpha},$$
where $\emptyset$ can be identified with $0$ and $\alpha \neq \emptyset$ can be interpreted as a $n$-tuple $(i_1, \ldots, i_n)$ with $i _1 < \ldots < i_n)$.

\medskip

\noindent $\mathbb{S} := \{ \gI = (0,  \gI_1, \ldots, \gI_n) \in \RR^{n+1}: \sum_{i=1}^n \gI_i = 1 \}$. A typical element of $\mathbb{S}$ will be denoted by $\gI$ or $\gJ$.

\medskip

\noindent For any $\gI \in \mathbb{S}$, we shall let $\CC_{\gI} = \{ \lambda_0 + \lambda_1 \, \gI: \lambda_0, \lambda_1 \in \RR \}$ and $\CC^+_{\gI} = \{ \lambda_0 + \lambda_1 \, \gI: \lambda_0\in \RR, \lambda_1 \geq 0 \}$.

\medskip

\noindent $\gS(\RR_n)$ will denote the set of self-adjoint Clifford numbers in $\RR_n$, i.e., all $a \in \RR_n$ such that $\bar{a} = a$.

\medskip

\noindent $\gP(\RR_n)$ will denote the set of positive semidefinite Clifford numbers in $\RR_n$, i.e., all $a \in \RR_n$ such that $a = b \bar{b}$ for some $b \in \RR_n$. In this case, we will write $a \succeq 0$.

\medskip

\noindent $\chi : \RR_n \to \RR^{2^n \times 2^n}$ will denote the injective $*$-homomorphism from the Clifford algebra $\RR_n$ to the set of real matrices of size $2^n$ (see Definition \ref{def:CHI}).

\medskip

\noindent $\cH_n = \cH \otimes \RR_n$ will denote a two-sided Clifford module over $\RR_n$, where $\cH$ is a real Hilbert space (see Definition \ref{def:CLIFFORD}).

\medskip

\noindent $\langle \cdot, \cdot \rangle: \cH_n \times \cH_n \to \RR_n$ given by \eqref{eq:INNERPRODUCT} and $\| x \| := \sqrt{{\rm Re}\, \langle x, x \rangle}$ for $x \in \cH_n$.

\medskip

\noindent For any closed submodule $Y \subseteq \cH_n$, we shall let
$$Y^{\perp} := \{ y \in \cH_n: \text{$\langle x,y \rangle = 0$ for all $x \in \cH_n$}\}.$$

\medskip

\noindent $\cL(\cH_n)$ will denote the set of right linear operators on $\cH_n$. The domain of an operator $T \in \cL(\cH_n)$ will be denoted by $\cD(T)$ (see Definition \ref{def:LINEAROPERATORS}).

\medskip

\noindent $\Ran T$ and $\Ker T$ will denote the range and kernel of $T \in\cL(\cH_n)$, respectively (see Definition \ref{def:KERNELANDRANGE}).

\medskip

\noindent $\cB(\cH_n)$ will denote the set of bounded linear operators on $\cH_n$ (see Definition \ref{def:BDDOP}. We shall let
$$\| T\|:=\sup_{\| x \|\leq 1} \|Tx\| = \lim_{\| x \| = 1 } \|T x\|.$$

\medskip

\noindent $\cG(T)$ will denote the graph of a linear operator (see Definition \ref{def:GRAPH}).

\medskip

\noindent For $T \in \cL(\cH_n)$, we shall let $\langle x,y \rangle_T := \langle x , y \rangle + \langle Tx, Ty \rangle$ and $\| x \|_T :=(\|x\| + \|T x \|^2)^{1/2}$ for $x \in \cD(T)$.

\medskip

\noindent For $S,T \in \cL(\cH_n)$, we will write  $S \subseteq T$ if $\cD(S) \subseteq \cD(T)$ and $Sx = Tx$ for $x \in \cD(S)$.

\medskip
\noindent $|T|:=(T^*T)^{1/2}$ (see  Theorem \ref{thm:POLARDECOMP}).
\medskip

\noindent For any closable operator $T \in \cL(\cH_n)$, we shall let $\overline{T}$ denote the closure of $T$ (see Definition \ref{def:Jun7er2}).

\medskip

\noindent We shall let $\cQ_s(T) := T^2 - 2{\rm Re}(s)T + |s|^2 I$.

\medskip

\noindent
 $\rho_S(T)$ denotes the $S$-resolvent set of $T$ and $\sigma_S(T)$ denotes the $S$-spectrum of $T$ (see Definition \ref{def:Dec14j1}).

\medskip

\noindent We shall let $S_L^{-1}(s,T)$ and $S_R^{-1}(s,T)$ denote the left $S$-resolvent operator and right $S$-resolvent operator of $T$ at $s$, respectively (see Definition \ref{def:Sresolvent}).

\medskip

\noindent For $T \in \cB(\cH_n)$, we shall let $r_S(T)$ denote the spectral radius of $T$ with respect to the $S$-spectrum (see \eqref{eq:SR}).

\medskip

\noindent For $T \in \cB(\cH_n)$, we shall let $\sigma_{PS}(T)$, $\sigma_{RS}(T)$ and $\sigma_{CS}(T)$ denote the point $S$-spectrum of $T$, residual $S$-spectrum of $T$ and continuous $S$-spectrum of $T$, respectively (see Definition \ref{def:SPECTRUMS}).

\medskip

\noindent For $T \in \cB(\cH_n)$, we shall let $\Pi_S(T)$ and $\Gamma_S(T)$ denote the approximate point $S$-spectrum of $T$ and compression $S$-spectrum of $T$, respectively (see Definition \ref{def:MORESPECTRUM}).

\medskip

\noindent $T \in \cL(\cH_n)$ will be called normal if $T$ is densely defined, closed, $\cD(T) = \cD(T^*)$ and $TT^* = T^*T$.

\medskip

\noindent $T \in \cL(\cH_n)$ will be called self-adjoint if $\cD(T) = \cD(T^*)$ and $Tx =T^*x$ for all $x \in \cD(T)$.

\medskip

\noindent $T \in \cL(\cH_n)$ will be called positive if $T$ is $\cD(T) = \cD(T^*)$ and $\langle Tx, x \rangle \succeq 0$ for all $x \in \cD(T)$, i.e., $\langle Tx, x \rangle \in \gP(\RR_n)$ for all $x \in \cD(T)$.

\medskip

\noindent $T \in \cL(\cH_n)$ will called anti self-adjoint if $\cD(T) = \cD(T^*)$ and $T = -T^*$ for $x \in \cD(T)$.

\medskip

\noindent $T \in \cB(\cH_n)$ will be called unitary if $TT^* = I$.

\medskip

\noindent $J \in \cB(\cH_n)$ will be called imaginary if $J$ is anti self-adjoint and $J$ is unitary.

\medskip

\noindent $\cP(\cH_n)$ will denote the set of orthogonal projections on $\cH_n$.

\medskip

\noindent $s-\lim_{i \to \infty} T_i$ will denote the limit of a sequence of operators $(T_i)_{i=1}^{\infty}$, where $T_i \in \cB(\cH_n)$ for $i=1,2,\ldots$, in the strong operator topology.

\medskip

\noindent $E$ will denote a spectral measure on $(\Omega, \sA)$, where $\sA$ is a $\sigma$-algebra of sets generated by $\Omega$.

\medskip

\noindent The set of Borel sets generated by $\Omega$ will be denoted by $\sB(\Omega)$.

\medskip

\noindent Given a spectral measure $E$ on $(\Omega, \sB(\Omega))$, $\supp E$ will denote the support of $E$ (see Definition \ref{def:SUPPORT}).

\medskip

\noindent Given an $\RR_n$-valued measure $\nu$ the total variation of $\nu$ will be denoted by $|\nu|$ (see \eqref{eq:TOTALVAR}).

\medskip

\noindent Given an imaginary operator $J \in \cB(\cH_n)$, we will say that $J$ is associated with a spectral measure $E$ on $(\Omega, \sA)$ if $E(M)J = J E(M)$ for all $M \in \sA$ (see Definition \ref{def:ASSOCIATED}).

\medskip

\noindent $\gB(\Omega,\sA, \CC_{\gI})$ will denote the Banach space of all bounded $\sA$-measurable functions $f: \Omega \to \CC_{\gI}$ equipped with the norm
$$\| f \|_{\infty} = \sup_{\lambda \in \Omega} |f(\lambda)|.$$

\medskip

\noindent $\gB_s(\Omega,\sA, \CC_{\gI})$ will denote the subspace of $\gB(\Omega,\sA,\CC_{\gI})$ of all simple functions.

\medskip

\noindent Given $f \in \gB(\Omega, \sA,\CC_{\gI})$ and a spectral measure $E$ on $(\Omega, \sA)$, we shall let $\II(f)$ denote the spectral measure of $f$ with respect to $E$ (see Definition \ref{def:BFUNC}).

\medskip

\noindent Given a spectral measure $E$ on $(\Omega, \sA)$, we shall let $\gF(\Omega,\sA, \CC_{\gI}, E)$ denote the set of all $\sA$-measurable functions $f: \Omega \to \CC_{\gI}\cup \{ \infty \}$ which are $E$-a.e. finite, i.e., $E(\{ \lambda \in \Omega: f(\lambda) = \infty \}) = 0$.

\medskip

\noindent Given $f \in \gF(\Omega,\sA, \CC_{\gI}, E)$ and a spectral measure $E$ on $(\Omega, \sA)$, we shall let $\II(f)$ denote the spectral measure of $f$ with respect to $E$ (see Theorem \ref{thm:BIG}(ii)).
\medskip

\noindent Given an imaginary operator $J \in \cB(\cH_n)$ and $\gI \in \mathbb{S}$,  $\cH_{\pm}(J, \gI) = \{ x \in \cH_n: Jx = x (\pm \gI) \}$ is a complex subspace of $\cH_n$ (see \eqref{eq:CI}).

\medskip

\noindent Given a densely defined operator $T \in \cL(\cH_n)$, we shall let $C_T := (I+T^*T)^{-1}$ and $Z_T := T C_T^{1/2}$ (see Definition \ref{def:Ztransform}).

\medskip

\section{Preliminaries}
\label{sec:PRELIM}
In this section we will formulate a number of definitions and results on Clifford algebras, Clifford modules, linear operators on Clifford modules, the $S$-spectrum and measure theory and integration theory with respect to a Clifford algebra-valued measure. While the topic of Clifford algebras is very classical and well-known, the theory of linear operators on Clifford modules and measure theory and integration theory with respect to a Clifford algebra-valued measure are not so well developed. We have furnished proofs for Clifford algebra/module analogues of results whenever the proof differs from the classical case with the aim of making the present manuscript as self contained as possible.

\subsection{Clifford algebras}
\label{sec:CAs}

\begin{defn}
A collection of $n$ elements $e_1, \ldots, e_n$, with $n = p + q$ and $p, q \in \mathbb{N}_0$ will be called {\it imaginary units} if
\begin{align}
e_i^2 =& \; +1 \quad\quad {\rm for} \quad i = 1, \ldots, p \nonumber \\
e_j^2 =& \; -1 \quad\quad {\rm for} \quad j=p+1, \ldots, n \nonumber \\
e_j e_k + e_k e_j =& \; 0 \quad\quad {\rm for} \quad j \neq k \nonumber \\
\tag*{and}  \nonumber \\
e_1 \ldots e_n \neq& \; \pm 1. \label{eq:PSEUDO}
\end{align}
We shall denote the real algebra generated by the imaginary units $e_1, \ldots, e_n$ the universal Clifford algebra and denote it by $\RR_{p,q}$. An element of $\RR_{p,q}$ is called a {\it Clifford number}.
\end{defn}

\begin{rem}
\label{rem:DIMENSION}
We note that it is only necessary to assume \eqref{eq:PSEUDO} if $p -q \equiv 1 ({\rm mod} \; 4)$. It is easy to check that $\{ e_{\alpha} \}_{\alpha \in \powerset(\{1, \ldots, n \})}$ is linearly independent and hence $\RR_{p,q}$ has dimension $2^n$ (as a real vector space).
\end{rem}

Let $\alpha \in \powerset(\{ 1, \ldots, n \})$ and if $\alpha \neq \emptyset$, then we may write $\alpha = \{ i_1, \ldots, i_r \} $, with $i_1 < \ldots < i_r$. Then we may let
\begin{align*}
e_\alpha :=& \; e_{i_1, \ldots, i_r} := e_{i_1} \ldots e_{i_r} \quad \quad {\rm if} \quad \alpha = \{ i_1, \ldots, i_r\}, \\
e_0 :=& \; e_{\emptyset} = 1 \quad \quad {\rm if} \quad \alpha = \emptyset
\end{align*}
and write an arbitrary Clifford number $a \in \RR_{p,q}$ as
$$a = \sum_{\alpha} a_\alpha e_\alpha := \sum_{\alpha \, \in \, \powerset( \{ 1, \ldots, n \} ) } a_{\alpha} e_{\alpha},$$
where the sum is taken over $\emptyset$ and subsets $\{ i_1, \ldots, i_r \}$ with $i_1 < \ldots < i_r$.

\begin{rem}
The only Clifford algebra considered in the remainder of this paper will be  $\RR_n := \RR_{0,n}$.
\end{rem}

We will let
$$\bar{a} := \sum_{\alpha} a_{\alpha} \bar{e}_{\alpha} = \sum_{\alpha} a_{\alpha} (- e_{\alpha} )$$
and
$$|a| := \left( \sum_{\alpha} a_{\alpha}^2 \right)^{1/2}  \quad \quad {\rm for} \quad a = \sum_{\alpha} a_{\alpha} e_{\alpha} \in \RR_{n}.$$
Hence one can easily check that for all $a,b \in \RR_n$, we have
\begin{align}
\overline{ab} =& \; \bar{b} \, \bar{a}  \nonumber\\
\overline{a+b} =& \; \bar{a} +\bar{b}  \nonumber \\
|a+b | \leq& \; |a| + |b| \label{eq:TRI} \\
|ab| =& \; |a|\, |b| \label{eq:MM} \quad\quad {\rm whenever} \quad b\bar{b} = |b|^2 \\
\tag*{and} \\
|ab| \leq& \; 2^{n-1} |a| \, |b|. \label{eq:SMM}
\end{align}

\begin{defn}
\label{def:unitsphere}
Given $s = (s_0, s_1, \ldots, s_n) \in \RR^{n+1}$, we may identify the vector $s \in \RR^{n+1}$ with the {\it paravector} $\mathfrak{s} = \sum_{i=0}^n s_i e_i \in \RR_n$. With a slight abuse of notation, for the remainder of this paper, we will use $s$ in place of $\mathfrak{s}$.

Let $\mathbb{S}$ denote the unit sphere of vectors in $\RR^{n+1}$, i.e.,
$$\mathbb{S} := \{ \gI = (0,\gI_1,\ldots, \gI_n) \in \RR^{n+1}:  \sum_{i=1}^n \mathfrak{I}^2_i = 1 \}.$$
It is easy to see that $\mathbb{S}$ is an $(n-1)$-sphere in $\RR^n$ and $\gI \in \mathbb{S}$ implies that $\gI^2 = -1$. Note that the real two-dimensional subspace of $\RR^{n+1}$ generated by $1$ and $\gI$ is complex plane $\CC_{\gI} := \RR + \RR \, \gI$. It is not hard to see that $\CC_{\gI}$ is isomorphic to the usual complex plane $\CC$.

Given a paravector $s = \sum_{i=0}^n s_i e_i \in \RR^{n+1}$, it is possible to find $\gI_s \in \mathbb{S}$ such that $s \in \CC_{\gI_s}$. Indeed, if $s \neq 0$, then we may write
$$s = s_0 + \left( \frac{ \sum_{i=1}^n s_i e_i }{ |s| } \right) |s|.$$
Thus, if we let $\gI_s := \frac{ \sum_{i=1}^n s_i e_i }{ |s| }$, then one can check that $\gI_s \in \mathbb{S}$ and hence $s \in \CC_{\gI_s}$.
\end{defn}

\begin{defn}[self-adjoint Clifford numbers]
\label{def:SACLIFFORD}
We will call $a \in \RR_n$ {\it self-adjoint} if $a = \bar{a}$. The set of all self-adjoint Clifford numbers in $\RR_n$ shall be denoted by $\gS(\RR_n)$.  We note that $\gS(\RR_n)$ is a real vector space.
\end{defn}

\begin{defn}[positive semidefinite Clifford numbers]
\label{def:CLIFFORD}
We will call $a \in \RR_n$ {\it positive semidefinite} if there exists $b\in \RR_n$ such that $a = b \bar{b}$. In this case, we shall write $a \succeq 0$. The set of all positive semidefinite Clifford numbers in $\RR_n$ shall be denoted by $\gP(\RR_n)$.
\end{defn}

We will need to recall a well-known injective $*$-homomorphism $\chi: \RR_n \to \RR^{2^n \times 2^n}$ which can be found, e.g., in \cite{Mitrea}.

\begin{defn}
\label{def:CHI}
Let $\chi: \RR_n \to \RR^{2^n \times 2^n}$ be the injective $*$-homomorphism given by the following inductive construction. We will first give meaning to $\chi(e_j)$ for $j=0,1,\ldots, n$. Let $\chi(e_0) = I_{2^n}$ and $\chi(e_j) := E_j^n$ for $j=1,\ldots, n$, where $\{ E_j^k \}_{j=1}^k$ are inductively defined via
$$E_1^1 := \begin{pmatrix}
0 & -1 \\ 1 & 0
\end{pmatrix}, \quad E_j^{k+1} := \begin{pmatrix}
E_j^k & 0 \\ 0 & -E_j^k
\end{pmatrix} \quad {\rm and} \quad E_{k+1}^{k+1} = \begin{pmatrix}
0 & -I_{2^k} \\
I_{2^k} & 0
\end{pmatrix}$$
for $j=1,\ldots, k$ and $k=1,\ldots, n-1$. Next, we let
$$\chi(e_{\alpha}) := \chi(e_{i_1}) \cdots \chi(e_{i_k}) \quad {\rm for} \quad \alpha = \{ i_1, \ldots, i_k \},$$
where $i_1 < \ldots < i_k$. Finally for $a = \sum_{\alpha} a_{\alpha} e_{\alpha} \in \RR_n$, we let
$$\chi(a) := \sum_{\alpha} a_{\alpha} \chi(e_{\alpha}).$$
\end{defn}

\begin{rem}
\label{rem:chiProp}
Let $a\in \RR_n$. It is easy to check that $\chi(a)$ is a positive semidefinite matrix in $\RR^{2^n \times 2^n}$ if and only if $a$ is a positive semidefinite Clifford number.
\end{rem}

\subsection{Clifford modules}
\label{sec:CMs}

\begin{defn}[Clifford modules over $\RR_n$]
\label{def:CM}
Let $\cH$ be a real Hilbert space with an inner product $\langle \cdot, \cdot \rangle_{\cH}$ and a natural norm $\| x \|_{\cH} := \langle x, x \rangle_{\cH}^{1/2}$ for $x \in \cH$. Then by $\cH_n$ we mean the {\it two-sided Clifford module generated by $\cH$ and $\RR_n$} with $n  >0$, i.e., $\cH_n$ consists of all vectors of the form
\begin{equation}
\label{eq:CM}
x  = \sum_\alpha x_\alpha \otimes e_\alpha,
\end{equation}
with
\begin{align}
x+y :=& \;  \sum_{\alpha} (x_\alpha+y_\alpha) \otimes e_\alpha, \label{eq:ADDITION} \\
xa :=& \; \sum_{\alpha}  x_\alpha  \otimes (e_\alpha a) = \sum_{\alpha, \beta} (x_{\alpha}a_{\beta}) \otimes (e_{\alpha} e_{\beta}) \label{eq:RIGHTMULT} \\
\tag*{and} \nonumber \\
ax :=& \; \sum_{\alpha}  x_\alpha  \otimes (a e_\alpha ) = \sum_{\alpha, \beta} (x_{\alpha}a_{\beta}) \otimes (e_{\beta} e_{\alpha}) \label{eq:LEFTMULT}
\end{align}
for $y = \sum_{\alpha} y_{\alpha}\otimes e_{\alpha}$ and $a = \sum_{\alpha} a_{\alpha} e_{\alpha}$. We will employ the short-hand $x = \sum_{\alpha} x_{\alpha}e_{\alpha}$ in place of $x =\sum_{\alpha} x_{\alpha} \otimes e_{\alpha}$. We may then endow a Clifford module $\cH_n$ with the $\RR_n$-valued ``{\it inner product}'' $\langle \cdot, \cdot \rangle$ given by
\begin{equation}
\label{eq:INNERPRODUCT}
\langle x, y \rangle := \sum_{\alpha, \beta} e_{\alpha} \bar{e}_{\beta} \langle x_\alpha, y_{\beta} \rangle_{\cH} \quad \quad {\rm for} \quad x,y \in \cH_n.
\end{equation}
\end{defn}

\begin{rem}
\label{rem:TWOSIDED}
For the great preponderance of this work, we will not need the fact that $\cH_n$ is a two-sided module. Thus, whenever possible we will make use of the fact that $\cH_n$ is a right module. However, for Theorem \ref{thm:SRF} we will need to use the two-sided which will be utilised when proving the spectral theorem for a bounded self-adjoint operator in Section \ref{sec:STBSA}.
\end{rem}

\begin{rem}
\label{rem:IP}
It is easy to check that the following facts hold:
\begin{align}
\langle x, x \rangle \succeq& \;  0 \quad {\rm and} \quad
\langle x, x \rangle = 0 \Longleftrightarrow x = 0 \label{eq:IP1} \\
\langle x+y, z \rangle =& \; \langle x, y \rangle + \langle y, z \rangle,  \label{eq:IP2} \\
\overline{\langle x, y \rangle} =& \; \langle y, x \rangle \label{eq:IP3} \\
\tag*{and}  \nonumber \\
\langle x a, y  \rangle =& \; \langle x, y \rangle  a \quad\quad {\rm for} \quad x,y,z \in \cH_n \quad {\rm and} \quad a \in \RR_n. \label{eq:IP4}
\end{align}
Furthermore, if we define
\begin{equation}
\label{eq:NORM}
\| x \|:= \left(\sum_\alpha \| x_\alpha \|^2_{\cH} \right)^{1/2} \quad \quad {\rm for}\quad x = \sum_\alpha x_\alpha e_\alpha,
\end{equation}
then $\| \cdot \|$ is a ``norm'' on $\cH_n$, i.e., for all $x, y \in \cH_n$ and $a \in \RR_n$, we have
\begin{align}
\| x \| >& \; 0 \quad\quad \text{ whenever $x$ is nonzero,}    \label{eq:NORM1} \\
\| x a \| =& \; \| a x \| = | a | \cdot \| x \| \quad {\rm for} \quad a \in \RR^{n+1}, \label{eq:NORM2}\\
\| x a\| = \| a x \|  \leq& \; 2^{n-1} | a | \| x \| \quad {\rm for} \quad a \in \RR_n \label{eq:NORM3}\\
\tag*{and} \nonumber \\
\| x + y \| \leq& \; \| x \| + \| y \|. \label{eq:NORM4}
\end{align}
\end{rem}

\begin{rem}
\label{rem:COMPLETE}
In view of \eqref{eq:NORM1}, \eqref{eq:NORM2} and \eqref{eq:NORM4}, the fact that $\cH$ is complete and the definition of a Clifford module (see Definition \ref{def:CLIFFORD}), it is very easy to see that $\cH_n$ together with $\| \cdot \|$ can be viewed as a complex Banach space over $\CC_{\gI}$ for any choice of $\gI \in \mathbb{S}$.
\end{rem}

\begin{defn}[orthonormal basis in $\cH_n$]
We will call $( \xi_i )_{i \in \cI} \subseteq \cH_n$ {\it linearly independent} if for any finite subset $\wt{\cI}\subseteq \cI$, we have that the equation
$$\sum_{i \in \wt{\cI}} \xi_i a_i = 0,$$
where $( a_i )_{i \in \wt{\cI}} \subseteq \RR_n$, only has the trivial solution, i.e., $a_i = 0$ for all $i \in \wt{\cI}$. Next, we will say that $\cB := ( \xi_i )_{i \in \cI} \subseteq \cH_n$ is an {\it orthonormal basis} for $\cH_n$ if $\cB$ is linearly independent, the closed right linear span of $\cB$ is $\cH_n$, i.e.,  every $x \in \cH_n$ belongs to the closure (with respect to the norm topology on $\cH_n$) of the set
$$\{ \sum_{i \in \widetilde{\cI}} \xi_i a_i: \text{$a_i \in \RR_n$ for all $i \in \widetilde{\cI}$ with $\widetilde{\cI}$ finite} \}$$
(in this case, we will utilise the short-hand $x = \sum_{i \in \cI} \xi_i a_i$)
and
$$\langle \xi_i, \xi_j \rangle = \begin{cases} 1 & {\rm if} \quad i = j, \\0 & {\rm if} \quad i \neq j. \end{cases}$$
\end{defn}

While, it is not true that
$$\| x \| = \langle x, x \rangle^{1/2} \quad \quad {\rm for} \quad x \in \cH_n,$$
we do have the following facts which are reminiscent of the classical Hilbert space setting.

\begin{lem}
\label{lem:KAPPA}
Let $\cH_n$ be a Clifford module over $\RR_n$. Then the following facts hold{\rm :}
\begin{enumerate}
\item[ {\rm (i)} ]$\| x \|^2 = {\rm Re}\, \langle x, x \rangle$ for $x \in \cH_n$.
\smallskip
\item[{\rm (ii)}] $| \langle x, y \rangle  | \leq  \| x \| \, \| y \|$ for $x, y \in \cH_n$.
\smallskip
\item[{\rm (iii)}] Every Clifford module $\cH_n$ has an orthonormal basis. Moreover, if $\cB := ( \xi_i )_{i \in \cI}$ is a basis for $\cH$, then $\cB$ is an orthonormal basis for $\cH_n$.
\smallskip
\item[{\rm (iv)}] For any orthonormal basis $( \xi_i)_{i \in \cI}$, we have
\begin{equation}
\label{eq:PARSEVAL}
x = \sum_{i \in \cI} \xi_i \langle x, \xi_i \rangle.
\end{equation}
\end{enumerate}
\end{lem}

\begin{proof}
See, e.g., Proposition 1.9 in \cite{Mitrea} for (i) and (ii). We will now prove (iii). Let $\cB := (\xi_i)_{i \in \cI}$ be an orthonormal basis for the real Hilbert space $\cH$. Then, in view of \eqref{eq:CM}, $\cB$ is a basis for $\cH_n$. In view of \eqref{eq:INNERPRODUCT}, we have
$$\langle \xi_i, \xi_j \rangle = \langle \xi_i, \xi_j \rangle_{\cH} = \begin{cases} 1 & {\rm if} \quad i = j \\
0 & {\rm if} \quad i \neq j.\end{cases}$$
Thus, $\cB$ is an orthonormal basis of $\cH_n$.

Let $\cB$ be as above. Since $\cB$ is a basis for the real Hilbert basis $\cH$, we may use the classical orthonormal expansion for a real Hilbert space to obtain
\begin{align*}
x =& \; \sum_{\alpha} x_{\alpha} e_{\alpha} \\
=& \; \sum_{\alpha} \left( \sum_{i \in \cI} \langle x_{\alpha}, \xi_i \rangle \right) e_{\alpha} \\
=& \; \sum_{i \in \cI} \sum_{\alpha} \xi_i \langle x_{\alpha} , \xi_i \rangle e_{\alpha} \\
=& \; \sum_{i \in \cI} \sum_{\alpha} \xi_i \langle x_{\alpha}e_{\alpha}, \xi_i \rangle \\
=& \; \sum_{i \in \cI} \xi_i \langle x, \xi_i \rangle.
\end{align*}
\end{proof}

In the following lemma, by $\dim \gS(\RR_n)$, we mean the dimension of the real vector space of self-adjoint Clifford numbers in $\RR_n$ (see Definition \ref{def:SACLIFFORD}). For example, $\dim \gS(\RR_0) = \dim \gS(\RR_1) = \dim \gS(\RR_2) =1$ and $\dim \gS(\RR_3) = 4$.

\begin{lem}[\cite{GR}, Polarisation formula for a Clifford module]
Let $\cH_n$ be a Clifford module over $\RR_n$. Then
\begin{equation}
\label{eq:POLAR}
4 \dim \mathfrak{S}(\RR_n) \langle x, y \rangle = \sum_\alpha ( \langle x + y e_\alpha, x+ y e_\alpha \rangle - \langle x - y e_\alpha, x- y e_\alpha \rangle )e_\alpha
\end{equation}
for all $x, y \in \cH_n$.
\end{lem}

\begin{lem}[Parallelogram law]
Let $\cH_n$ be a Clifford module. Then for any $x, y \in \cH_n$, we have
\begin{equation}
\label{eq:PAR}
\| x\|^2 + \| y \|^2 = 2 \| x \|^2 + 2 \| y\|^2,
\end{equation}
where $\| \cdot \|$ is defined by \eqref{eq:NORM}.
\end{lem}

\begin{proof}
Since
\begin{align*}
\| x +y \|^2 + \| x + y \|^2 =& \; {\rm Re}( \langle x, x \rangle + \langle x, y \rangle + \langle y, x \rangle + \langle y, y \rangle) \\
+& \; {\rm Re}( \langle x, x \rangle - \langle x, y \rangle - \langle y, x \rangle + \langle y, y \rangle) \\
=& \; 2 {\rm Re} \, \langle x, x \rangle + 2 {\rm Re}\, \langle y, y \rangle \\
=& \; 2 \| x \|^2 + 2 \| y \|^2,
\end{align*}
we have \eqref{eq:PAR}.
\end{proof}

\begin{defn}
Let $K \subseteq \cH_n$. We will call $K$ {\it convex}  if for any choice of $x,y \in K$, we have
$$c \, x + (1-c) \, y \in K \quad \quad {\rm for} \quad 0 \leq c \leq 1.$$
\end{defn}

The next lemma provides a natural generalisation of the closed point in a closed  convex subset property that holds for any complex Hilbert space.

\begin{lem}
\label{lem:CPP}
Let $K$ be a non-empty, closed and convex subset of a Clifford module $\cH_n$.  For any $x \in K$, there exists a unique vector $y \in K$ such that $\| x - y\|$ is as small as possible.
\end{lem}

\begin{proof}
We begin by letting
$$m := \inf_{z \in K} \| x - z \|.$$
Let $(y_i)_{i=1}^{\infty}$ be any minimising sequence for $x$, i.e.,
\begin{equation}
\label{eq:MINI}
\lim_{i \to \infty} m_i = m \quad \quad {\rm for} \quad m_i := \| x - y_i \|.
\end{equation}
Next, if we apply \eqref{eq:PAR} to the vectors $x = (x-y_i)/2$ and $y = (x - y_j)/2$, we obtain
\begin{equation}
\label{eq:PARRESULT}
\left \|   x - \frac{y_i + y_j}{2}     \right\|^2 + \left \| \frac{y_i - y_j}{2} \right\|^2 = \frac{ d_i^2 + d_j^2 }{2}.
\end{equation}
Since $K$ is a convex set, we must have that $(y_i+y_j)/2 \in K$. Consequently, we have $\| x - (y_i + y_j)/2 \| \geq m$. Thus, \eqref{eq:MINI} and \eqref{eq:PARRESULT} imply that $(y_i)_{i=0}^{\infty}$ is a Cauchy sequence in $\cH_n$. Thus, as $K$ is closed, we have that
$$y := \lim_{i \to \infty} y_i \in K.$$
Notice that $y$ has the property that $\| x - y \|$ is as small as possible, since
$$\| x- y\| = \lim_{i \to \infty} \| x - y_i \| = m.$$

Finally, suppose $y' \in K$ such that $\| x - y'\| = m$. Then simply use \eqref{eq:PAR} with $x-y$ and $x- y'$ to deduce that $y = y'$.
\end{proof}

\begin{thm}
\label{thm:ORTHOCOMPS}
Let $Y$ be a closed right submodule of a Clifford module $\cH_n$ and
$$Y^{\perp} := \{ x \in \cH_n:\text{$\langle x, y \rangle =0$ for all $y \in Y$}\}.$$
Then the following statement hold{\rm :}
\begin{enumerate}
\item[(i)] $Y^{\perp}$ is a closed right submodule of $\cH_n$.
\smallskip
\item[(ii)] $\cH_n = Y \oplus Y^{\perp}$.
\smallskip
\item[(iii)] $(Y^{\perp})^{\perp}$.
\end{enumerate}
\end{thm}

\begin{proof}
We will first show (i). Fix any $y \in Y$ and suppose $x, z \in Y^{\perp}$ and $a \in \RR_n$. Then
\begin{align*}
\langle x a + z , y \rangle =& \; \langle x, y \rangle a + \langle z, y \rangle \\
=& \; 0 + 0 = 0.
\end{align*}
Thus, $Y^{\perp}$ is a right submodule of $\cH_n$. Next, suppose $(x_i)_{i=0}^{\infty}$ is a convergent sequence, where $x_i \in Y^{\perp}$ for $i=0,1,\ldots$ and let
$$x := \lim_{i \to \infty} x_i.$$
Then
$$\langle x, y \rangle = \langle x - x_i, y \rangle + \langle x_i, y \rangle = \langle x - x_i, y \rangle$$
and hence for any $\epsilon > 0$, we have
\begin{align*}
| \langle x, y \rangle| =& \; | \langle x - x_i, y \rangle | \leq \| x - x_i \| \cdot \| y \| \\
< & \; \varepsilon  \quad \quad {\rm for} \quad \text{$i$ sufficiently large.}
\end{align*}
But then we have $|\langle x, y \rangle| = 0$, in which case $x \in Y^{\perp}$. Thus, (i) holds.

We will now prove (ii). Suppose $x \in \cH_n$ is arbitrary. Then by Lemma \ref{lem:CPP},  we can find a unique vector $y \in Y$ such that $\| x - y \|$ is as small as possible. Consequently, if we put $z := x - y$, then for any $v \in Y$ and $t \in \RR$, we have
$$\| z \|^2 \leq \| z + t \, v \|^2.$$
Thus,
\begin{align*}
\| z \|^2 \leq& \; {\rm Re} \, \langle z + t \, v, z + t\, v \rangle \\
=& \; {\rm Re}\, \langle z, z \rangle + t ( {\rm Re} \langle z, v \rangle +{\rm Re}\, \langle v, z \rangle ) + t^2 {\rm Re}\, \langle v, v \rangle \\
=& \; \| z\|^2 + 2 t \,  {\rm Re} \langle z, v \rangle + t^2 \| v \|^2,
\end{align*}
in which case we have
\begin{equation}
\label{eq:SUBMOD}
{\rm Re}\, \langle z, v \rangle =0 \quad \quad {\rm for} \quad v \in Y.
\end{equation}
 Since $Y$ is a submodule, we can replace $y$ by $y e_{\alpha}$ in \eqref{eq:SUBMOD} and realise that $\langle z, v \rangle =0$ for all $y \in Y$. Thus, $z \in Y^{\perp}$ and we have the decomposition $\cH_n = Y \oplus Y^{\perp}$.

 To prove that the decomposition $\cH_n = Y \oplus Y^{\perp}$ is unique, suppose $x \in \cH_n$ can be written as $x = y + z$ and $x = y' + z'$, where $y, y' \in Y$ and $z, z' \in Y^{\perp}$. But then $y - y' = z - z'$ will simultaneously belong to $Y$ and $Y^{\perp}$. Consequently, $\langle y - y', y - y' \rangle = 0$
 forces $y = y'$ (see Definition \ref{def:CM}) and hence $z = z'$.

 Finally, we note that (iii) is a direct consequence of (ii).
\end{proof}

\section{Linear operators on Clifford modules}

\begin{defn}[linear operator on a Clifford module]
\label{def:LINEAROPERATORS}
Let $\cL(\cH_n)$ denote the set of linear operators $\wt{T}: \cD(\wt{T}) \to \cH$, where $\cD(T) \subseteq \cH_n$. The subspace $\cD(\wt{T}) \subseteq \cH_n$ will be called the {\it domain} of $\widetilde{T} \in \cL(\cH_n)$. We shall let $\cL(\cH_n)$ denote the set of all operators $T: \cD(T) \to \cH_n$ of the form
$T = \sum_{\alpha}  e_{\alpha}T_{\alpha}$, where $T_{\alpha} \in \cL(\cH_n)$ which acts on
$\cD(T) = \bigcap_{\alpha} \cD(T_{\alpha}) \subseteq \cH_n$ via
\begin{equation}
\label{eq:ACTION}
Tx = \sum_{\alpha,\beta} T_{\alpha} (x_{\beta}) e_{\alpha} e_{\beta} \quad \quad {\rm for} \quad x = \sum_{\alpha} x_{\alpha} e_{\alpha} \in \cD(T).
\end{equation}
\end{defn}

\begin{rem}
\label{rem:RIGHTLINEAR}
Let $T \in \cL(\cH_n)$. One can use \eqref{eq:ACTION} to check that for all $x, y \in \cD(T)$ and $a \in \RR_n$, we have
\begin{align}
T(x a + y) =& \; (Tx ) a + Ty \label{eq:RL1}
\end{align}
Thus, in view of \eqref{eq:RL1}, $\cL(\cH_n)$ consists of right linear operators.
\end{rem}

\begin{defn}[kernel and range]
\label{def:KERNELANDRANGE}
Given $T \in \cL(\cH_n)$, with domain of $T$ denoted by $\cD(T)$.
Then the range and kernel of $T$ will be given by
$$\Ran \, T = \{\text{$y \in \cH_n${\rm :} $Tx = y$ for $x \in \cD(T)$}\}$$
and
$$\Ker \, T = \{\text{$x \in \cD(T)${\rm :} $Tx = 0$}\},$$
respectively.
\end{defn}

\begin{defn}[bounded linear operator on a Clifford module]
\label{def:BDDOP}
We will call an operator $T \in \cL(\cH_n)$ {\it bounded} if
\begin{equation}
\label{eq:K1}
\cD(T) = \cH_n
\end{equation}
and
\begin{equation}
\label{eq:K2}
\| T \| := \sup_{\| x \| \leq 1} \| Tx \| < \infty \quad \quad {\rm for} \quad x \in \cH_n.
\end{equation}
We will let $\cB(\cH_n)$ denote the set of all operators $T \in \cL(\cH_n)$ such that \eqref{eq:K1} and \eqref{eq:K2} hold.
\end{defn}

\begin{rem}
\label{rem:COMPLETEOPERATORNORM}
One can use \eqref{eq:NORM1}, \eqref{eq:NORM2} and \eqref{eq:NORM4} and the definition of the operator norm on $\cB(\cH_n)$ (see \eqref{eq:K2}) to see that $\delta(T,W) = \| T - W \|$ is a metric on $\cB(\cH_n)$. One can proceed as in the classical case when $\cH_n$ is a complex Hilbert space to see that $\cB(\cH_n)$ is complete with respect to the metric $\delta$.
\end{rem}

\begin{defn}[graph of a linear operator]
\label{def:GRAPH}
Suppose $T \in \cL(\cH_n)$. The {\it graph of $T$} is the set
$$\cG(T) := \{ (x, Tx): x \in \cD(T) \}.$$
\end{defn}

\begin{lem}
\label{lem:Jun7tm2}
Let $\cH_n$ be a Clifford module over $\RR_n$. A right submodule $\mathcal{K}$ of $\cH_n \oplus \cH_n$ satisfies
\begin{equation}
\label{eq:Jun7rk1}
\mathcal{K} = \{ (x, Tx): x \in \cD(T) \},
\end{equation}
for some $T \in \cL(\cH_n)$ if and only if
\begin{equation}
\label{eq:Jun7rk2}
(0, y) \in \mathcal{K} \Longrightarrow y = 0.
\end{equation}
\end{lem}

\begin{proof}
If $\mathcal{K}$ is as in \eqref{eq:Jun7rk1}, then \eqref{eq:Jun7rk2} obviously holds. Conversely, if \eqref{eq:Jun7rk2} holds, then $(x, y)$ and $(x, z)$ belonging to $\mathcal{K}$ implies that $y = z$, i.e., there exists a function $T: \cD(T) \to \cH_n$. The fact that $T \in \cL(\cH_n)$ follows easily from the right linearity of $\mathcal{K}$. Thus, \eqref{eq:Jun7rk1} holds.
\end{proof}

A very simple consequence of the polarisation formula \eqref{eq:POLAR} is the following lemma.

\begin{lem}
\label{lem:0lemma}
Suppose $T \in \cL(\cH_n)$ is a densely defined operator. If \\$\langle Tx, x\rangle = 0$ for all $x \in \cD(T)$, then $Tx = 0$ for all $x \in \cD(T)$.
\end{lem}

\begin{lem}
\label{lem:surjective}
Suppose $S, T \in \cL(\cH_n)$ such that $S \subseteq T$, $S$ is surjective and $T$ is injective. Then $S = T$.
\end{lem}

\begin{proof}
Fix $x \in \cD(T)$. Thus, as $S$ is surjective, we can find $y \in \cD(S)$ such that $S y = Tx$.  Since $S \subseteq T$, we have $Ty = Tx$. By the injectivity of $T$, we have $x = y$. Thus, $\cD(T) \subseteq \cD(S)$ and hence $\cD(T) = \cD(S)$, in which case $S = T$.
\end{proof}

\begin{defn}[graph norm]
Suppose $T \in \cL(\cH_n)$. It is easy to check that $\cD(T)$ is a right submodule of $\cH_n$ which can be endowed with $\langle \cdot, \cdot \rangle_T: \cD(T) \times \cD(T) \to \RR_n$ given by
\begin{equation}
\label{eq:RRnIP}
\langle x, y \rangle_T := \langle x, y \rangle + \langle Tx, Ty \rangle \quad \quad {\rm for} \quad x, y \in \cD(T),
\end{equation}
where $\langle \cdot, \cdot \rangle_T$ obeys \eqref{eq:IP1}-\eqref{eq:IP4}
and the corresponding norm
\begin{equation}
\label{eq:RRnNORM}
\| x \|_T := (\| x \|^2 + \| Tx \|^2 )^{1/2} \quad \quad {\rm for} \quad x \in \cD(T).
\end{equation}

\end{defn}

\begin{defn}[closed operator]
\label{def:Dec14tt9}
An operator $T \in \cL(\cH_n)$ is called {\it closed} if the set $\{ (x, Tx): x \in \cH_n \}$ is a closed subset of $\cH_n\times
\cH_n$ (endowed with the product topology). Let $S$ and $T$ both belong $\cL(\cH_n)$. We will write $S = T$ if $\cD(S) = \cD(T)$ and $Sx = Tx$ for all $x \in \cD(S) = \cD(T)$. We will write $S \subseteq T$ if $\cD(S) \subseteq \cD(T)$ and $Sx = Tx$ for all $x \in \cD(S)$. Clearly, $S = T$ if and only if $S \subseteq T$ and $T \subseteq S$. An operator $T \in \cL(\cH_n)$ is called {\it closable} if there exists a closed operator $X \in \cL(\cH_n)$ so that $T \subseteq X$.
\end{defn}

\begin{thm}
\label{thm:closed}
Let $T \in \cL(\cH_n)$. The follow statements are equivalent{\rm :}
\begin{enumerate}
\item[(i)] $T$ is closed.
\item[(ii)] For any sequence $(x_i)_{i=1}^{\infty}$, with $x_i \in \cD(T)$ for $i=1,2,\ldots$, such that
$$\lim_{i \to \infty} x_i = x,$$
where $x \in \cD(T)$, and
$$\lim_{i \to \infty} T(x_n) = y,$$
where $y \in \cH_n$, we have $Tx = y$.
\item[(iii)] $\cD(T)$ together with $\| \cdot, \cdot  \|_T$ {\rm(}see \eqref{eq:GRAPH}{\rm )} is a complete normed right module over $\RR_n$.
\end{enumerate}
\end{thm}

\begin{proof}
In view of Definition \ref{def:Dec14tt9}, ${\rm (i)} \Longleftrightarrow {\rm (ii)}$ is immediate. We will now show ${\rm (i)} \Longleftrightarrow {\rm (iii)}$. In view of \eqref{eq:RRnNORM}, $\cD(T)$ together with $\| \cdot \|_T$ is a complete normed right module over $\RR_n$ if and only if $\cG(T)$ is complete, i.e., $\cG(T)$ is a closed.
\end{proof}

\begin{thm}
\label{thm:Jun7pw2}
Let $T \in \cL(\cH_n)$. The following statements are equivalent{\rm :}
\begin{enumerate}
\item[(i)] $T$ is closable.
\smallskip
\item[(ii)] $\overline{ \{ (x, Tx): x \in \cD(T) \}} = \{ (x, Wx): \text{for some operator $W \in \cL(\cH_n)$} \}$.
\smallskip
\item[(iii)] For any sequence $(x_i)_{i=1}^{\infty}$, where $x_i \in \cD(T)$ for $i=1,2,\ldots$, such that
$$\lim_{i \to \infty} x_i = 0$$
and
$$\lim_{i \to \infty} T(x_i) = y,$$
where $y \in \cH_n$, then $y = 0$.
\end{enumerate}
\end{thm}

\begin{proof}
We will first show ${\rm (i)} \Longrightarrow {\rm (ii)}$. If $S \in \cL(\cH_n)$ is any closed operator such that $T \subseteq S$, then
$$\{ (x, Tx): x \in \cD(T) \} \subseteq \{ (x, Sx): x \in \cD(S) \}.$$
Hence, as $S$ is closed,
$$\overline{\{ (x, Tx): x \in \cD(T) \}} \subseteq \{ (x, Sx): x \in \cD(S) \}.$$
Therefore, in view of Lemma \ref{lem:Jun7tm2}, (ii) holds.

We will now show ${\rm (ii)} \Longrightarrow {\rm (i)}$. If (ii) holds, then $T \subseteq W$ and hence $W$ is closed since
$$\{ (x, Wx): x \in \cD(W) \}$$
is closed. Thus, $T$ is closable.

The proof of ${\rm (ii} \Longrightarrow {\rm (iii)}$ follows immediately from the Lemma \ref{lem:Jun7tm2} and the fact that $(0,y) \in \overline{\cG(T)}$ implying that $y = 0$.

\end{proof}

\begin{defn}
\label{def:Jun7er2}
Let $T \in \cL(\cH_n)$ be closable. We let
$$\overline{T}x := \lim_{i \to \infty} T(x_i)$$
denote the operator in $\cL(\cH_n)$ with domain
\begin{align*}
\cD(T) =& \; \{ x \in \cH_n: x = \lim_{i \to \infty} i_n\ \text{for $( x_i )_{n=0}^{\infty} \subseteq \cD(T)$ and $\{ T(x_i) \}_{i=0}^{\infty}$} \\
& \; \text{converges in $\cH_n$}\}.
\end{align*}
In view of Theorem \ref{thm:Jun7pw2}, the definition of $\overline{T}$ is independent of the choice of sequence $( x_i )_{i=0}^{\infty}$. Note that for any closed operator $W \in \cL(\cH_n)$ such that $T \subseteq W$,
$$\overline{T} \subseteq W.$$
\end{defn}

\begin{defn}[continuous operator]
\label{def:CONTOP}
Let $T \in \cL(\cH_n)$. $T$ will be called {\it continuous} if for any sequence $(x_i)_{i=1}^{\infty}$, where $x_i \in \cD(T)$ for $i=1,2,\ldots$, such that
$$\lim_{i \to \infty} x_i = x,$$
where $x \in \cD(T)$, we have
$$\lim_{i \to \infty} T\, x_i = Tx.$$
\end{defn}

\begin{thm}
\label{thm:CLOSEDCONT}
Let $T \in \cL(\cH_n)$. Then $T$ is continuous if and only if $T \in \cB(\cH_n)$.
\end{thm}

\begin{proof}
The proof can be carried out as in the classical complex Hilbert space case, see, e.g., \cite{Lax}.
\end{proof}

\begin{defn}
\label{def:CORE}
Let $T \in \cL(\cH_n)$. A subset $\cE$ of $\cD(T)$ will be called {\it core} of $\cD(T)$ if $\cE$ is a right submodule of $\cD(T)$ and $\cE$ is dense in $(\cD(T), \| \cdot \|_T)$.
\end{defn}

\begin{defn}[adjoint operator]
\label{def:Dec11rro3}
Given $T \in \cL(\cH_n)$ which is densely defined, we let $T^* \in \cL(\cH_n)$ denote the unique operator so that
$$\langle Tx, y \rangle = \langle x, T^* y \rangle, \quad x\in \cD(T), y \in \cD(T^*),$$
where the domain of $T^*$ is given by
\begin{align*}
\cD(T^*) =& \; \{\text{$y \in \cH_n$ : there exists $z \in \cH_n$ with $\langle Tx, y \rangle = \langle x, z \rangle$} \\
 & \text{for every $x \in \cD(T)$}\}.
 \end{align*}
\end{defn}

\begin{rem}
\label{rem:ADJOINTADJOINT}
Let $T \in \cB(\cH_n)$. Then, in view of Definition \ref{def:Dec11rro3}, it is easy to check that $(T^*)^* = T$.
\end{rem}

\begin{lem}
\label{lem:BDD}
Let $T, W \in \cB(\cH_n)$. Then $TW \in \cB(\cH_n)$,
\begin{equation}
\label{eq:PRODUCTBDD}
\| T W \| \leq \| T \| \cdot \| W \|
\end{equation}
and
\begin{equation}
\label{eq:C*identity}
\| T^* T \| =  \| T^* \|^2 = \| T \|^2.
\end{equation}
\end{lem}

\begin{proof}
The fact that $TW \in \cB(\cH_n)$, \eqref{eq:PRODUCTBDD} and \eqref{eq:C*identity} can be proved exactly as in the classical complex Banach algebra case, see, e.g., Theorem 8 on page 168 of \cite{Lax} for a proof of the classical complex Hilbert space case of \eqref{eq:PRODUCTBDD} and \cite{Conway} for a classical complex Hilbert space case of \eqref{eq:C*identity}.
\end{proof}

\begin{defn}[self-adjoint, anti self-adjoint and unitary]
\label{def:Sept15t1}
Let $T \in \cL(\cH_n)$. We will call $T$ {\it self-adjoint}, {\it anti self-adjoint} and {\it unitary} if $T = T^*$ with $\mathcal{D}(T)=\mathcal{D}(T^*)$, $T = - T^*$ with $\mathcal{D}(T)=\mathcal{D}(T^*)$ and $T T^* = T^* T = I$, respectively.
\end{defn}

\begin{rem}
\label{rem:SAASAU}
In view of Theorem \ref{thm:Dec14yy1}(i), we have that if $T \in \cL(\cH_n)$ is self-adjoint or anti self-adjoint operator, then $T$ is closed.
\end{rem}

\begin{thm}
\label{thm:Dec14yy1}
If $T \in \cL(\cH_n)$ is densely defined and $W \in \cL(\cH_n)$, then{\rm :}
\begin{enumerate}
\item[(i)] $T^* \in \cL(\cH_n)$ is closed.
\item[(ii)] ${\rm \Ran}(T)^{\perp} = {\rm Ker}(T^*)$.
\item[(iii)] If $T \subseteq W$, then $W^* \subseteq T^*$.
\end{enumerate}
\end{thm}

\begin{proof}
The proofs can completed in much the same way as the case when $\cH_n$ is a complex Hilbert space (see, e.g., Proposition 1.6 in \cite{Schmuedgen}).
\end{proof}

\begin{thm}
\label{thm:Dec14uu1}
If $T \in \cL(\cH_n)$ is densely defined, then{\rm :}
\begin{enumerate}
\item[(i)] $T$ is closable if and only if $\cD(T^*)$ is dense in $\cH_n$.
\item[(ii)] If $T$ is closable, then $\overline{T} = T^{**}$, where $T^{**} := (T^*)^*$.
\item[(iii)] $T$ is closed if and only if $T = T^{**}$.
\item[(iv)] If $T$ is closable and ${\rm Ker}(T) = \{ 0 \}$ and $\Ran(T)$ is dense in $\cH_n$, then $T^*$ is invertible and $(T^*)^{-1} = (T^{-1})^*$.
\item[(v)] If $T$ is closable and $\Ker(T) = \{ 0 \}$, then $T^{-1}$ is closable if and only if $\Ker(\overline{T}) = \{ 0 \}$. In this case, we have $(\overline{T})^{-1} = \overline{ T^{-1} }$.
\item[(vi)] Suppose $T$ is invertible. Then $T$ is closed if and only if $T^{-1}$ is closed.
\end{enumerate}
\end{thm}

\begin{proof}
The proofs can completed in much the same way as the case when $\cH_n$ is a complex Hilbert space (see, e.g., Theorem 1.8 in \cite{Schmuedgen}).
\end{proof}

\begin{lem}
\label{lem:GRAPHLEMMA}
Let $T \in \cL(\cH_n)$ be a densely defined operator. Then the graph of $T^*$ satisfies
\begin{equation}
\label{eq:GRAPH}
\cG(T^*) = V(\cG(T))^{\perp} = V(\cG(T)^{\perp}),
\end{equation}
where $V: \cH_n \oplus \cH_n \to \cH_n \oplus \cH_n$ denotes the unitary operator given by
$$V(x,y) =  (-y,x).$$
\end{lem}

\begin{proof}
Suppose $x \in \cD(T)$ and $y \in \cD(T^*)$. Using Definition \ref{def:Dec11rro3}, we have
\begin{align*}
\langle \langle V(x, Tx), (y, T^* y ) \rangle =& \; \langle (-Tx, x ), (y, T^* y ) \rangle \\
=& \; \langle -Tx, y \rangle + \langle x, T^* y \rangle \\
=& \; 0,
\end{align*}
in which case we have $\cG(T^*) \subseteq V( \cG(T))^{\perp}$.

Conversely, suppose $(y,z) \in V(\cG(T))^{\perp}$. Then for any $x \in \cD(T)$, we have
$$\langle V(x, Tx), (y,z) \rangle = \langle -Tx, y \rangle + \langle x, z \rangle =0,$$
and hence $\langle Tx, y \rangle = \langle x, z \rangle$. But then we have again use Definition \ref{def:Dec11rro3} to obtain $y \in \cD(T^*)$ and $z = T^*y$, i.e., $(y, z) \in \cG(T^*)$. Therefore, we have $V( \cG(T))^{\perp} \subseteq \cG(T^*)$.

The second equality is an immediate consequence of $V$ being unitary on $\cH_n \oplus \cH_n$.
\end{proof}

\begin{thm}[Riesz representation theorem for Clifford modules]
\label{thm:RRforHM}
Let $\cH_n$ be a Clifford module over $\RR_n$. Suppose $B: \cH_n \times \cH_n \to \RR_n$ is bounded, i.e., there exists $M \geq 0$ such that
$$|B(x,y)| \leq M \| x \| \| y \| \quad \quad {\rm for} \quad x,y \in \cH_n,$$
and satisfies the following{\rm :}
\begin{enumerate}
\item[(i)] $B(x+y, z) = B(x,z) + B(y,z)$ and $B(x,y+z) = B(x,y) + B(x,z)$ for $x,y,z \in \cH_n$.
\smallskip
\item[(ii)] $B(x \, a, y) = B(x,y) a$ and $B(x, y\, a) = \bar{a}\, B(x,y)$ for $x,y \in \cH_n$ and $a \in \RR_n$.
\end{enumerate}
Then there exists a unique $T \in \cB(\cH_n)$ such that
\begin{equation}
\label{eq:OPEQ}
B(x,y) = \langle T  x, y \rangle \quad \quad {\rm for} \quad x,y \in \cH_n.
\end{equation}
\end{thm}

\begin{proof}
The proof of Corollary 1.11 in \cite{Mitrea} can easily be adjusted to our present right Clifford module setting.
\end{proof}

The following theorem appears in \cite{Mitrea}.

\begin{thm}[Hahn-Banach theorem for a Clifford module]
\label{thm:HB}
Let $\cH_n$ be a Clifford module over $\RR_n$ and $\cS_n$ be a right submodule of $\cH_n$. Suppose $\sL : \cS_n \to \RR_n$ is continuous (i.e., bounded) and satisfies
\begin{equation}
\label{eq:HB}
\sL(x a + y) = \sL(x)a + \sL(y) \quad \quad {\rm for} \quad x,y \in \cS_n \quad {\rm and} \quad a \in \RR_n.
\end{equation}
Then $\sL$ has a continuous extension $\sL: \cH_n \to \RR_n$ {\rm (}with a slight abuse of notation, we shall use $\sL$ to denote the extension{\rm )} such that \eqref{eq:HB} holds for all $x, y \in \cH_n$ and $a \in \RR_n$.
\end{thm}

\begin{thm}[closed graph theorem for a Clifford module]
\label{thm:CGT}
Suppose $T\in \cL(\cH_n)$ is a closed operator with $\cD(T) = \cH_n$. Then $T \in \cB(\cH_n)$.
\end{thm}

\begin{proof}
The proof given in Theorem 2.2.7 in \cite{Pedersen} can easily be adjusted to the present Clifford module setting.
\end{proof}

\begin{lem}
\label{lem:invert}
Suppose $C \in \cB(\cH_n)$ is invertible in $\cB(\cH_n)$. Then the following statements hold{\rm :}
\begin{enumerate}
\item[(i)] $C - D$ is invertible in $\cB(\cH_n)$ whenever
$$\| D \| < \frac{ 1 }{ \| C^{-1} \| }.$$
\item[(ii)] Suppose $C$ is invertible in $\cB(\cH_n)$ and $\| C - D \| < \| C^{-1} \|^{-1}$. Then $D$ is invertible in $\cB(\cH_n)$.
\end{enumerate}
\end{lem}

\begin{proof}
The proof of (i) and (ii) is exactly the same as in the classical case, see, e.g., Theorem 2 in Chapter 17 of \cite{Lax} for the proof of (i).
\end{proof}

\begin{defn}[normal operator on a Clifford module]
\label{def:Aug24y1}
Suppose $T \in \cL(\cH_n)$. We will call $T$ {\it normal} if $T$ is densely defined, $T$ is closed, $\cD(T) = \cD(T^*)$ and $TT^* = T^* T$.
\end{defn}

\begin{lem}
\label{lem:Aug25yq1}
Let $T \in \cL(\cH_n)$ be normal. If $S \in \cL(\cH_n)$ so that $T \subseteq S$ and $\cD(S) \subseteq D(S^*)$, then $S = T$.
\end{lem}

\begin{proof}
If $T \subseteq S$, then $S^* \subseteq T^*$ and hence
$$\cD(T) \subseteq \cD(S) \subseteq \cD(S^*) \subseteq \cD(T^*) = \cD(T),$$
i.e., $\cD(S) = \cD(T)$. Therefore, $S = T$.
\end{proof}

\subsection{$S$-resolvent set, $S$-spectrum for a linear operator and the spectral radius formula}

\begin{defn}[$S$-resolvent set and $S$-spectrum for $T \in \cL(\cH_n)$]
\label{def:Dec14j1}
Let \\
$T \in \cL(\cH)$ be densely defined and $\mathcal{Q}_s(T): \cD(T^2) \to \cH$ be given by
\begin{equation}
\label{eq:Q}
\mathcal{Q}_s(T)x  =(T^2 - 2 {\rm Re}(s) T + |s|^2 I )x \quad {\rm for} \quad x \in \mathcal{D}(T^2).
\end{equation}
The $S$-resolvent set of $T$ is defined as follows
\begin{align*}
\rho_S(T)
=& \; \{\text{$s\in \RR^{n+1}$ {\rm :} $\Ker \,\mathcal{Q}_s(T)=\{0\}$, $\Ran \,\mathcal{Q}_s(T)$ is dense in $\cH_n$ and} \\
& \; \text{$\mathcal{Q}_s(T)^{-1} : \Ran \cQ_s(T) \to \cD(T^2$) is bounded}\}.
\end{align*}
The $S$-spectrum is defined as
$$
\sigma_S(T)=\RR^{n+1} \, \setminus\,  \rho_S(T).
$$
\end{defn}

\begin{rem}
\label{rem:PARAVECTOR}
In \cite{CSS}, the $S$-spectrum for $T \in \cL(\cH_n)$ was considered
for paravector operators, i.e., operators of the form
$$
T = \sum_{i=0}^n e_{i}T_{i},
$$
because the purpose of this theory was to define
 a functional calculus for $(n+1)$-
 tuples of noncommuting operators.
Quite early on in our investigation of the Clifford spectral theorem, we directed our attention to fully Clifford operators because of the crucial decomposition
\begin{equation}
T = A  +J_0 B.
\end{equation}
of normal bounded operators given in Theorem \ref{thm:AJB}.
This decomposition implies that, even if $T$ is a paravector operator,
 the operators $J_0$ and $B$ will not be paravector operators in general.

\medskip
A further observation is that
in the case $T$ is a paravector operator with noncommuting components
the $S$-spectrum is defined by
the operator $T^2 - 2 {\rm Re}(s) T + |s|^2 I$ that is not a paravector.

\medskip
For an explanation of why this theory is so flexible see the introduction of the paper \cite{CKPS},
where some considerations are made on various Cauchy formulas that
define various holomorphic functional calculi.
The properties of the $S$-spectrum and the $S$-resolvent operators
for fully Clifford operators
remain the same with the same proofs valid for
paravector operators, see \cite{CGKSfunc,CKPS,GhiloniRecupero}.
\end{rem}

\begin{rem}
\label{rem:ALTNERATIVE}
For any $T \in \cL(\cH_n)$, it is easy to show that  the $S$-resolvent set is equal to
\begin{align}
\rho_S(T) =  \{ s \in \RR^{n +1}: \cQ_s(T)^{-1}: & \;\Ran \, \cQ_s(T) \to \cD(T^2) \nonumber \\
& \; \text{is bijective and bounded} \} \label{eq:SRESOLVENT}
\end{align}
\end{rem}

\begin{defn}[left and right $S$-resolvent operator]
\label{def:Sresolvent}
Suppose $T \in \cB(\cH_n)$. For any choice of $s \in \rho_S(T)$, we shall define the {\it left $S$-resolvent operator} via
\begin{equation}
\label{eq:LEFTRESOLVENT1}
S_L^{-1}(s, T) := -\cQ_s(T)^{-1}(T-\bar{s}\, I) \in \cB(\cH_n)
\end{equation}
and the {\it right $S$-resolvent operator} via
\begin{equation}
\label{eq:RIGHTRESOLVENT1}
S_R^{-1}(s, T) := -(T-\bar{s}\, I)\cQ_s(T)^{-1} \in \cB(\cH_n).
\end{equation}
\end{defn}

\begin{rem}
\label{rem:SLICE}
One can check that $S_L^{-1}(s,T): \rho_S(T) \to \cB(\cH_n)$ is a {\it right slice hyperholomorphic function} (see Lemma 3.10 in \cite{CGKSfunc} for details). Analogously, one can check that $S_R^{-1}(s,T) : \rho_S(T) \to \cB(\cH_n)$ is a {\it left slice hyperholomorphic function}.
\end{rem}

\begin{defn}[axially symmetric]
We will call a set $\Omega \subseteq \RR^{n+1}$ {\it axially symmetric}, if whenever $s = s_0 + s_1 \gI \in \Omega$, where $s_0,s_1 \in \RR$ and $\gI \in \mathbb{S}$,  then  $s_0 + s_1 \gJ \in \Omega$ for all $\gJ \in \mathbb{S}$.
\end{defn}

\begin{rem}[the $S$-spectrum is axially symmetric]
\label{rem:AXIALLYSYMMETRIC}
Let $T \in \sL(\cH_n)$. Then $\sigma_S(T)$ is axially symmetric. If $\sigma_S(T) = \emptyset$, then there is nothing to prove. If $\sigma_S(T) \not= \emptyset$, then one need only notice that $s \in \rho_S(T)$ depends only on ${\rm Re}(s)$ and $|s|$.
\end{rem}

\begin{thm}[properties of the $S$-spectrum for $T \in \cB(\cH_n)$]
\label{thm:SPECNONEMPTYCOMPACT}
Suppose \\$T \in \cB(\cH_n)$. Then $\sigma_S(T)$ is a non-empty compact subset of
$$\{ s \in \RR^{n+1}: 0 \leq |s| \leq \| T \| \}.$$
\end{thm}

\begin{proof}
We will first show that $\sigma_S(T)$ is non-empty.  For any choice of $\varepsilon > 0$,  the series $\sum_{i=0}^{\infty} T^i s^{-i-1}$ converges uniformly in norm on $s \in \Omega_{\varepsilon} := \{ a \in \RR^{n+1}: |a| = \| T \| + \varepsilon \}$ to $S_L^{-1}(s,T)$. Thus, for any $\gI \in \mathbb{S}$, we may use the fact that
$$\int_{\Omega_{\varepsilon} \cap \CC_{\gI}} s^{-i-1} \, ds (-\gI) = \begin{cases}
2\pi & {\rm if} \quad i = 0, \\
0 & {\rm if}\quad i \neq 0,
\end{cases}
$$
to obtain
\begin{equation}
\label{eq:CAUCHY}
\int_{\Omega_{\varepsilon}} S_L^{-1}(s,T) ds (-\gI) = \sum_{i=0}^{\infty} T^i \left( \int_{\Omega_{\varepsilon}} s^{-i-1} ds (-\gI) \right) = 2\pi \, I.
\end{equation}
If $M_{\varepsilon}:=\overline{\{ a \in \RR^{n+1}: |a| \leq \| T \| + \varepsilon \}} \subseteq \rho_S(T)$, then Remark \ref{rem:SLICE} asserts that $S_L^{-1}(s,T)|_{M_{\varepsilon} }$ is a right slice hyperholomorphic $\cB(\cH_n)$-valued function. However, an analogue of Cauchy's integral formula for $\cB(\cH_n)$-valued functions asserts that \eqref{eq:CAUCHY} must be $0$, which is clearly not the case. Thus, $M_{\varepsilon}$ is not a subset of $\rho_S(T)$. But then $\sigma_S(T)$ cannot be empty.

We will now show that $\sigma_S(T)$ is a closed subset of $\RR^{n+1}$. Notice that $\varphi: \RR^{n+1} \to \cB(\cH_n)$ given by $\varphi(s) = \cQ_s(T)$ is a continuous function. Lemma \ref{lem:invert} can be used to show that the set of invertible operators in $\cB(\cH_n)$ is open. Thus, $\rho_S(T) = \varphi^{-1}(\{ W \in \cB(\cH_n): W^{-1} \in \cB(\cH_n) \})$ is open. Consequently, $\sigma_S(T) = \RR^{n+1} \setminus \rho_S(T)$ is closed.

Next, we will show that $\sigma_S(T) \subseteq \{ s \in \RR^{n+1}: |s| \leq \| T \| \}$. In view of \eqref{eq:PRODUCTBDD}, we have $\| T^i \| \leq \|T \|^i $ for $i=0,1,\ldots$. Thus, as $s \in \RR^{n+1}$, we have the estimate
$$\| T^i s^{-i-1} \| \leq 2^{n-1} |s^{-i-1}|\cdot \| T^i\| = \frac{\| T\|^i }{|s|^{i+1}} \quad \quad {\rm for}\quad i=0,1,\ldots$$ and hence
\begin{equation}
\label{eq:CONVERGENCE}
\text{$\sum_{i=0}^{\infty} \| T^i s^{-i-1} \|$ converges if and only if $|s| < \| T \|$}.
\end{equation}
We now claim that
\begin{equation}
\label{eq:CLAIMSRESOLVENT}
(T^2 -2{\rm Re}(s)T + |s|^2 \, I) \left( \sum_{i=0}^{\infty} T^i s^{-i-1} \right) = \bar{s}\, I - T.
\end{equation}
Indeed, notice that
\begin{align*}
& \; (T^2 -2{\rm Re}(s)T + |s|^2 \, I) \left( \sum_{i=0}^{\infty} T^i s^{-i-1} \right) \\
=& \; \sum_{i=0}^{\infty} \left( T^{i+1} s^{-i-1} - T^{i+1} s^{-i-1}(s+\bar{s}) + T^i s^{-i-1} s \bar{s}   \right) \\
=& \; \sum_{i=0}^{\infty} \left( T^{i+1} s^{-i} - T^{i+1} s^{-i} - T^{i+1} s^{-i-1} \bar{s} + T^i s^{-i} \bar{s} \right) \\
=& \; \bar{s} \, I - T.
\end{align*}
Thus,  we have \eqref{eq:CLAIMSRESOLVENT}. Putting together \eqref{eq:CONVERGENCE} and \eqref{eq:CLAIMSRESOLVENT}, we have that $\sigma_S(T) \subseteq \{ s \in \RR^{n+1}: |s| \leq \| T \| \}$. As $\sigma_S(T)$ is closed, we have that is a closed subset of a compact set, i.e., $\sigma_S(T)$ is compact.

Finally, we will show that $\sigma_S(T)$ is closed. Suppose $s \in \rho_S(T)$ and $\tilde{s} \in \RR^{n+1}$ be such that $|s - \tilde{s}|$ is sufficiently small. Then
\begin{align*}
\cQ_s(T) - \cQ_{\tilde{s}}(T) =& \; 2 {\rm Re}(\tilde{s} - s)T + (|s|^2 - |\tilde{s}|^2)I \\
=& \; \frac{1}{|s|^2 - |\tilde{s}|^2 } \left(    \frac{2 {\rm Re}(\tilde{s} - s ) }{|s|^2 - |\tilde{s}|^2 } \, T + I \right),
\end{align*}
one can use Lemma \ref{lem:INVERT} to see that $\cQ_{\tilde{s}}$ is invertible in $\cB(\cH_n)$. Thus, $\tilde{s} \in \rho_S(T)$ and $\rho_S(T)$ is open, i.e., $\sigma_S(T)$ is closed.
\end{proof}

\begin{defn}[spectral radius]
Let $T \in \cB(\cH_n)$. Then we shall denote the {\it spectral radius} of $T$ by $r_S(T)$ which is given by
\begin{equation}
\label{eq:SR}
r_S(T) := \sup_{s \in \sigma_S(T)} |s| = \max_{s \in \sigma_S(T)} |s|.
\end{equation}
\end{defn}

\begin{thm}[spectral radius formula]
\label{thm:SRF}
Let $T \in \cB(\cH_n)$ be normal. Then
\begin{equation}
\label{eq:SRF}
\| T \| = r_S(T).
\end{equation}
\end{thm}

\begin{proof}
Theorem 6.7 in [\cite{CGKSfunc} asserts, in particular, that for any $T \in\cB(\cH_n)$, we have
$$\lim_{i \to \infty} (\| T^i \|_1)^{\frac{1}{i}} =r_S(T),$$
where
$$\| T \|_1 := \sum_{\alpha} \| T_{\alpha} \|_{\cH} \quad\quad {\rm for}\quad T = \sum_{\alpha} T_{\alpha} e_{\alpha}.$$
Notice that
\begin{align*}
\| T\| :=& \; \sup_{\| x\|=1} \| T x \| = \sup_{\|x\|=1} \| \sum_{\alpha,\beta} T_{\alpha}x_{\beta}e_{\alpha} e_{\beta} \| \\
\leq& \; \sup_{\|x\|=1} \sum_{\alpha,\beta} \|T_{\alpha} x_{\beta} e_{\alpha} e_{\beta} \| \\
\leq& \; 2^{n-1}\sup_{\|x\| =1 } \sum_{\alpha,beta} \| T_{\alpha} x_{\beta} \|_{\cH} \cdot |e_{\alpha}e_{\beta}| \\
\leq& \; 2^{n-1}2^{n-1} \sum_{\alpha,\beta} \left( \sup_{\|x\|=1} \|T_{\alpha}x_{\beta}\|_{\cH} \right)\\
\leq& \; 2^{2(n-1)} \sum_{\alpha,\beta} \|T_{\alpha}\|_{\cH} \\
=& \; 2^{3n-2} \| T  \|_1.
\end{align*}

In view of \eqref{eq:C*identity}, one can use induction on $i$ to show that
$\| T^{2^i} \|^{\frac{1}{2^i}} = \| T \|$. In view of Theorem \ref{thm:SPECNONEMPTYCOMPACT}, we have $r_S(T) \leq \| T\|$. Thus,
$$
r_S(T) \leq  \|T \| = \|T^{2^i} \|^{\frac{1}{2i}} \quad\quad {\rm for} \quad i=1,2,\ldots
$$
and hence
\begin{align*}
r_S(T) \leq& \; \lim_{i \to \infty} \|T^{2^i} \|^{\frac{1}{2i}} \\
=& \; \lim_{i \to \infty} \left( 2^{\frac{3n-2}{i}} ( \| T^{2^i} \|_1)^{\frac{1}{2^i}} \right)\\
=& \; r_S(T).
\end{align*}
and hence we have \eqref{eq:SRF}.
\end{proof}

Following the possible splittings of the classical spectrum for operators on complex Banach space we can give the same spitting also for fully linear operators on Clifford modules.
We recall that the splitting of the spectrum is defined according to where an operators is not invertible.
We will mention two possible splittings. The first one is the point,  residual, continuous  $S$-spectrum of a Clifford operator.

\begin{defn}[point,  residual, continuous  $S$-spectrum]
\label{def:SPECTRUMS}
Let $T:\mathcal{D}(T)\to \mathcal{H}_n$.  We split the $S$-spectrum into the three disjoint sets:
\begin{itemize}
\item[(P)] The point $S$-spectrum of $T$:
$$
\sigma_{PS}(T)=\{ s\in \mathbb{R}^{n+1} \ :\ \ {\rm Ker} ( \mathcal{Q}_s(T))\not=\{0\} \}.
$$
\item[(R)] The residual  $S$-spectrum of $T$:
$$
\sigma_{RS}(T)=\left\{ s\in \mathbb{R}^{n+1} \ :\ \ {\rm Ker} ( \mathcal{Q}_s(T))=\{0\},\ \overline{{\rm Ran}(\mathcal{Q}_s(T))}\not=\mathcal{H}_n \right\}.
$$
\item[(C)] The continuous  $S$-spectrum of $T$:
\end{itemize}
\begin{align*}
\sigma_{CS}(T)=& \; \{ s\in \mathbb{R}^{n+1} \ :\  {\rm Ker} ( \mathcal{Q}_s(T))=\{0\},\
\overline{{\rm Ran}(\mathcal{Q}_s(T))}
=\mathcal{H}_n \\
& \;\; {\rm and} \; \mathcal{Q}_s(T)^{-1} \not\in \mathcal{B}(\mathcal{H}_n)\}.
\end{align*}
\end{defn}
\label{def:MORESPECTRUM}
\begin{rem}
Notice that if $A \in \cB(\cH_n)$ that satisfies the two conditions:
\begin{enumerate}
\item[(i)] There exists $K>0$ such that $\|Av\|\geq K\|v\|$ for $v\in \mathcal{D}(A)$ (bounded from below)
\item[(ii)] the range of $A$ is dense in $\cH_n$,
\end{enumerate}
then $A$ is invertible.
\end{rem}
So in analogy to the classical case for the $S$-spectrum we have:
 \begin{defn}[approximate point and compression $S$-spectrum]
 Let $T$ be a Clifford bounded linear operator.
The approximate point $S$-spectrum of $T$, denoted by $\Pi_{S}(T)$, is defined as
$$
\Pi_{S}(T)=\{s\in \mathbb{R}^{n+1}\ :\  T^2 - 2{\rm Re}(s)T + |s|^2\mathcal{I} \text{ is not bounded from below}\}.
$$
The compression $S$-spectrum of $T$, denoted by $\Gamma_{S}(T)$, is defined as
$$
\Gamma_{S}(T)=\{s\in \mathbb{R}^{n+1}\ :\   \text{ the range of }\ T^2 - 2{\rm Re}T + |s|^2\mathcal{I} \text{ is not dense}\}.
$$
 \end{defn}
The set $\Pi_{S}(T)$ contains the $S$-eigenvalues.

\subsection{Basic facts for normal operators}
\label{sec:SA}

In this section we are defining self-adjoint, anti self-adjoint operators and positive operators on a Clifford module. We are also formulating and proving a number of facts which will be useful when proving the spectral theorem for a bounded self-adjoint operator (see Section \ref{sec:STBSA}) and also the spectral theorem for an unbounded normal operator (see \ref{sec:UBN}).

For the quaternionic setting similar properties hold, see the book
\cite{6CKG} and the references therein.

\begin{lem}
\label{lem:SA}
Let $T \in \cL(\cH_n)$ be given. The following statements are equivalent{\rm :}
\begin{enumerate}
\item[(i)] $T$ is self-adjoint.
\item [(ii)] $\overline{\langle Tx, x \rangle} = \langle Tx, x \rangle$ for all $x \in \cD(T)$.
\end{enumerate}
\end{lem}

\begin{proof}
If $T$ is self-adjoint, then
$$\langle Tx, x \rangle = \langle x, T x \rangle \quad\quad {\rm for} \quad x \in \cD(T)$$
and hence
\begin{equation}
\label{eq:QF}
\overline{\langle x, Tx \rangle } = \langle Tx, x \rangle\quad\quad {\rm for} \quad x \in \cD(T).
\end{equation}
On the other hand, if (ii) is in force, then \eqref{eq:POLAR} can be used to show that $\langle Tx, y \rangle = \langle x, T y \rangle$ for all $x, y \in \cD(T)$. Thus, (i) holds.
\end{proof}

\begin{defn}[positive operator on a Clifford module]
\label{def:POS}
Let $T \in \cL(\cH_n)$ be given. The operator $T$ is called {\it positive} if $\langle Tx, x \rangle \succeq 0$ for all $x \in \cD(T)$.
\end{defn}

The following theorem will be useful when considering the unbounded case of the spectral theorem for normal operators.

\begin{thm}
\label{thm:Ctransform}
Suppose $T \in \cL(\cH_n)$ is a densely defined closed operator. Then the following statement hold{\rm :}
\begin{enumerate}
\item[(i)] $I+ T^*T$ is a bijective mapping on $\cH_n$. If $C_T := (I + T^*T)^{-1} \in \cB(\cH_n)$, then $C_T \in \cB(\cH_n)$, $C_T$ is positive and $I- C_T$ is positive.
\item[(ii)] The operator $T^*T \in \cL(\cH_n)$ is positive and $\cD(T^*T)$ is a core for $T$. In particular, if $T$ is self-adjoint, then $\cD(T^2)$ is a core for $T$.
\end{enumerate}
\end{thm}

\begin{proof}
We will first prove (i). In view of Lemma \ref{lem:GRAPHLEMMA}, we have $\cG(T^*) = V(\cG(T))^{\perp}$, where $V$ denotes the unitary operator in the statement of Lemma \ref{lem:GRAPHLEMMA}. Notice that $\cH_n \oplus \cH_n = \cG(T^*) \oplus V(\cG(T))$. Consequently, corresponding to every $z \in \cH_n$, there exist $x \in \cD(T)$ and $y \in \cD(T^*)$ such that
$$(0, z) = (y, T^*y) = V(x, Tx) = (y-Tx, T^*y + x).$$
Thus, $y = Tx$ and $z = x+T^*y = (I+ T^*T)x$, in which case the operator $I+T^*T\in \cL(\cH_n)$ is surjective. To see that $I+T^*T$ is injective, notice that for any $x,y \in \cD(T^*T)$, we have
\begin{align}
\| (I + T^*T)(x-y) \|^2 =& \; {\rm Re} \langle (x-y)+T^*T(x-y), (x-y) +T^* T(x-y) \rangle \nonumber \\
=& \; {\rm Re}\,  \langle x-y, x-y \rangle + 2 {\rm Re}\, \langle T(x-y), T(x-y) \rangle \\
 +& \;  {\rm Re}\, \langle T^*T(x-y), T^*T(x-y) \rangle  \nonumber \\
=& \; \| x-y \|^2 + 2 \| T(x-y) \|^2 +  \| T^*T(x-y) \|^2. \label{eq:IMPORT}
\end{align}
Thus, as $I+T^*T$ is injective and surjective, we have that $I+T^*T$ is bijective and let $C_T:= (I+T^*T)^{-1}$.

Suppose $z = (I+T^*T)x$ for some $x \in \cD(T^*T)$. Then $C_T z = x$ and one can use \eqref{eq:IMPORT} with $y = 0$ to obtain
$$\| C_T z \| = \| x \| \leq \| (I+T^*T) x \| = \| z \|$$
and hence $C_T \in \cB(\cH_n)$. The fact that $C_T$ is self-adjoint is an immediate consequence of
$$\langle C_T z, z \rangle = \langle x, z \rangle = \langle x, (I+T^*T)x \rangle = \langle (I+T^*T)x, x \rangle.$$
The remaining conclusions in (i) can be easily justified via \eqref{eq:IMPORT}.

We will now prove (ii). Since $C_T \in \cB(\cH_n)$ is self-adjoint, we have that $C_T^{-1}$ is also self-adjoint by Theorem \ref{thm:Dec14uu1}(iv), i.e., $I+T^*T$ is self-adjoint. The fact that $T^*T$ is positive follows from
$$\langle T^*Tx, x \rangle = \langle Tx, Tx \rangle \succeq 0 \quad {\rm for} \quad x \in \cD(T^*T).$$
To check that $\cD(T^*T)$ is a core for $T$, we must check that $\cD(T^*T)$ is dense in the Clifford module $(\cD(T), \| \cdot \|_T)$. If $y \in \cD(T)$ satisfies $\langle y, x \rangle_T = 0$ for all $x \in \cD(T^*T)$, then realise that
$$0 = \langle y, x \rangle + \langle Ty, Tx \rangle = \langle y, (I+T^*T)x \rangle$$
for every $x \in \cD(T^*T)$. Consequently, $y = 0$ since $\Ran(I+T^*T) = \cH_n$, in which case we have that $\cD(T^*T)$ is a core for $T$.

The second assertion is very obvious.
\end{proof}

\begin{lem}
\label{lem:15July1}
Let $T \in \cL(\cH_n)$ be self-adjoint. Then $\sigma_S(T) \subseteq \RR$.
\end{lem}

\begin{proof}
In order to show that $\sigma_S(T) \subseteq \RR$, it suffices to show that
$\rho_S(T) \subseteq \{ s \in \RR^{n+1}: {\rm Re}(s) = 0 \}$. Let $s = s_0 + s_1 \in \RR^{n+1}$, where $s_0 \in \RR$ and $s_1 \in {\rm Im} (\RR^{n+1} ) \setminus \{ 0 \}$. Notice that
$$\cQ_s(T) := T^2 - 2 {\rm Re}(s) T + |s|^2 I = (T - s_0\,  I)^2 + |s_1|^2 I$$
and hence $\cD(T^2) = \cD( (T-s_0 \, I)^2)$ and as $T$ is self-adjoint, we have that $(T - s_0 \, I)^2 + |s|^2 I$ is self-adjoint.  Moreover, in view of the second assertion in Theorem \ref{thm:Ctransform}(ii), we have that $\cD(T^2)$ is dense in $\cH_n$. We claim that ${\rm Re}(\langle (T - s_0 \, I)^2 \, x, x \rangle) \geq 0$ for all $x \in \cD(T^2)$. Indeed,
\begin{align*}
{\rm Re}( \langle (T- s_0\, I )^2 x, x \rangle =& \; {\rm Re}( \langle (T-s_0\, I)x, (T- s_0\, I)x \rangle ) \\
=& \; \| (T-s_0\, I)x \|^2 \geq 0 \quad\quad {\rm for} \quad x \in \cH_n.
\end{align*}
Notice that
\begin{align}
\| \cQ_s(T) x \|^2 =& \; {\rm Re} \langle \{ (T-s_0 \, I)^2 + |s_1|^2 I \} x, \{ (T-s_0 \, I)^2 + |s_1|^2 \, I \} x \rangle \nonumber \\
=& \; \| (T-s_0\, I)^2 x \|^2 + 2 {\rm Re} \langle (T- s_0\, I)^2 x, |s_1|^2 x \rangle + |s_1|^4 \| x \|^2 \nonumber \\
=& \; \| (T-s_0 \, I)^2 x \|^2 + 2 |s_1|^2 \| (T-s_0 \, I)x \|^2 + |s_1|^4 \| x \|^2 \nonumber \\
\geq& \; |s_1|^4 \| x \|^2 \quad\quad {\rm for} \quad x \in \cD(T^2) \label{eq:D2}
\end{align}
and hence $\cQ_s(T)^{-1}: \Ran \cQ_s(T)\to \cD(T^2)$ is bounded (just take $x = \cQ_s(T)^{-1} y$ for $y \in \Ran \cQ_s(T)$ in \eqref{eq:D2}) and $\Ker \cQ_s(T) = \{ 0 \}.$ As $\cQ_s(T)$ is self-adjoint, we have that
\begin{align*}
\overline{ \Ran \cQ_s(T) } =& \; \Ran ( \cQ_s(T)^{\perp} )^{\perp} \\
=& \; \Ker ( \cQ_s(T)^* )^{\perp} \\
=& \; \Ker \cQ_s(T) )^{\perp} \\
=& \; \{ 0 \}^{\perp} = \cH_n.
\end{align*}
Thus, for all $s = s_0 + s_1$, with $s_1 \in {\rm Im}(\cH_n) \setminus \{ 0 \}$, we have $s \in \rho_S(T)$, i.e., $\rho_S(T) \subseteq \RR^{n+1} \setminus \RR$, i.e., $\sigma_S(T) \subseteq \RR$.

\end{proof}

\begin{lem}
\label{lem:15July1pos}
Let $T \in \cL(\cH_n)$ be a positive operator. Then $\sigma_S(T) \subseteq [0, \infty)$.
\end{lem}

\begin{proof}
A careful inspection of the proof of Lemma \ref{lem:15July1}, bearing in mind the additional hypothesis that $T$ is a positive operator, will reveal that for all $s < 0$, we have that $\cQ_s(T)^{-1} \in \cB(\cH_n)$ and hence $(- \infty, 0) \subseteq \rho_S(T)$, i.e., $\sigma_S(T) \subseteq [0, \infty)$.
\end{proof}

\begin{lem}
\label{lem:15July1ASA}
Let $T \in \cL(\cH_n)$ be anti self-adjoint. Then
\begin{align}\sigma_S(T) \subseteq {\rm Im}(\RR^{n+1}) := \{ a \in \RR^{n+1}: {\rm Re}(a) = 0 \}.\label{eq:15Julye1ASA}
\end{align}
\end{lem}

\begin{proof}
Since $T = -T^*$, we have
\begin{align}
\| \cQ_s(T) x \|^2 =& \; {\rm Re} \langle T^2 x, T^2  x \rangle + 2 (s_0^2 - |s_1|^2) {\rm Re}\langle Tx, Tx \rangle
+ (s_0^2 + |s_1|^2) {\rm Re}\langle x, x \rangle \nonumber \\
=& \; \| T^2 x \| + 2(s_0^2 - |s_1|^2) \| Tx \|^2 + (s_0^2 + |s_1|^2) \| x\|^2 \quad {\rm for} \; x \in \cD(T^2). \label{eq:DD2}
\end{align}
Therefore, if $|s_0| \geq |s_1|$, then we have
\begin{equation}
\label{eq:DD3}
\| \cQ_s(T) x \|^2 \geq s_0^4 \| x \|^2 \quad \quad {\rm for} \quad x \in \cD(T^2)
\end{equation}
and we may proceed as in Lemma \ref{lem:15July1} to show that $s = s_0 + s_1 \in \rho_S(T)$, i.e., \eqref{eq:15Julye1ASA} holds.

If $|s_0| < |s_1|$, then
$(s_0^2 - |s_1|^2) \| T^2 x \|  \|x \| \leq \| (s_0^2-|s_1|^2) \| T x \|^2$ for $x \in \cD(T)$ together with \eqref{eq:DD2} can be used to show that \eqref{eq:DD3} holds. But then we may proceed as above to obtain \eqref{eq:15Julye1ASA}.
\end{proof}

\begin{thm}
\label{thm:SPLITTING}
Let $T\in \mathcal{B}(\mathcal{H}_n)$ be a normal operator. Then we have
$$
\sigma_{PS}(T)=\sigma_{PS}(T^*),
\ \ \ \ \ \ \sigma_{RS}(T)=\sigma_{RS}(T^*)=0,
\ \ \ \ \ \sigma_{CS}(T)=\sigma_{CS}(T^*).
$$
\end{thm}
\begin{proof}
Since $T$ is normal and $\mathcal{Q}_{s}(T)^*=\mathcal{Q}_{s}(T^*)$ it is clear that $\mathcal{Q}_{s}(T)^*$ is normal.
For bounded linear operators the kernel $T$ and the kernel of its adjoint are equal so
$$
\Ker (\mathcal{Q}_{s}(T))=\Ker (\mathcal{Q}_{s}(T^*))
$$
so by the definition of point $S$-spectrum se have
$$
\sigma_{PS}(T)=\sigma_{PS}(T^*).
$$
The fact that $\sigma_{RS}(T)=\sigma_{RS}(T^*)=0$ follows by contradiction, in fact if $0\not=s\in \sigma_{RS}(T)$ we get
$$
\{0\}=\Ker (\mathcal{Q}_{s}(T))=\Ker (\mathcal{Q}_{s}(T^*))=(\Ran(\mathcal{Q}_{s}(T))^\perp\not=\{0\}.
$$
In the same way we can prove that $\sigma_{RS}(T^*)=0$.
Since $T$ and $T^*$ have the same $S$-spectrum and the three components of the $S$-spectrum, by definition,
 are pairwise disjoint it follows that $\sigma_{CS}(T)=\sigma_{CS}(T^*)$.

\end{proof}

\section{Measure theory and integration theory for $\RR_n$-valued measures}
\label{sec:MT}

\begin{defn}[positive $\RR_n$-valued measure]
Let $\Omega$ be a non-empty set and $\sA$ denote a $\sigma$-algebra on $\Omega$. We will call $\mu: \sA \to \RR_n \cup \{ \infty\}$ {\it positive} if $\mu(M) \succeq 0$ (see Definition \ref{def:CLIFFORD}) for every $M \in \sA$ such that $\mu(M) \neq \infty$ and $\mu$ is $\sigma$-additive, i.e.,
\begin{equation}
\label{eq:SIGMAADDITIVITY}
\mu\left( \bigcup_{i=1}^{\infty} M_n \right) = \sum_{i=1}^{\infty} \mu(M_n)
\end{equation}
for every sequence $M := (M_i)_{i=1}^{\infty}$, where $M_i \cap M_j = \emptyset$ for $i \neq j$ and $M_i \in \sA$ for $i=1,2,\ldots$.  In this case, we shall write $\mu$ is $\gP(\RR_n)$-valued.
\end{defn}

\begin{defn}[finite, semi-finite and $\sigma$-finite $\gP(\RR_n)$-valued measure]
\label{def:FINITEMEASURE}
Let $\Omega$ be a non-empty set and $\sA$ denote a $\sigma$-algebra on $\Omega$. We will call a $\gP(\RR_n)$-valued measure $\mu$ {\it finite} if $\mu(\Omega) < \infty$, {\it semi-finite} if for every $M \in \sA$ such that
$$\text{$\mu(M) = \infty$, there exists $N \in \sA$ such that $\mu(N) < \infty$ and $N \subseteq M$}$$
and {\it $\sigma$-finite} if
$$\Omega = \bigcup_{i=1}^{\infty} M_i,$$
where $\mu(M_i) < \infty$ for $i=1,2,\ldots$. Note that if $\mu$ is finite, then the finite additivity of $\mu$ together with $\mu(\Omega) < \infty$ implies that $\mu(M) < \infty$ for every $M \in \sA$.
\end{defn}

\begin{defn}[Borel measure]
Let $X$ be a Hausdorff space. We will call a $\gP(\RR_n)$-valued measure $\mu$ on the Borel $\sigma$-algebra generated by $X$ a {\it positive $\RR_n$-valued Borel measure}.
\end{defn}

\begin{defn}[$\RR_n$-valued measure]
\label{def:JORDAN}
Let $\Omega$ be a non-empty set and $\sA$ denote a $\sigma$-algebra on $\Omega$.  We will call $\mu: \sA \to \RR_n$ a $\RR_n$-valued measure if $\mu$ is $\sigma$-additive. In this case, $\mu$ has the {\it Jordan decomposition}
\begin{equation}
\label{eq:JORDAN}
\mu = \sum_{\alpha} \{ \mu^{(\alpha)}_+ - \mu^{(\alpha)}_- \} e_{\alpha},
\end{equation}
where $\mu_{\pm}^{(\alpha)}$ are positive measures (in the usual sense) for every \\$\alpha \in \powerset( \{1 ,\ldots, n \})$.
The {\it support of a $\RR_n$-valued measure $\mu$} shall be denoted by $\supp \mu$ and is defined as the set $N$ which satisfies
$$\mu(M) = \mu(M \cap N)  \quad\quad {\rm for} \quad M \in \sA.$$
We will call $\mu$ {\it finite} if $\mu(\Omega) < \infty$.
\end{defn}

\begin{defn}[Integral with respect to a $\RR_n$-valued measure]
Let $\mu$ be a finite $\RR_n$-valued measure on a non-empty set $\Omega$, $\sA$ be a $\sigma$-algebra generated by $\Omega$, $\gI \in \mathbb{S}$ and $\mu = \sum_{\alpha} \{ \mu_+^{(\alpha)} - \mu_-^{(\alpha)} \}e_{\alpha}$ be the Jordan decomposition for $
\mu$. Then for any $\sA$-measurable function $f: \Omega \to \RR$, we shall define
\begin{equation}
\label{eq:realLEFTINTEGRAL}
\int_{\Omega} f(\lambda) \, d\mu(\lambda) := \sum_{\alpha} \left( \int_{\Omega} f(\lambda) \, d\mu_+^{(\alpha)} - \int_{\Omega} f(\lambda)\, d\mu_-^{(\alpha)} \right) e_{\alpha}
\end{equation}
and for any $\sA$-measurable function $f: \Omega \to \CC_{\gI}$, we shall define
\begin{align}
\int_{\Omega} f(\lambda) \, d\mu(\lambda) := & \; \sum_{\alpha} \left( \int_{\Omega} {\rm Re}(f(\lambda)) \, d\mu_+^{(\alpha)} - \int_{\Omega} {\rm Re}(f(\lambda)) \, d\mu_-^{(\alpha)} \right)e_{\alpha} \nonumber \\
+& \;  \sum_{\alpha}  \left( \int_{\Omega} {\rm Im}(f(\lambda)) \, d\mu_+^{(\alpha)} - \int_{\Omega} {\rm Im}(f(\lambda))  d\mu_-^{(\alpha)} \right)\, \gI \, e_{\alpha}, \label{eq:complexLEFTINTEGRAL}
\end{align}
provided that all of the four integrals on the right-hand side exist and we do not end up with the indeterminate expression $\infty - \infty$. Similarly, for any $\sA$-measurable function $f: \Omega \to \CC_{\gI}$, we can define an integral with $f(\lambda)$ on the right via
\begin{align}
\int_{\Omega}  d\mu(\lambda) \, f(\lambda) := & \; \sum_{\alpha} \left( \int_{\Omega} {\rm Re}(f(\lambda)) \, d\mu_+^{(\alpha)} - \int_{\Omega} {\rm Re}(f(\lambda)) \, d\mu_-^{(\alpha)} \right) e_{\alpha} \nonumber \\
+& \; \sum_{\alpha}  \left( \int_{\Omega} {\rm Im}(f(\lambda)) \, d\mu_+^{(\alpha)} - \int_{\Omega} {\rm Im}(f(\lambda)) \, d\mu_-^{(\alpha)} \right) \, \gI \, e_{\alpha}. \label{eq:complexRIGHTINTEGRAL}
\end{align}
\end{defn}

\begin{defn}[measure space and $\mu$-integrability]
Let $\Omega$ be a non-empty set and $\sA$ denote a $\sigma$-algebra on $\Omega$. Suppose $\mu = \sum_{\alpha} \mu^{(\alpha)} e_{\alpha}$ is a $\gP(\RR_n)$-valued measure on $\sA$. We shall call the triple $(\Omega, \sA, \mu)$ a {\it measure space}. We will write that a property holds {\it $\mu$-a.e. on $\Omega$} whenever the desired property holds except on a set $M \in \sA$, where $\mu(M) = 0$. A function $f: \Omega \to \RR \cup \{ \infty \}$ is called {\it measurable} if $\{ \lambda \in \Omega: f(\lambda) \leq t \} \in \sA$ for every $t \in \RR$. We will call a function $f: \Omega \to \CC_{\gI} \cup \{ \infty\}$, where $\gI \in \mathbb{S}$, $\mu$-integrable if $\int_{\Omega} f(\lambda) d\mu^{(\alpha)}(\lambda)$ converge for all $\alpha \in \powerset(\{ 1, \ldots, n \}).$
\end{defn}

\begin{thm}
\label{thm:RR}
Let $X$ be a compact Hausdorff space and $\mathscr{C}(X, \RR)$ denote the normed space of real-valued continuous functions on $X$ together with the supremum norm $\| \cdot \|_{\infty}$. Corresponding to any bounded positive linear functional $\mathscr{L}: \mathscr{C}(X, \RR) \to \RR$, there exists a unique positive Borel measure $\mu$ on $X$ such that
$$\mathscr{L}(p) = \int_X p(t) \, d\mu(t) \quad \quad {\rm for} \quad p \in \mathscr{C}(X, \RR).$$
In this case, $\mu(M)  \leq \| \mathscr{L} \|$ for every set $M$ that belongs to the Borel $\sigma$-algebra generated by $X$, i.e., $\sB(X)$.
\end{thm}

\begin{proof}
The existence and uniqueness of $\mu$ is a special case of Theorem D in Section 56 of \cite{Halmos}. The last assertion follows immediately from the fact that
$$\| \mathscr{L}\| \geq \mathscr{L}(1) = \int_X d\mu(t) = \mu(X) \geq \mu(M) \quad \quad {\rm for} \quad M \in \sB(X).$$
\end{proof}

\begin{cor}
\label{cor:RR}
Let $X$ be a compact Hausdorff space. Corresponding to any bounded linear functional $\mathscr{L}: \mathscr{C}(X, \RR) \to \RR_n$, there exists a unique $\RR_n$-valued Borel measure $\mu$ on $X$ such that
\begin{equation}
\label{eq:INT}
\mathscr{L}(p) = \int_X p(t) \, d\mu(t) \quad \quad {\rm for} \quad p \in \mathscr{C}(X, \RR).
\end{equation}
In this case, $|\mu(M)|  \leq \| \mathscr{L} \|$ for every set $M \in \sB(X)$.
\end{cor}

\begin{proof}
Write $\sL(f) = \sum_{\alpha} \{ \sL_+^{(\alpha)}(f) - \sL_-^{(\alpha)}(f) \} e_{\alpha}$, where $\sL_{\pm}^{(\alpha)}: \sC(X, \RR) \to \RR$ are positive linear functionals for $\alpha \in \powerset( \{1, \ldots, n \})$. It is easy to see that $\sL_{\pm}^{(\alpha)}$ are bounded linear functionals and we may apply Theorem \ref{thm:RR} to obtain the unique positive Borel measures $\mu^{(\alpha)}_{\pm}$ so that
$$\sL_{\pm}^{(\alpha)}(f) = \int_X f(t) d\mu_{\pm}^{(\alpha)}(t) \quad \quad {\rm for} \quad f \in \sC(X, \RR).$$
Thus, if we let $\mu = \sum_{\alpha} \{ \mu_+^{(\alpha)} - \mu_-^{(\alpha)} \}e_{\alpha}$, then we obtain \eqref{eq:INT}.

To see that $\mu$ is unique, suppose there is another $\RR_n$-valued measure $\nu$ on $X$ so that \eqref{eq:INT} holds. Write the Jordan decomposition $\nu = \sum_{\alpha} \{ \nu^{(\alpha)}_+ - \nu^{(\alpha)}_- \}e_{\alpha}$, where $\nu^{(\alpha)}_{\pm}$ are positive Borel measures on $X$.  Then, for any $\alpha \in \powerset(\{ 1, \ldots, n \})$, we must have $\mu_{\pm}^{(\alpha)} = \nu_{\pm}^{(\alpha)}$. Thus, $\mu = \nu$.

The final assertion can be shown in much the same way as the final assertion of Theorem \ref{thm:RR}.
\end{proof}

\begin{defn}[transformation of a $\RR_n$-valued measure]
\label{defn:TRANSR_nMEASURE}
Suppose $\mu$ is a positive $\RR_n$-valued measure on a $\sigma$-algebra of sets $\sA$ generated by a set $\Omega$ and $\psi: \Omega \to \Omega'$. Let $\sA'$ be the family of sets given by $M' \in \sA'$ if $\psi^{-1}(M') \in \sA$. Notice that $\sA'$ is a $\sigma$-algebra on $\Omega'$ and $\mu'(M') := \mu(\psi^{-1}(M'))$ is a positive $\RR_n$-valued measure on $\sA'$.
\end{defn}

\begin{thm}
\label{thm:TRANSSCALARMEAS}
Let $\Omega, \Omega', \sA, \sA', \mu$ and $\mu'$ be as in Definition \ref{defn:TRANSR_nMEASURE}. Suppose $f: \Omega' \to \CC_{\gI}$, where $\gI \in \mathbb{S}$, is a $\mu'$-a.e. finite $\Omega'$-measurable function.  Then $f \circ \psi$ is a $\mu$-a.e. finite $\sA$-measurable function on $\Omega$ and
\begin{equation}
\label{eq:TRANSSCALAR}
\int_{\Omega'} f(\lambda') d\mu'(\lambda') = \int_{\Omega} f(\psi(\lambda)) d\mu(\lambda).
\end{equation}
\end{thm}

\begin{proof}
Let $\mu' = \sum_{\alpha} \{  \nu_+^{(\alpha)} - \nu_-^{(\alpha)} \} e_{\alpha}$ and $\mu = \sum_{\alpha} \{ \mu_+^{(\alpha)} - \mu_-^{(\alpha)} \}e_{\alpha}$ be Jordan decompositions (see \ref{eq:JORDAN}) for $\mu'$ and $\mu$, respectively. Applying the classical result to $\nu_{\pm}^{(\alpha)}$ and $\mu_{\pm}^{(\alpha)}$ (which are just positive measures in the usual sense), see, e.g., Theorem C in Section 39 of \cite{Halmos}, we obtain both assertions.
\end{proof}

\begin{thm}
\label{thm:CliffordRRPOS}
Let $X$ be a compact Hausdorff space. Corresponding to any bounded positive linear functional $\mathscr{L}: \mathscr{C}(X, \RR) \to \gS(\RR_n)$, there exists a unique positive $\RR_n$-valued Borel measure $\mu$ on $X$ such that
\begin{equation}
\label{eq:INTpos}
\mathscr{L}(p) = \int_X p(t) \, d\mu(t) \quad \quad {\rm for} \quad p \in \mathscr{C}(X, \RR).
\end{equation}
In this case, $|\mu(M)|  \leq \| \mathscr{L} \|$ for every set $M \in \sB(X)$.
\end{thm}

\begin{proof}
First, let $K \geq 0$ be such that
$$| \mathscr{L}(p) | \leq K \| p \|_{\infty} \quad \quad {\rm for} \; {\rm every} \; p \in \mathscr{C}(X, \RR).$$
Next, note that for every $p \in \mathscr{C}(X, \RR)$, we may uniquely decompose $\mathscr{L}(p)$ as $\mathscr{L}(p) = \sum_{\alpha} \mathscr{L}_{\alpha}(p) e_{\alpha}$, where $\mathscr{L}_{\alpha}(p) \in \RR$. Let $L(p): \mathscr{C}(X, \RR)  \to \chi(\RR_n)$ given by $L(p) = \chi(\mathscr{L}(p))$ for $p \in \mathscr{C}(X, \RR)$.  Since $\chi$ is a $*$-homomorphism and $\mathscr{L}: \mathscr{C}(X, RR) \to \gS(\RR_n)$, we have that $L(p)$ is a real Hermitian matrix for every $p \in \mathscr{C}(X,\RR)$. Moreover, in view of the first assertion in Remark \ref{rem:chiProp} and the estimate
\begin{align*}
\|L(p)\|_{\infty} =& \; \| \chi(\mathscr{L}(p)) \| =  \sum_{\alpha} |\mathscr{L}_{\alpha}(p)| \\
\leq& \; 2^{n/2} (\sum_{\alpha} \mathscr{L}_{\alpha}(p)^2)^{1/2} \\
=& \; 2^{n/2} | \mathscr{L}(p)| \\
\leq & \; 2^{n/2} K \| p \|_{\infty}
\end{align*}
where $\| \cdot \|_{\infty}$ denotes the maximum row sum norm of a matrix,
we have that $L(p)$ is a bounded positive $\RR^{2^n \times 2^n}$-valued linear functional. Consequently, a finite dimensional version of the operator-valued version of the Riesz representation theorem (see, e.g., Theorem 19 in \cite{Berberian}) gives us the existence of a uniquely determined positive $\RR^{2^n \times 2^n}$-valued measure $\nu := (\nu_{ik})_{i,j=1}^{2^n}$ such that
$$L(p) = \int_X p(t) \, d\nu(t) := \left(\int_X p(t) d\nu_{ij}(t) \right)_{i,j=1}^{2^n} \quad \quad {\rm for} \quad p \in  \mathscr{C}(X, \RR).$$

If we write $L(p) = (L_{ij}(p))_{i,j=1}^{2^n}$, then $L_{ij}: \mathscr{C}(X, \RR) \to \RR$ is a bounded linear functional such that
\begin{equation}
\label{eq:REP}
L_{ij}(p) = \int_X p(t) d\nu_{ij}(t) \quad \quad {\rm for} \quad p \in  \mathscr{C}(X, \RR).
\end{equation}
Consequently, Corollary \ref{cor:RR}, with $n =0$, asserts that $\nu_{ij}$ is the only $\RR$-valued Borel measure such that \eqref{eq:REP} holds. If we use the fact that $L(p) = \chi(\mathscr{L}(p))$ and the aforementioned uniqueness of the $\RR$-valued measure in \eqref{eq:REP}, then for any $M \in \sB(X)$, we  have that $\nu(M) = \chi(a_M)$ for some $a_M \in \RR_n$. In view of Remark \ref{rem:chiProp}, the fact that $\nu(M)$ is a positive semidefinite matrix implies and $\nu(M) = \chi(a_M)$ implies that $a_M \in \gP(\RR_n)$ for all $M \in \sB(X)$. It is easy to check that $\mu(M) := a_M= \chi^{-1}(\nu(M))$ is a positive $\RR_n$-valued measure with the property that
$$\mathscr{L}(p) = \chi^{-1}(L(p)) = \int_X p(t) \, d\mu(t) \quad \quad {\rm for} \quad p \in \mathscr{C}(X, \RR),$$
i.e., \eqref{eq:INTpos} holds.

The uniqueness of $\mu$ such that \eqref{eq:INTpos} holds follows immediately from the uniqueness of $\nu$ and the injectivity of $\chi$. The final assertion can be proved in much the same way as the proof of the final assertion in Theorem \ref{thm:RR}.
\end{proof}

\section{Spectral integrals}
\label{sec:SIs}

Throughout this section, $\sA$ will denote an algebra of subsets of $\Omega$, $\cH_n$ will denote a Clifford module over $\RR_n$ and all infinite sums of operators in $\cB(\cH_n)$ will be meant in the strong operator topology, i.e., $\sum_{k=1}^{\infty} T_k = T$ if and only if
$$Tx = \lim_{k \to \infty} \sum_{j=1}^k T_j x \quad \quad {\rm for} \quad x \in \cH_n.$$

\subsection{Basic properties}
\label{sec:BASIC}
\begin{defn}[orthogonal projection]
An operator $T \in \cB(\cH_n)$ is called an {\it orthogonal projection} if $T$ is self-adjoint and $T^2 = T$.  The set of orthogonal projections on $\cH_n$ will be denoted by $\cP(\cH_n)$.
\end{defn}

\begin{defn}[spectral premeasure and spectral measure]
\label{def:SPECTRALMEASURE}
Let $\sA$ be an algebra of subsets in $\Omega$. We will call $E: \sA \to \cP(\cH_n)$ a {\it spectral premeasure} if the following conditions hold{\rm :}
\begin{enumerate}
\item[(i)] $E$ is countably additive, i.e.,
$$E\left( \bigcup_{i=1}^{\infty} M_i \right) = \sum_{i=1}^{\infty} E(M_i)$$
for every sequence of mutually disjoint sets $(M_i)_{i=1}^{\infty}$ such that $M_i \in \Omega$ for $i=1,2,\ldots$ and $\cup_{i=1}^{\infty} M_i \in \Omega$.
\smallskip
\item[(ii)] $E(\Omega) = I$.
\end{enumerate}
If $\sA$ is a $\sigma$-algebra, then $E$ will be called a {\it spectral measure}. In this case, we will write that $E$ is a spectral premeasure (resp., measure) on $(\Omega, \sA)$.
\end{defn}

\begin{lem}
\label{lem:DISJOINT}
Let $\sA$ be an algebra of sets in $\Omega$. Suppose $E: \sA \to \cP(\cH_n)$ is a finitely additive map, i.e.,
$$E\left( \bigcup_{i=1}^{k} M_i \right) = \sum_{i=1}^{k} E(M_i)$$
for every collection of mutually disjoint sets $(M_n)_{i=1}^k$ such that $M_i \in \Omega$ for $i=1,\ldots,k$ and $\bigcup_{i=1}^k M_i \in \Omega$. Then
\begin{equation}
E(M) E(N) = E(M \cap N) \quad \quad {\rm for} \quad M, N \in \sA.
\end{equation}
In particular, we have $E(M) E(N) = 0$ if $M \cap N = \emptyset$.
\end{lem}

\begin{proof}
The proof can be carried out in the same way as the complex Hilbert space case (see, e.g., Lemma 4.3 in \cite{Schmuedgen}).
\end{proof}

\begin{lem}
\label{lem:2Sept}
Suppose $\sA$ is an algebra of sets in $\Omega$. The mapping
$E: \sA \to \cP(\cH_n)$ is a spectral premeasure if and only if the following conditions hold{\rm :}
\begin{enumerate}
\item[(i)] For any $x \in \cH_n$, the set function $E_x(\cdot)$ given by
$$E_x(M) :=  \langle E(M)x, x \rangle,$$
where $M \in \sA$, is a countably additive $\gP(\RR_n)$-valued set function.
\smallskip
\item[(ii)] $E(\Omega) = I$.
\end{enumerate}
\end{lem}

\begin{proof}
Suppose $E$ is a spectral premeasure. Then, by Definition \ref{def:SPECTRALMEASURE}, (ii) holds. It follows from Lemma \ref{lem:DISJOINT} with $M=N$ that $E(M)^2 x = E(M)x$ for $x \in \cH_n$ and $M \in \sA$, we have
$$
E_x(M)
=  \langle E(M)x, E(M)x \rangle \in \gP(\RR_n)  \quad \text{for $M \in \sA$ and $x \in \cH_n$.}
$$
Thus, (i) holds.

Conversely, suppose (i) and (ii) hold. Suppose $(M_i)_{i=1}^{\infty}$ be a sequence of mutually disjoint sets such that $M_i \in \sA$ for $i=1,2,\ldots$ and hence $M := \cup_{i=1}^{\infty} M_n \in \sA$, since $\sA$ is a $\sigma$-algebra. For any $x\in \cH_n$, we note that $E_x$ is finitely additive and hence $E$ is finitely additive as well. Thus, we may use Lemma \ref{lem:DISJOINT} to deduce that $(E(M_i))_{i=1}^{\infty}$ is a sequence of orthogonal projections such that $E(M_j)E(M_i) = 0$ whenever $j \neq n$. Consequently, $\sum_{i=1}^{\infty} E(M_i)$ has a limit in the strong operator topology of $\cB(\cH_n)$. Since
\begin{align*}
E_x(M) =& \; \sum_{i=1}^{\infty} E_x(M_i) \\
=& \; \sum_{i=1}^{\infty} \langle E(M_i)x, x \rangle \\
=& \; \left \langle \sum_{i=1}^{\infty} E(M_i)x, x \right \rangle \quad \quad {\rm for} \quad x \in \cH_n,
\end{align*}
we may use the polarisation formula \eqref{eq:POLAR} to obtain
$$\langle E(M)x, y \rangle = \langle \sum_{i=1}^{\infty} E(M_i)x, y \rangle \quad \quad {\rm for} \quad x, y \in \cH_n,$$
i.e., $E(M) = \sum_{i=1}^{\infty} E(M_i)$. Thus, $E$ is a spectral premeasure.
\end{proof}

\begin{rem}
\label{rem:GENERAL}
In the event that $\sA$ is a $\sigma$-algebra, a careful inspection of the proof of Lemma \ref{lem:2Sept} shows that we can easily adapt the proof to obtain the following result where spectral premeasure and countably additive $\gP(\RR_n)$-valued set function are replaced by spectral measure and $\gP(\RR_n)$-valued measure, respectively.
\end{rem}

\begin{lem}
\label{lem:2SeptMEASURE}
Suppose $\sA$ is a $\sigma$-algebra of sets in $\Omega$. The mapping $E: \sA \to \cP(\cH_n)$ is a spectral measure if and only if the following conditions hold{\rm :}
\begin{enumerate}
\item[(i)] For any $x \in \cH_n$, the set function $E_x(\cdot)$ given by\\ $E_x(M) :=  \langle E(M)x, x \rangle$, where $M \in \sA$, is a $\gP(\RR_n)$-valued measure.
\smallskip
\item[(ii)] $E(\Omega) = I$.
\end{enumerate}
Moreover, $E_x(\cdot)$ is a finite $\gP(\RR_n)$-valued measure.
\end{lem}

\begin{proof}
See Remark \ref{rem:GENERAL} for the characterisation of spectral measures. The second assertion is a direct consequence of $E(\Omega) = I$. Indeed, $E_x(\Omega) = \langle E(\Omega) x , x \rangle = \langle x, x \rangle < \infty$.
\end{proof}

\begin{rem}
\label{rem:Exy}
Fix $x, y \in \cH_n$ and let $E$ be a spectral measure on $\sA$. Then consider the $\RR_n$-valued measure $E_{x,y}(M) := \langle E(M)x, y \rangle$ for $M \in \sA$. One can use the polarisation formula \eqref{eq:POLAR} to verify
\begin{equation}
\label{eq:Exy}
E_{x,y} = \frac{    \sum_{\alpha} ( E_{x+ye_\alpha} - E_{x-ye_\alpha} ) e_\alpha  }{ 4 \dim \gS(\RR_n)}.
\end{equation}
\end{rem}

It turns out that the $\RR_n$-valued measure $E_{x,y}$ has the properties detailed in the following lemma.

\begin{lem}
\label{lem:MONOTONE}
Let $\sA$ be an algebra of sets in $\Omega$. Suppose $E$ is a spectral premeasure on $\sA$. Then for any sequences of sets $(M_i)_{i=1}^{\infty}$ and $(N_i)_{i=1}^{\infty}$ in $\sA$ such that $M_{i+1} \subseteq M_i$, $N_{i} \subseteq N_{i+1}$ for $i=1,2,\ldots$ and $M := \cap_{i=1}^{\infty} M_i \in \sA$, we have
\begin{equation}
E(M) = s-\lim_{i \to \infty} E(M_i) \quad {\rm and} \quad E(N) = s-\lim_{i \to \infty} E(N_i),
\end{equation}
where $N := \cup_{i=1}^{\infty} N_i$.
\end{lem}

\begin{proof}
The proof is very straight forward.
\end{proof}

We will now define the support of a spectral measure. We shall suppose $\Omega$ is a Hausdorff topological space which has a countable base of open sets.  In what follows, $\sB(\Omega)$ will denote the Borel $\sigma$-algebra of sets generated by $\Omega$, i.e., the smallest $\sigma$-algebra which contains all open sets in $\Omega$. For the sake of brevity, we shall sometimes write that $E$ is a spectral measure on $\sB(\Omega)$ in place of $E$ is a spectral measure on $(\Omega, \sB(\Omega))$.

\begin{defn}[support of a spectral measure]
\label{def:SUPPORT}
Let $E$ be a spectral measure on $\mathscr{B}(\Omega)$. We shall define the support of $E$ by
$$\supp E := \Omega \setminus \bigcup_{\stackrel{\text{$M \in \sB(\Omega)$ open}}{E(M)=0}} M.$$
\end{defn}

The following lemma will be important when proving the bounded case for the spectral theorem for normal operators on a Clifford module.

\begin{lem}
\label{lem:IMPORTANT}
Let $E$ be a spectral measure on $(\Omega, \sB(\Omega))$. Then
\begin{equation}
\label{eq:SPRTALT}
\supp E = \{ \lambda \in \Omega:\text{$E(M) \neq 0$ for every open set $M \in \sB(\Omega)$ with $\lambda \in M$}\}.
\end{equation}
\end{lem}

\begin{proof}
The characterisation \eqref{eq:SPRTALT} is a straight forward consequence of Definition \ref{def:SUPPORT}.
\end{proof}

Let $\sA$ be an algebra of sets in $\Omega$. For any $\RR_n$-valued measure $\nu$ on $\Omega$, we shall let the {\it total variation} of $\nu$ be given by
\begin{equation}
\label{eq:TOTALVAR}
|\nu| := \sup_{(M_i)_{i=1}^k} |\nu(M_i)|,
\end{equation}
where $(M_i)_{i=1}^k$ is any sequence of mutually disjoint sets with $M_i \in\sA$ for $i=1,\ldots,k$ and $k \in \NN$ is arbitrary.

\begin{defn}[An $L_2$-space with respect to $\RR_n$-valued measure]
\label{def:L_2}
Let $\sA$ be an algebra of sets in $\Omega$, $\gI \in \mathbb{S}$ and $\nu$ be a $\RR_n$-valued measure on $\sA$. Then we shall let
$$
L_2(\Omega, \sA, \CC_{\gI}, \nu) := \{ \text{$\sA$-measurable $f: \Omega \to \CC_{\gI}$}: \int_{\Omega} |f(\lambda)|^2 d|\nu|(\lambda) < \infty \}.
$$
For $x \in \cH_n$ and a spectral measure $E$ on $\sB(\Omega)$ we let $\mu_x(M) = \langle E(M)x, x \rangle$ for $M \in \sB(\Omega)$ and
$$
L_2(\Omega, \sA, \CC_{\gI}, E_x) := \{ \text{$\sA$-measurable $f: \Omega \to \CC_{\gI}$}: \int_{\Omega} |f(\lambda)|^2 d|\mu_x|(\lambda) < \infty \}.
$$
\end{defn}

\begin{rem}
\label{rem:L_2}
Let $\Omega, \sA, \gI$, $E$, $x$, $E_x$ and $\nu$ be as in Definition \ref{def:L_2}. Write the Jordan decomposition $\nu = \sum_{\alpha} = \{ \nu_+^{(\alpha)} - \nu_-^{(\alpha)} \} e_{\alpha}$ for $\nu$. Then it is easy to see that
\begin{equation}
\label{eq:L2all}
L_2(\Omega, \sA,\CC_{\gI}, \nu)
= \bigcap_{\alpha \in \powerset(\{ 1, \ldots, n \})} \{ L_2(\Omega, \sA,\CC_{\gI}, \nu_{+}^{(\alpha)}) \cap L_2(\Omega, \sA,\CC_{\gI},\nu_-^{(\alpha)})\},
\end{equation}
and
\begin{align}
L_2(\Omega, \sA, \CC_{\gI}, E_x) = \{ \text{$\sA$-measurable $f: \Omega \to \CC_{\gI}$}: & \; \int_{\Omega} |f(\lambda)|^2 d( {\rm Re}\mu_x(\lambda))\nonumber \\
<& \; \infty \}, \label{eq:L2E_x}
\end{align}
where $\mu_x(M) := \langle E(M)x, x \rangle$ for $M \in \sA$.
\end{rem}

\begin{lem}
\label{lem:MEASURECOMP}
Let $E'$ be a spectral premeasure on $(\Omega, \sA')$. Then there exists a spectral measure $E$ on $(\Omega, \sA)$, where $\sA$ is the $\sigma$-algebra generated by $\sA'$, such that $E(M) = E'(M)$ for all $M \in \sA'$.
\end{lem}

\begin{proof}
For any $x \in \cH_n$, we may use Lemma \ref{lem:2Sept} to see that $\mu'_x(\cdot)$ given by $\mu'_x(M) = \langle E(M)x, x \rangle$ is a premeasure on $\sA'$. Moreover, using Lemma \ref{lem:KAPPA}, we obtain
\begin{align*}
|\mu'_x(\Omega)| =& \; |\langle \mu'(\Omega) x, x \rangle| = |\langle x, x \rangle| \\
\leq& \; \| x \|^2.
\end{align*}
Thus, $\mu'_x$ is a finite $\RR_n$-valued premeasure on $\sA'$ and hence all of the positive measure components in the Jordan decomposition of
$$
\mu'_x = \sum_{\alpha} (\mu'_{x,\alpha,+} - \mu'_{x,\alpha,-})e_{\alpha},
$$
namely $(\mu_{x, \alpha, \pm})_{\alpha \in \powerset( \{ 1, \ldots, n \} )}$, are also finite positive measures. Thus, we may appeal to Theorem A in \S 13 of \cite{Halmos} to produce uniquely determined measures $(\mu_{x,\alpha,+})_{\alpha \in \powerset(\{ 1, \ldots, n\})}$ on the $\sigma$-algebra $\sA$. Consequently, $\mu_x := \sum_{\alpha} (\mu_{x,\alpha,+}- \mu_{x,\alpha,-})e_{\alpha}$ is the unique $\gS(\RR_n)$-valued measure on $\sA$ such that $\mu_x(M) = \mu'_x(M)$ for $M \in \sA'$.

Let $x, y \in \cH_n$ and $\mu_{x,y}$ be the $\RR_n$-valued measure given by
\begin{equation}
\label{eq:MUxy}
\mu_{x,y}(M) := \frac{ \sum_{\alpha} (  \mu_{x+ye_{\alpha}} - \mu_{x-ye_{\alpha}})e_{\alpha}       }{4 \dim \gS(\RR_n) }  \quad\quad {\rm for} \quad M \in \sA.
\end{equation}
Next, let
$$\tilde{\sA} := \{ M \in \sA: \text{the map $x \mapsto \mu_{x,y}(M)$ is right linear for all $y \in \cH_n$} \}.$$
One can use \eqref{eq:MUxy} and Lemma \ref{lem:MONOTONE} to show that $\tilde{\sA}$ has the following property. The union $\cup_{m=1}^{\infty} M_i \in \tilde{\sA}$ for every sequence of sets $(M_i)_{m=1}^{\infty}$ such that $M_i \in \tilde{\sA}$ for $i=1,2,\ldots$ with $M_{i} \subseteq M_{i+1}$. Thus, as $\sA' \subseteq \tilde{\sA}$, we may use Theorem B in \S 6 of \cite{Halmos} to deduce that $\sA \subseteq \tilde{\sA}$ and hence $\tilde{\sA} = \sA$.

In much the same way as above, one can show that the map $y \mapsto \mu_{x,y}(M)$ is anti-right linear, i.e., $\mu_{x,ya+z}(M) = \bar{a} \mu_{x,y}(M)+\mu_{x,z}(M)$ for all $M \in \sA$ and $x,z \in \cH_n$. It is easy to check that
$$\mu'_{x+xe_{\alpha} }(M) - \mu'_{x-xe_{\alpha}}(M) = 2 ( \mu'_{x,x}(M)e_{\alpha} - e_{\alpha}\mu'_{x,x}(M))$$
and hence by the uniqueness of $\mu$ we have
$$\mu_{x+xe_{\alpha} }(M) - \mu_{x-xe_{\alpha}}(M) = 2 ( \mu_{x}
(M)e_{\alpha} - e_{\alpha}\mu_{x}(M)).$$
Consequently,
$$\mu_{x,x}(M) = 2 \sum_{\alpha} (\mu_{x}(M) - e_{\alpha} \mu_{x}(M) e_{\alpha}),$$
in which case $\mu_{x,x}(M) \in \gP(\RR_n)$ and ${\rm Re}\, \mu_{x,x}(M) \geq 0$ for all $M \in \sA$ and $x \in \cH_N$.
For any $M \in \sA$, we may use the polarisation formula \eqref{eq:POLAR} to check that $\langle E'(M)x, y \rangle = \mu'_{x,y}(M)$ and  Lemma \ref{lem:KAPPA}(ii) to deduce
\begin{align*}
|\mu_{x,y}(M)| \leq& \; | \mu_{x,y}(\Omega) | \\
=& \; | \mu'_{x,y}(\Omega) | \\
=& \; | \langle E'(\Omega)x, y \rangle | = | \langle x, y \rangle \\
=& \; \| x \|^2 \| y \|^2 \quad \quad {\rm for} \quad x, y \in \cH_n.
\end{align*}
Thus, for any $M \in\sA$, we may invoke Theorem \ref{thm:RRforHM} to obtain a positive operator $E(M)  \in\cB(\cH_n)$ such that $\langle E(M)x, y \rangle = \mu_{x,y}(M)$ for $x,y \in \cH_n$. For $M \in \sA'$, we have
$$E'(M)x, x \rangle = \langle E(M)x, x \rangle \quad \quad {\rm for} \quad  x \in \cH_n$$
and, hence, one can use the polarisation formula \eqref{eq:POLAR} to check that
$E'(M) = E(M)$ for $M \in \sA'$.

Finally, let $\mathscr{P} := \{ M \in \sA: E(M) \in \cP(\cH_n) \}$. It is very easy to see that $\mathscr{P}$ has the monotone property and hence $\sA' \subseteq \mathscr{P}$. But then $\sA \subseteq \mathscr{P}$. Thus, we may use Lemma \ref{lem:2SeptMEASURE} to deduce that $E$ is a spectral measure.
\end{proof}

\begin{lem}
\label{lem:REGULARMEASURE}
Suppose $\Omega$ is a locally compact Hausdorff space which has a countable base of open sets and $E$ is a spectral measure on $ \sB(\Omega)$. For any $x \in \cH_n$, let $\mu_x(M) := \langle E(M)x, x \rangle$ for $M \in \sB(\Omega)$ and let $\mu_x = \sum_{\alpha} \{ \mu_{x,+}^{(\alpha)} - \mu_{x,-}^{(\alpha)} \} e_{\alpha}$ be the Jordan decomposition for $\mu_x$. Then
\begin{equation}
\label{eq:REGULAR}
\text{$\mu_{x,+}^{(\alpha)}$ is a  finite positive Borel measure for every $\alpha \in \powerset( \{1,\ldots, n\})$.}
\end{equation}
\end{lem}

\begin{proof}
The last assertion of Lemma \ref{lem:2SeptMEASURE} ensures that $\mu_x$ is a finite $\gP(\RR_n)$-valued Borel measure. Consequently, $\mu_{x, +}^{(\alpha)}$ is a finite positive Borel measure for all $\alpha \in \powerset( \{1, \ldots, n \})$. But then we may use the well-known fact that every positive Borel measure on a locally compact Hausdorff space which has a countable base of open sets is regular (see, e.g., Proposition 7.2.3 in \cite{Cohn}) to obtain \eqref{eq:REGULAR}.
\end{proof}

\begin{thm}
\label{thm:COMMUTING}
Suppose $\Omega_i$ is a locally compact Hausdorff space which has a countable base of open sets and $E_i: \sB(\Omega_i) \to \cP(\cH_n)$ is a spectral measure on $\sB(\Omega_i)$ for $i=1,\ldots, d$. If $E_i E_j = E_j E_i$ for $i,j=1,\ldots, d$, then there exists a unique spectral measure $E: \sB(\Omega) \to \cP(\cH_n)$, where $\Omega := \Omega_1 \times \cdots \times \Omega_d$, such that
\begin{equation}
\label{eq:PRODUCTMEASURE}
E(M_1 \times \cdots \times M_d) = E(M_1) \cdots E(M_d) \quad {\rm for} \quad M_i \in \sB(\Omega_i) \; {\rm and} \; i=1,\ldots,d.
\end{equation}
\end{thm}

\begin{proof}
Let us consider the case $d=2$ (the more general case follows in much the same way).  Let $\sA'$ denote the algebra of sets generated by sets of the form $M_1 \times M_2$, where $M_i \in \sB(\Omega_i)$ for $i=1,2$. Thus, every $N \in \sA'$ can be written as $N = \cup_{j=1}^k N_j$, where $N_1, \ldots N_k$ are mutually disjoint and are of the form
$$N_j = M_{1j} \times M_{2j} \in \sA' \quad \quad {\rm for} \quad j=1,\ldots,k.$$
We shall now define $E(N) := \sum_{j=1}^k E_1(M_{1j}) E_2(M_{2j})$.

We claim that $E(N) \in \cP(\cH_n)$. To this end, note that if $i \neq j$, then $N_i \cap N_j = \emptyset$ forces $M_{1i} \cap M_{1j} = \emptyset$ or $M_{2i} \cap M_{2j} = \emptyset$. In either case, we can make use of Lemma \ref{lem:DISJOINT} to deduce
\begin{align*}
0 =& \; E_1(M_{1i} \cap M_{1j}) E_2(M_{2i} \cap M_{2j}) \\
=& \; E_1(M_{1i}) E_1(M_{1j}) E_2(M_{2i}) E_2(M_{2j}) \\
=& \;  E_1(M_{1i}) E_2(M_{2i}) E_1(M_{1j}) E_2(M_{2j}).
\end{align*}
Consequently, $E(N) \in \cP(\cH_n)$. One can easily modify the proof of Theorem E in \S 8 of \cite{Halmos} to show that $E$ is independent of the representation that we choose for $N \in \cB(\Omega)$.

What remains is to show that $E$ is countably additive, i.e., for any monotone sequence of sets $(N_i)_{i=1}^{\infty}$, where $N_i \in \sB(\Omega')$, with $N := \cup_{i=1}^{\infty} N_i$, we must show that $E(N) = \sum_{i=1}^{\infty} E(N_i)$. In view of the fact that ${\rm Re} \langle E(N')x, x \rangle$ is a positive measure for any choice of $N' \in \sB(\Omega')$ and $x \in \cH_n$, we can prove
$${\rm Re} \langle E(N)x, x \rangle \leq \sum_{i=1}^{\infty} {\rm Re} \langle E(N_i x, x \rangle$$
in much the same way as the proof of Theorem 4.10 in \cite{Schmuedgen} with the caveat that terms of the form $\langle E(N')x, x \rangle$ must be replaced by ${\rm Re} \langle E(N')x, x \rangle$. Since $N_i \subseteq N$, we immediately have $\sum_{i=1}^k {\rm Re}\langle E(N_i)x, x \rangle \leq {\rm Re} \langle E(N)x, x \rangle$ and hence we arrive at
\begin{equation}
\label{eq:ALL}
{\rm Re}\, \langle E(N)x, x \rangle = \sum_{i=1}^{\infty} {\rm Re}\, \langle E(N_i)x, x \rangle.
\end{equation}
Since $E \in \cP(\cH_n)$, we have
$$\| E(N') x \|^2 = {\rm Re}\, \langle E(N')x, E(N') x \rangle = {\rm Re}\, \langle E(N')x, x \rangle \quad \quad {\rm for} \quad N' \in \sB(\Omega)$$
and thus we may use \eqref{eq:ALL} and the fact that all sums are taken with respect to the strong operator topology to obtain $E(N) = \sum_{i=1}^{\infty} E(N_i)$, i.e., $E$ is countably additive.

Finally, since the $\sigma$-algebra generated by $\sA'$ is $\sB(\Omega)$, we may use Lemma \ref{lem:MEASURECOMP} to deduce that the spectral premeasure $E$ has an extension to a spectral measure. With a slight abuse of notation, the aforementioned spectral measure on $\sB(\Omega)$ will be denoted by $E$.  One can use the definition of $E$ to verify \eqref{eq:PRODUCTMEASURE}. The uniqueness of a spectral $E$ which obeys \eqref{eq:PRODUCTMEASURE} drops out immediately from the fact that the $\sigma$-algebra generated by the algebra of sets of the form $M_1 \times M_2$, where $M_i \in \sB(\Omega_i)$ for $i=1,2$, is $\sB(\Omega)$.

\end{proof}

\subsection{Spectral integrals of bounded measurable functions}
\label{sec:SIsbdd}

We anticipate the definition of imaginary operator
given in Definition \ref{def:IMAGINARY}.
We will call an operator $J_0 \in \cB(\cH_n)$ a {\it partial imaginary operator} if $J_0$ is a partial isometry and $J_0^* = - J_0$. We will call $J \in \cB(\cH_n)$ an {\it imaginary operator} if $J$ is unitary and $J^* = -J$.

\begin{defn}
\label{def:ASSOCIATED}
Given a spectral measure $E$ on $(\Omega, \sA)$ and an imaginary operator $J \in \cB(\cH_n)$, we will say that {\it $J$ is associated with the spectral measure $E$} if
\begin{equation}
\label{eq:BLAH}
JE(M) = E(M)J \quad \quad {\rm for} \quad M \in \sA.
\end{equation}
\end{defn}

Given a spectral measure $E$ on $(\Omega, \sA)$ as above and an imaginary $J$ associated with $E$, we wish to give meaning to the integral of a measurable function $f: \Omega \to \CC_{\gI}$, where $\gI \in \mathbb{S}$, against a spectral measure $E$. Let $\mathfrak{B}(\Omega, \sA, \CC_{\gI})$ denote the Banach space of all bounded $\sA$-measurable functions $f: \Omega \to \CC_{\gI}$, where $\gI \in \mathbb{S}$, equipped with the norm
$$\| f \|_{\infty} = \sup \{|f(\lambda)|: \lambda \in \Omega\}.$$
We let $\chi_{M}$ denote the characteristic function with respect to $M \in \Omega$, i.e.,
$$\chi_M(\lambda) = \begin{cases} 1 & {\rm if} \quad \lambda \in M \\
0 & {\rm if} \quad \lambda \notin M. \end{cases}$$
Let $\mathfrak{B}_s(\Omega, \sA, \CC_{\gI})$ denote the subspace of simple functions in \\$\mathfrak{B}(\Omega, \sA, \CC_{\gI})$, i.e., the subspace of functions $f \in \mathfrak{B}(\Omega, \sA, \CC_{\gI})$ which are of the form
\begin{equation}
\label{eq:SIMPDEF}
f(\lambda) = \sum_{j=1}^k c_j \, \chi_{M_j}(\lambda),
\end{equation}
where $c_1, \ldots, c_k \in \CC_{\gI}$ and $M_1, \ldots, M_k$ are pairwise disjoint sets belonging to $\mathfrak{B}(\Omega, \sA, \CC_{\gI})$.

Given $f \in \mathfrak{B}_s(\Omega, \sA, \CC_{\gI})$, we shall let
\begin{equation}
\label{eq:SIMPLE}
\II(f) := \sum_{j=1}^k \{ {\rm Re}(c_j) E(M_i) + {\rm Im}(c_j) E(M_j)J \} \in \cB(\cH_n).
\end{equation}
It can be easily checked that the finite additivity of $E$ implies that the definition of $\II(f)$ in \eqref{eq:SIMPLE} is independent of the \eqref{eq:SIMPDEF}.

\begin{lem}
\label{lem:BOUNDED}
Let $f \in \mathfrak{B}_s(\Omega, \sA, \CC_{\gI})$. Then $\| \II(f) \| \leq \| f \|_{\infty}$.
\end{lem}

\begin{proof}
Since $M_1 \ldots, M_k$ are mutually disjoint sets belonging to $\sB(\Omega)$, Lemma \ref{lem:DISJOINT} asserts that $E(M_i)E(M_j) = 0$ whenever $i \neq j$. Thus, for any $x \in \cH_n$, we have
\begin{align*}
\| \II(f)x \|^2 =& \; \sum_{i,j=1}^k {\rm Re} \left \langle \{ {\rm Re}(c_i)I  + {\rm Im}(c_i) J \} E(M_i) x, \{ {\rm Re}(c_j)I + {\rm Im}(c_j) J \} E(M_j) x  \right\rangle \\
=& \; \sum_{i=1}^k {\rm Re} \left \langle \{ {\rm Re}(c_i)I + {\rm Im}(c_i) J \} E(M_i) x, \{ {\rm Re}(c_i) + {\rm Im}(c_i) J \} x  \right\rangle \\
=& \; \sum_{i=1}^k {\rm Re} \left \langle \{ {\rm Re}(c_i) I+ {\rm Im}(c_i) J \} E(M_i) x, \{ {\rm Re}(c_i)I + {\rm Im}(c_i) J \} E(M_i) x  \right\rangle \\
=& \; \sum_{i=1}^k \| ({\rm Re}(c_i)I + {\rm Im}(c_i)J )E(M_i) x \|^2 \\
\leq& \; \sum_{i=1}^k |c_i|^2 \| E(M_i) x \|^2 \\
\leq& \; \sum_{i=1}^k \| f \|^2_{\infty} \| E(M_i)x \|^2 = \| f \|_{\infty}^2 \left \| \sum_{i=1}^k E(M_i) x \right\|^2 \\
\leq& \; \| f \|_{\infty}^2 \| x\|^2.
\end{align*}
Consequently, $\| \II(f) \| \leq \| f \|_{\infty}$ holds.
\end{proof}

\begin{defn}[$\II(f)$ for $f \in \mathfrak{B}(\Omega, \sA, \CC_{\gI})$]
\label{def:BFUNC}
We will now give a definition for $\II(f)$ when $f \in \mathfrak{B}(\Omega, \sA, \CC_{\gI})$.  Since $\mathfrak{B}_s(\Omega, \sA, \CC_{\gI})$ is a dense subset of the Banach space $\mathfrak{B}(\Omega, \sA, \CC_{\gI})$ with respect to the supremum norm $\| \cdot \|_{\infty}$, given any $f \in \mathfrak{B}(\Omega, \sA, \CC_{\gI})$, we have the existence of a Cauchy sequence $(f_j)_{j=1}^{\infty}$, where $f_j \in \mathfrak{B}_s(\Omega, \sA, \CC_{\gI})$ for $j=1,2\ldots$, such that
$$\lim_{j \to \infty} \| f - f_j \|_{\infty} = 0.$$
Lemma \ref{lem:BOUNDED} can be used to show that $(\II(f_j))_{j=1}^{\infty}$ is a Cauchy sequence of operators belonging to $\cB(\cH_n)$. As $\cB(\cH_n)$ is a complete metric space (see Remark \ref{rem:COMPLETEOPERATORNORM}), we have the existence of an operator $\II(f) \in \cB(\cH_n)$ such that
$$\lim_{j \to \infty} \| \II(f) - \II(f_j) \| = 0.$$
Moreover, Lemma \ref{lem:BOUNDED} can be used to show that the limit $\II(f)$ does not depend on the choice of Cauchy sequence.
\end{defn}

\begin{thm}
\label{thm:BFUNC}
For any $f, g \in \mathfrak{B}(\Omega, \sA, \CC_{\gI})$, we have the following{\rm :}
\begin{enumerate}
\item[(i)] $\II(\bar{f}) = \II(f)^*$.

\smallskip

\item[(ii)] $\II(fg) = \II(f)\II(g)$.

\smallskip

\item[(iii)] $\II(\gI) = J$ and $\II(cf + g) = c\, \II(f) + \II(g)$ for all  $c \in \CC_{\gI}$ and $\gI \in \mathbb{S}$.

\smallskip

\item[(iv)] $\langle \II(f)x, y \rangle = \int_{\Omega} {\rm Re}(f(\lambda)) d\langle E(\lambda)x, y \rangle + \int_{\Omega} {\rm Im}(f(\lambda)) d\langle JE(\lambda)x, y \rangle$ for all $x,y \in \cH_n$.

\smallskip

\item[(v)] $\| \II(f) x \|^2 = \int_{\Omega} | f(\lambda)|^2 d ({\rm Re}\, \langle E(\lambda)x, x \rangle)$ for all $x \in \cH_n$.

\smallskip

\item[(vi)] $\| \II(f) \| \leq \| f \|_{\infty}$.

\smallskip

\item[(vii)] For any sequence of functions $(f_j)_{j=1}^{\infty}$, where $f_j \in \mathfrak{B}(\Omega, \sA, \CC_{\gI})$ for $j=1,2,\ldots$, which converges pointwise $E$-a.e. on $\Omega$ to $f$ and there exists $\kappa > 0$ such that $|f_j(\lambda)| \leq \kappa$ for all $\lambda \in \Omega$ and $j=1,2,\ldots$, we have
$$s-\lim_{j \to \infty} \II(f_n) = \II(f).$$
\end{enumerate}
\end{thm}

\begin{proof}
In view of Definition \ref{def:BFUNC}, it suffices to prove (i)-(vi) for $f, g \in \mathfrak{B}_s(\Omega, \sA, \CC_{\gI})$. Let $f(\lambda) = \sum_{j=1}^k c_j \chi_{M_j}(\lambda)$ and $g(\lambda) = \sum_{j=1}^k d_j \chi_{N_j}(\lambda)$. Using the fact that $J E(M) = E(M) J$ for all $M \in \sA$ and $J^* = -J$, we have
\begin{align*}
\II(f)^* =& \; \sum_{j=1}^k E(M_j)^* \{ {\rm Re}(c_j)I + {\rm Im}(c_j) J^* \} E(M_j) \\
=& \; \sum_{j=1}^k E(M_j) \{ {\rm Re}(c_j)I - {\rm Im}(c_j) J \} E(M_j) \\
=& \; \sum_{j=1}^k  \{ {\rm Re}(c_j) I - {\rm Im}(c_j) J \} E(M_j) \\
=& \; \II(\bar{f}).
\end{align*}
Thus, we have proved (i).

Next, since
\begin{align*}
\II(f) \II(g) =& \; \left( \sum_{i=1}^k \{ {\rm Re}(c_i)I + {\rm Im}(c_i) J \}E(M_i)  \right)
\left( \sum_{j=1}^k \{ {\rm Re}(d_j)I + {\rm Im}(d_j) J \}E(N_j)  \right) \\
=& \; \sum_{i,j=1}^k \{ {\rm Re}(c_i) {\rm Re}(d_j)I - {\rm Im}(c_i) {\rm Im}(d_j)I \\ \;\;\;\; +& ({\rm Re}(c_i) {\rm Im}(d_j)  + {\rm Re}(d_j) {\rm Im}(c_i) )J \}  E(M_i \cap N_j) \\
=& \; \II(f) \II(g),
\end{align*}
we have proved (ii).

To prove the first statement in (iii), simply observe that $f(\lambda) = \gI \, \chi_{\Omega}(\lambda) = \gI$ and hence $\II(\gI) = J$. The second statement in (iii) is an easy consequence of (ii) and the fact that $\II(\gI) = J$.

The proof of (iv) is very straight forward. To prove (v), we may use (i), (ii) and (iv) to verify that
\begin{align*}
\| \II(f) x \|^2 =& \; {\rm Re} \langle \II(f)x, \II(f)x \rangle \\
=& \; {\rm Re}\, \langle \II(\bar{f}) \II(f)x, x \rangle \\
=& \; {\rm Re}\, \langle \II(|f|^2) x, x \rangle \\
=& \; {\rm Re} \left( \int_{\Omega} |f(\lambda)|^2 d \langle E(\lambda) x, x \rangle\right) \\
=& \; \int_{\Omega} |f(\lambda)|^2 d( {\rm Re} \, \langle E(\lambda)x, x \rangle) \quad \quad {\rm for} \quad  x\in \cH_n.
\end{align*}
Statement (v) is a direct consequence of (iv). Indeed, using (iv), we have
$$\| \II(f)x \|^2 \leq \| f \|_{\infty} \, {\rm Re}\langle x, x \rangle = \| f \|_{\infty} \| x\|^2 \quad \quad {\rm for} \quad x \in \cH_n,$$
in which case we have (v).

Finally, to prove (vii), we may use (vi) to see that
$$\| \{ \II(f) - \II(f_n) \} x \|^2 = \int_{\Omega} |f(\lambda) - f_n(\lambda)|^2 d ( {\rm Re}\langle E(\lambda)x, x \rangle ) \quad \quad {\rm for} \quad x \in \cH_n$$
and hence (vii) follows immediately from the Lebesgue dominated convergence theorem.
\end{proof}

\begin{lem}
\label{lem:L2}
Let $\sA$ be an algebra of sets in $\Omega$ and $E$ be a spectral measure on $\sA$. Then the following facts hold{\rm :}
\begin{enumerate}
\item[(i)] $|E_{x,y}|(M) \leq
\sqrt{ {\rm Re}\, E_x(M) } \, \sqrt{ {\rm Re} \, E_y(M) }$ for all $x,y \in \cH_n$ and $M \in \sA$.

\smallskip

\item[(ii)] Let $f \in L_2(\Omega, \CC_{\gI}, E_x)$ and $g \in L_2(\Omega,\CC_{\gI}, E_y)$. Then
\begin{align}
& \; \left| \int_{\Omega} {\rm Re}(f(\lambda) g(\lambda)) \, d\langle E(\lambda)x, y \rangle +\int_{\Omega} {\rm Im}(f(\lambda) g(\lambda)) \, d\langle E(\lambda)Jx, y \rangle \right| \nonumber \\
\leq& \; \int_{\Omega} |f(\lambda)g(\lambda)| d|E_{x,y}|(\lambda) \nonumber \\
\leq& \;  2^n \| f \|_{L_2(\Omega,\CC_{\gI}, E_x)} \| g \|_{L_2(\Omega,\CC_{\gI}, E_y)}, \label{eq:CSL_2}
\end{align}
where $E_{x,y}(M) = \langle E(M)x, y\rangle$ for $M \in \sB(\Omega)$ and for all $x, y \in \cH_n$.
\end{enumerate}
\end{lem}

\begin{proof}
We will first prove (i). Suppose $M = \cup_{k=1}^p M_k$, where $M_k \in \sA$ and $M_k \cap M_j = \emptyset$ whenever $k \neq j$. Using Lemma \ref{lem:KAPPA}, we have
\begin{align*}
|E_{x,y}(M_k) | =& \;  | \langle E(M_k)x, y \rangle | = | \langle E(M_k)x, E(M_k)y \rangle | \\
\leq& \;  \| E(M_k)x \| \| E(M_k) y \| \\
=& \;  \sqrt{ {\rm Re}\, E_x(M_k) } \, \sqrt{ {\rm Re} \, E_y(M_k) }.
\end{align*}
Consequently, we have
\begin{align*}
\sum_{k=1}^p |E_{x,y}(M_k)| \leq& \;  \sum_{k=1}^p \sqrt{ {\rm Re}\, E_x(M_k) } \, \sqrt{ {\rm Re} \, E_y(M_k) } \\
\leq& \; \left( \sum_{k=1}^p  {\rm Re}\, E_x(M_k) \right)^{1/2} \, \left( \sum_{k=1}^p {\rm Re} \, E_y(M_k) \right)^{1/2} \\
=& \; \sqrt{{\rm Re}\, E_x(M) } \, \sqrt{ {\rm Re}\, E_y(M) }.
\end{align*}
Finally, if we take the supremum over all possible disjoint unions of $M$ in $\sA$, then we obtain (i).

We will now prove (ii). We will verify \eqref{eq:CSL_2} for simple functions, in which case \eqref{eq:CSL_2} will hold by the density of simple functions in both $L_2$-spaces. Suppose $f(\lambda) = \sum_{k=1}^p f_k \chi_{M_k}(\lambda)$ and $g(\lambda) = \sum_{k=1}^p g_k \chi_{M_k}(\lambda)$ are simple functions in $L_2(\Omega, \CC_{\gI}, E_x)$ and $L_2(\Omega, \CC_{\gI}, E_y)$, respectively. Using \eqref{eq:SMM}, Assertion (i) and \eqref{eq:SMM}, we have
\begin{align*}
& \; \left| \int_{\Omega} {\rm Re}(f(\lambda) g(\lambda)) d\langle E(\lambda)x,y\rangle + \int_{\Omega} {\rm Im}(f(\lambda)g(\lambda)) d\langle E(\lambda)Jx, y \rangle \right|\\
 =& \; \left| \sum_{k=1}^p \{ {\rm Re}(f_k g_k) \langle E(M_k)x, y \rangle + {\rm Im}(f_kg_k)\langle E(M_k)Jx, y \rangle \} \right| \\
\leq& \; 2^{n-1} \sum_{k=1}^p |f_k \, g_k| | \{ |E_{x,y}|(M_k)  + |E_{Jx,y}|(M_k)\} \\
\leq& \;  2^{n-1} \sum_{k=1}^p |f_k \, g_k | \left( \sqrt{ {\rm Re}\, E_x(M_k) } \, \sqrt{ {\rm Re}\, E_y(M_k) } +  \sqrt{ {\rm Re}\, E_{Jx}(M_k) } \, \sqrt{ {\rm Re}\, E_y(M_k) } \right)\\
=& \;  2^{n-1} \sum_{k=1}^p |f_k \, g_k | \left( \sqrt{ {\rm Re}\, E_x(M_k) } \, \sqrt{ {\rm Re}\, E_y(M_k) } +  \sqrt{ {\rm Re}\, E_{x}(M_k) } \, \sqrt{ {\rm Re}\, E_y(M_k) } \right)\\
\leq& \;  2^n \left(   \sum_{k=1}^p |f_k|^2 {\rm Re}\, E_x(M_k)  \right)^{1/2} \\
\times& \;
\left(   \sum_{k=1}^p |g_k|^2 {\rm Re}\, E_y(M_k)  \right)^{1/2}.
\end{align*}
Thus, \eqref{eq:CSL_2} holds for simple functions.
\end{proof}

\subsection{Spectral integrals of unbounded measurable functions}
\label{sec:SIubdd}
Fix $\gI \in \mathbb{S}$. Let $\gF(\Omega, \sA, \CC_{\gI}, E)$ denote the set of all $\sA$-measurable functions \\$f: \Omega \to \CC_{\gI} \cup \{ \infty \}$ which are $E$-a.e. finite, i.e., $E ( \{ \lambda \in\Omega: f(\lambda) = \infty \}) = 0$.

\begin{defn}[bounding sequence]
\label{def:BOUNDINGSEQ}
Let $(M_j)_{j=1}^{\infty}$ be a sequence of sets, where $M_j \in \sB(\Omega, \sA, \CC_{\gI})$ and $M_j \subseteq M_{j+1}$ for $j=1,2,\ldots$. We will call $(M_j)_{j=1}^{\infty}$ a {\it bounding sequence} for a subset $\gG$ of $\gF(\Omega, \sA, \CC_{\gI}, E)$ if every $f \in \gG$ is bounded on $M_j$ and $E(\cup_{j=1}^{\infty} M_j) = I$.
\end{defn}
\begin{rem}
We note that if $(M_j)_{j=1}^{\infty}$ is a bounding sequence, then for the results appearing in Subsection \ref{sec:BASIC}, we have $E(M_j) \preceq E(M_{j+1})$ for $j=1,2,\ldots$,
$$\lim_{j \to \infty} E(M_j) x = x \quad \quad {\rm for} \quad x \in \cH_n$$
and $\cup_{j=1}^{\infty} E(M_j) \cH_n$ is dense in $\cH_n$.
\end{rem}

\begin{rem}
\label{rem:EVERYBDD}
Every finite subset $\{ f_1, \ldots, f_k \} \subseteq \gF(\Omega, \sA, \CC_{\gI}, E)$ has a bounding sequence. To see this, let
$$M_j := \{ \lambda \in \Omega: |f_j(\lambda)| \leq n \quad {\rm for} \quad j =1,\ldots, k \}$$
and $M = \cup_{j=1}^{\infty} M_j$. Thus, $\Omega \, \setminus \, M \subseteq \cup_{j=1}^k \{ \lambda \in \Omega: f_j(\lambda) = \infty \}$, in which case $E(\Omega \, \setminus \, M) = 0$.  Consequently, $(M_j)_{j=1}^{\infty}$ is a bounding sequence for $\{ f_1, \ldots, f_k \}$.
\end{rem}

\begin{thm}
\label{thm:BIG}
Let $f \in \gF(\Omega, \sA, \CC_{\gI}, E)$ and
\begin{equation}
\label{eq:DOM}
\cD( \II(f) ) := \{ x \in \cH_n: \int_{\Omega} |f(\lambda)|^2 d( {\rm Re}\langle E(\lambda)x, x \rangle ) < \infty \}.
\end{equation}
For any bounded sequence $(M_j)_{j=1}^{\infty}$ for $f$, we have the following{\rm :}
\begin{enumerate}
\item[(i)] $\cD(\II(f)) = \{ x \in \cH_n: \text{$(\II(f \chi_{M_j})x)_{j=1}^{\infty}$ converges in $\cH_n$}\} = \{ x \in \cH_n: \sup_{j \in \NN} \| \II(f \chi_{M_j}) x \| < \infty \}$.
\smallskip
\item[(ii)] For any $x \in \cD(\II(f))$, the limit of the sequence $(\II(f \chi_{M_j})_{k=1}^{\infty}$ does not depend on the choice of the bounding sequence $(M_j)_{j=1}^{\infty}$. Moreover, there is a linear operator $\II(f) \in \cL(\cH_n)$ with domain $\cD(\II(f))$ given by
\begin{equation}
\label{eq:IIf}
\II(f)x = \lim_{j \to \infty} \II(f \chi_{M_j})x \quad \quad {\rm for} \quad x \in \cD(\II(f)).
\end{equation}

\smallskip
\item[(iii)] The right submodule $\cup_{j=1}^{\infty} E(M_j)\cH_n \subseteq \cD(\II(f)) \subseteq \cH_n$ and is a core for $\II(f)$. Moreover,
\begin{equation}
\label{eq:SUBDOM}
E(M_j) \II(f) \subseteq \II(f) E(M_j) = \II(f \chi_{M_j}) \quad \quad {\rm for} \quad j=1,2,\ldots.
\end{equation}
\end{enumerate}
\end{thm}

\begin{proof}
The proof is broken into steps.

\bigskip

\noindent {\bf Step 1:} {\it Prove {\rm (i)}}.

\bigskip

Let $x \in \cD(\II(f))$. Thus, by definition,  $f \in L_2(\Omega, \sA, \CC_{\gI}, E_x)$ and we may use Lebesgue's dominated convergence theorem to obtain $f \chi_{M_i} \to f$ in $L_2(\Omega, \CC_{\gI}, E_x)$. Since $(M_j)_{j=1}^{\infty}$ is a bounding sequence for $f$, we have that $f$ is bounded on $M_j$. Thus, $f \chi_{M_j} \in \gB(\Omega, \sA, \CC_{\gI})$ and $\II(f \chi_{M_j}) \in \cB(\cH_n)$ given by Definition \ref{def:BFUNC}. Using Theorem \ref{thm:BFUNC}(v), we obtain
\begin{align*}
\| \II(f \chi_{M_i})x - \II(f \chi_{M_j}) x \|^2 =& \; \| \II(f \chi_{M_i} - f \chi_{M_j} )x \|^2 \\
=& \; \int_{\Omega} | f(\lambda)\chi_{M_i}(\lambda) - f(\lambda) \chi_{M_j}(\lambda) |^2 d ( {\rm Re} \langle E(\lambda)x, x \rangle) \\
=& \; \| f \chi_{M_i} - f \chi_{M_j} \|^2_{L_2(\Omega, \CC_{\gI}, E_x)} \quad {\rm for} \quad i,j=1,2,\ldots.
\end{align*}
 Thus, putting all of the above observations together, we have that \\ $(\II(f \chi_{M_i}) x)_{i=1}^{\infty}$ is a Cauchy sequence in $\cH_n$ and hence $\II(f \chi_{M_i})x)_{i=1}^{\infty}$ converges in $\cH_n$.

Next, if $\{ \II(f \chi_{M_i})x : i=1,2,\ldots \}$ converges in $\cH_n$, then $\{ \| \II(f \chi_{M_i})x \| : i=1,2,\ldots \}$ is bounded.

Finally, suppose that $\{ \| \II(f \chi_{M_i} )x \|: i=1,2,\ldots \}$ is bounded. Then
$$\sup_{i \in \NN} \| \II(f \chi_{M_i})x \| < \infty.$$
Since $|f(\lambda)\chi_{M_i}(\lambda)|^2$ converges monotonically to $|f(\lambda)|^2$ $E_x$-a.e. on $\Omega$, we may use Lebesgue's monotone convergence theorem to obtain
\begin{align*}
\int_{\Omega} |f(\lambda)|^2 d( {\rm Re}\langle E(\lambda)x, x \rangle) =& \; \lim_{i \to \infty} \int_{\Omega} f(\lambda) \chi_{M_i}(\lambda)|^2 d( {\rm Re}\langle E(\lambda)x, x \rangle) \\
=& \; \lim_{i \to \infty} \| \II(f \chi_{M_i})x \|^2 < \infty.
\end{align*}
Thus, $f \in L_2(\Omega, \sA, \CC_{\gI}, E_x)$ and hence $x \in \cD(\II(f))$. We have managed to show
\begin{align*}
\cD(\II(f)) \subseteq& \; \{ x \in \cH_n: \text{$(\II(f \chi_{M_j})x)_{j=1}^{\infty}$ converges in $\cH_n$}\} \\
\subseteq& \;  \{ x \in \cH_n: \sup_{j \in \NN} \| \II(f \chi_{M_j}) x \| < \infty \} \\
\subseteq& \; \cD(\II(f)),
\end{align*}
in which case we have (i).

\bigskip

\noindent {\bf Step 2:} {\it Prove {\rm (ii)}}.

\bigskip

Suppose $(N_i)_{i=1}^{\infty}$ is also a bounding sequence for $f$ and $x \in \cD(\II(f)).$ Using Theorem \ref{thm:BFUNC}(i) and (v), we have
\begin{align*}
\| \II(f \chi_{M_i}x - \II(f \chi_{N_j} )x \|=& \; \| \II( f \chi_{M_i} - f \chi_{N_j}) x \|_{L_2(\Omega, \CC_{\gI}, E_x)} \\
\leq& \; \| f \chi_{M_i} - f \|_{ L_2(\Omega, \sA, \CC_{\gI}, E_x) } + \| f \chi_{N_j} - f \|_{ L_2(\Omega, \sA, \CC_{\gI}, E_x)} \\
\to& \; 0 \quad\quad {\rm as} \quad i,j \to \infty.
\end{align*}
Thus,
$$\lim_{i \to \infty} \II(f \chi_{M_i})x = \lim_{j \to \infty} \II(f \chi_{N_j})x \quad \quad {\rm for} \quad x \in \cD(\II(f)).$$
By (i), we have that $\cD(\II(f))$ is a right linear subspace of $\cH_n$ and \eqref{eq:IIf} gives rise to a right linear operator $\II(f) \in \cL(\cH_n)$.

\bigskip

\noindent {\bf Step 3:} {\it Prove {\rm (iii)}}.

\bigskip

Suppose $y \in \cH_n$. Recall that $E(M_i) = \II(\chi_{M_i})$ and hence we may use Theorem \ref{thm:BFUNC} to obtain
\begin{align}
\II( f \chi_{M_i}) y =& \; \II(f \chi_{M_i} \chi_{M_j} )y = \II(f \chi_{M_i}) E(M_j)y  \nonumber \\
=& \; E(M_j) \II(f \chi_{M_i})y \quad\quad {\rm whenever} \quad i \geq j.\label{eq:SUPPEQ}
\end{align}
Consequently, $\sup \{ \| \II(f \chi_{M_i})y \|: i=1,2,\ldots \} < \infty$ and hence $E(M_i)y \in \cD(\II(f))$, i.e., the union $\cup_{i=1}^{\infty} E(M_i) \cH_n \subseteq \cD(\II(f))$.

Letting $i \to \infty$ in \eqref{eq:SUPPEQ}, we obtain
$$\II(f) E(M_i)y = \II(f \chi_{M_i})y \quad \quad {\rm for} \quad y \in \cH_n.$$
Suppose $x \in \cD(\II(f))$ and letting $i \to \infty$ in \eqref{eq:SUPPEQ},  we obtain the equality in \eqref{eq:SUBDOM}.

Finally, using the fact that $E(M_i)y \to y$ as $i \to \infty$ for all $y \in \cH_n$ and $\II(f) E(M_i)x = E(M_i) \II(y) \to \II(f)x$ for all $x \in \cD(\II(f))$, we have that $\cup_{i=1}^{\infty} E(M_i) \cH_n$ is a core for $\II(f)$.
\end{proof}

\begin{rem}
\label{rem:OBSERVATION}
If $f \in \gB(\Omega, \sA, \CC_{\gI}, E)$, then $(M_i)_{i=1}^{\infty}$, where $M_i = \Omega$ for $i=1,2,\ldots$, is a bounding sequence for $f$. Consequently, the operator $\II(f) \in \cL(\cH_n)$ given by \eqref{eq:IIf} coincides with $\II(f) \in \cB(\cH_n)$ in Definition \ref{def:BFUNC}.
\end{rem}

\begin{thm}
\label{thm:FG}
Suppose $f,g \in \gF(\Omega, \sA, \CC_{\gI}, E)$. Then
\begin{align}
\langle \II(f))x, \II(g)y \rangle =& \; \int_{\Omega} {\rm Re}(f(\lambda) \overline{g(\lambda)}) d \langle E(\lambda)x, y \rangle  \nonumber \\
+& \; \int_{\Omega} {\rm Im}(f(\lambda) \overline{g(\lambda)}) d \langle JE(\lambda)x, y \rangle \label{eq:XY}
\end{align}
for $x \in \cD(\II(f))$ and $y \in \cD(\II(g))$ and
\begin{equation}
\label{eq:fNORM}
\| \II(f)x \|^2 = \int_{\Omega} |f(\lambda)|^2 d( {\rm Re}\langle E(\lambda)x, x \rangle \quad \quad {\rm for} \quad x \in \cD(\II(f)).
\end{equation}
\end{thm}

\begin{proof}
We will first show \eqref{eq:XY}. Let $(M_i)_{i=1}^{\infty}$ be a bounding sequence for $\{ f, g \}$ and hence $fg \chi_{M_i} \in \gB(\Omega, \sA, \CC_{\gI}, E)$. We may make use of Theorem \ref{thm:BFUNC}(i) and (ii) to obtain
\begin{align}
& \; \int_{\Omega} {\rm Re}(f(\lambda) \overline{g(\lambda)}) d\langle E(\lambda)x, y \rangle
+ \int_{\Omega}{\rm Im}(f(\lambda)\overline{g(\lambda)})d\langle JE(\lambda)x,y\rangle \nonumber \\
=& \; \langle \II(f \bar{g}\chi_{M_i})x, y \rangle \nonumber \\
=& \; \langle \II(f \chi_{M_i})x, \II(g \chi_{M_i})y \rangle. \label{eq:FGsplit}
\end{align}
Since $x \in \cD(\II(f)$, we have $f \in L_2(\Omega, \sA, \CC_{\gI}, E_x)$ by \eqref{eq:DOM}. Similarly, since $y \in \cD(\II(g)$, we have $g \in L_2(\Omega,\sA, \CC_{\gI},  E_y)$. One can use Lemma \ref{lem:L2}(ii) to see that the integrals
$$\int_{\Omega} {\rm Re}(f(\lambda)\overline{g(\lambda)} d \langle E(\lambda)x, y \rangle \;
{\rm and} \;
\int_{\Omega} {\rm Im}(f(\lambda)\overline{g(\lambda)} ) d \langle E(\lambda)Jx, y \rangle$$
are both convergent. We may also use Lemma \ref{lem:L2}(ii) and the fact that $f \to f \chi_{M_i}$ in $L_2(\Omega, \sA,\CC_{\gI},E_x)$ to see that
\begin{align*}
& \; | \int_{\Omega} {\rm Re}(f(\lambda)\overline{g(\lambda)})\chi_{M_i}(\lambda) d\langle E(\lambda)x, y \rangle - \int_{\Omega} {\rm Re}(f(\lambda)\overline{g(\lambda)}) d\langle E(\lambda)x, y \rangle  \\
+& \; \int_{\Omega} {\rm Im}(f(\lambda)\overline{g(\lambda)})\chi_{M_i}(\lambda) d\langle E(\lambda)Jx, y \rangle - \int_{\Omega} {\rm Im}(f(\lambda)\overline{g(\lambda)}) d\langle E(\lambda)Jx, y \rangle | \\
=& \; | \int_{\Omega} {\rm Re}(  \{ f(\lambda)\chi_{M_i}(\lambda) - f(\lambda) \} \overline{g(\lambda)}) d\langle E(\lambda)x, y \rangle \\
+& \; \int_{\Omega} {\rm Im}(  \{ f(\lambda)\chi_{M_i}(\lambda) - f(\lambda) \} \overline{g(\lambda)}) d\langle E(\lambda)Jx, y \rangle| \\
\leq& \; 2^n \| f \chi_{M_i} - f \|_{L_2(\Omega, \sA,\CC_{\gI}, E_x)} \| g \|_{L_2(\Omega,\sA, \CC_{\gI}, E_y)} \\
\to& \; 0 \quad {\rm as} \quad i \to \infty.
\end{align*}
Thus, if we let $i \to \infty$ in \eqref{eq:FGsplit}, we obtain \eqref{eq:XY}.

Formula \eqref{eq:fNORM} is easily obtained from \eqref{eq:XY} by putting $y = x$ and $g = f$.
\end{proof}

\begin{thm}
\label{thm:UNBFUNCCALC}
Let $J$ be an imaginary operator associated with the spectral measure $E$. For any $f, g \in \gF(\Omega, \sA, \CC_{\gI}, E)$ and $c \in \CC_{\gI}$, where $\gI \in \mathbb{S}$, we have the following{\rm :}
\begin{enumerate}
\item[(i)] $\II(\bar{f}) = \II(f)^*$.

\smallskip
\item[(ii)] $\II(fg) = \overline{\II(f)\II(g)}$.

\smallskip
\item[(iii)] $\II(fc +g) = \overline{\II(f) ({\rm Re}(c)I + {\rm Im}(c)J )+ \II(g)}$ for all $c \in \CC_{\gI}$.

\smallskip
\item[(iv)] $\II(f)$ is a closed normal operator belonging to $\cL(\cH_n)$.

\smallskip
\item[(v)] $\cD(\II(f) \II(g)) = \cD( \II(fg) ) \cap \cD(\II(g))$.

\end{enumerate}
\end{thm}

\begin{proof}
Let $(M_i)_{i=1}^{\infty}$ be a bounding sequence $f, g \in \gF(\Omega, \sA, \CC_{\gI}, E)$ (see Remark \ref{rem:EVERYBDD}). We note that it is easy to check that $(M_i)_{I=1}^{\infty}$ is a bounded sequence for the functions $f+g, fg$ and $\overline{f}$.  Consequently, Theorem \ref{thm:BIG}(iii) can be used to see that
\begin{equation}
\label{eq:cE}
\mathcal{E} := \{ \bigcup_{i=1}^{\infty} E(M_i) x: x \in \cH_n \}
\end{equation}
is a core for $\II(f+g)$ and $\II(fg)$. The remainder of the proof is broken into steps.

\smallskip

\noindent {\bf Step 1:} {\it Prove {\rm (i)}}.

\smallskip

Suppose $x \in \cD(\II(f))$ and $ y\in \cD(\II(g))$. Using \eqref{eq:SUBDOM} and Theorem \ref{thm:BFUNC}(i) we obtain
\begin{align*}
\langle E(M_i)\II(f)x, y \rangle =& \; \langle \II(f \chi_{M_i}) x, y \rangle = \langle x, \II( \overline{f \chi_{M_i} y} ) \rangle = \langle x, \II( \overline{f} \chi_{M_i} )y\rangle \\
=& \; \langle x, E(M_i) \II(\overline{f}) y \rangle.
\end{align*}
Letting $i \to \infty$ above results in $\langle \II(f)x, y \rangle = \langle x, \II(\overline{f}) y\rangle$, i.e., $\II(\overline{f}) \subseteq \II(f)^*$. Next, suppose $x \in \cH_n$ and $y \in \cD(\II(f)^*)$. Then we can again make use of \eqref{eq:SUBDOM} and Theorem \ref{thm:BFUNC}(i) to obtain
$$\langle x, E(M_i) \II(f)^* y \rangle
= \langle \II(f) E(M_i)x, y \rangle
= \langle \II(f \chi_{M_i}) \, x, y \rangle
= \langle x, \II( \overline{f \chi_{M_i}} )  y \rangle.$$
Thus, $E(M_i) \II(f)^* y = \II(\overline{f \chi_{M_i} } )y$ and hence
\begin{align*}
\| \II(\overline{f \chi_{M_i}} ) y \|^2 =& \;  {\rm Re} \langle \II(\overline{f \chi_{M_i}} ) y, \II(\overline{f \chi_{M_i} } ) y \rangle  \\
=& \; {\rm Re} \langle E(M_i) \II(f)^* y, E(M_i) \II(f)^* y \rangle  \\
=& \; \| E(M_i) \II(f)^* y \|^2.
\end{align*}
But then
$$\sup_{i=1,2,\ldots } \| \II(f \chi_{M_i} ) \| \leq \| \II(f)^* y \|,$$
in which case, Theorem \ref{thm:BIG}(i) ensures that $y \in \cD(\II(\overline{f}))$. Thus, $\cD( \II(f)^* ) \subseteq \cD(\II(\overline{f}) )$ and hence $\II(\overline{f}) \subseteq \II(f)^*$ implies that $\II(f)^* = \II(\overline{f})$ as required.

\smallskip

\noindent {\bf Step 2:} {\it Prove {\rm (ii)}}.

\smallskip

We begin by noting that \eqref{eq:SUBDOM} asserts that $\cE$ is a dense right submodule such that $\cE \subseteq \cD(\II ( \overline{g} ) \II(\overline{f} ) )$.  Consequently, we may use (i) and Theorem \ref{thm:Dec14uu1} to see that $\II(f)\II(g)$ is closable. One can use \eqref{eq:SUBDOM} and Theorem \ref{thm:BFUNC}(i) to check that
\begin{equation}
\label{eq:E1}
\II(fg) E(M_i) = \II(f) \II(g) E(M_i)
\end{equation}
and hence
\begin{equation}
\label{eq:E2}
E(M_i) \II(f) \II(g) \subseteq \II(f) \II(g) E(M_i) = \II(fg) E(M_i).
\end{equation}
If we let $i \to \infty$ in \eqref{eq:E1} and we make use of the fact that $\cE$ is a core for $\II(fg)$, then we obtain $\II(fg) \subseteq \overline{\II(f)\II(g)}$. On the other hand, as $\II(fg) E(M_i) \in \cB(\cH_n)$, we may use \eqref{eq:E2} to conclude $E(M_i) \overline{\II(f)\II(g)} \subseteq \II(fg) E(M_i)$. Letting $i \to \infty$ in $E(M_i) \overline{\II(f)\II(g)} \subseteq \II(fg) E(M_i)$ and using \eqref{eq:IIf}, we obtain $\overline{\II(f)\II(g)} \subseteq \II(fg)$. But then we have (ii).

\smallskip

\noindent {\bf Step 3:} {\it Prove} (iii).

\smallskip

The fact that $\II(fc) = \II(f) ({\rm Re}(c) I + {\rm Im}(c)J )$, where $I$ is the identity operator and $J$ is the imaginary operator, is an immediate consequence of Theorem \ref{thm:BFUNC}(iii) and (ii) in the present theorem. Thus, in order to prove (iii), we need only consider $\II(f+g)$. Let $\cE$ be as in \eqref{eq:cE}. Then, in view of (i), the set containments
$$\cE \subseteq \cD(\II(\overline{f}) + \II(\overline{g}) ) = \cD(\II(f)^* + \II(g)^*) \subseteq \cD( \{ \II(f) + \II(g) \}^* )$$
demonstrate that $\II(f) + \II(g)$ is closable. Moreover, $\cE$ is a dense right submodule of $\cH_n$. Using \eqref{eq:SUBDOM} and Theorem \ref{thm:BFUNC}(i), we obtain
\begin{align}
(\II(f)+\II(g)) E(M_i) =& \; \II(f \chi_{M_i}) + \II(g \chi_{M_i}) \nonumber \\
=& \;  \II( f\chi_{M_i} + g\chi_{M_i} ) \nonumber \\
=& \;  \II(f+g) E(M_i)\label{eq:FF1}
\end{align}
and
\begin{align}
E(M_i)( \II(f)+\II(g)) =& \; E(M_i) \II(f) + E(M_i) \II(g)  \nonumber \\
\subseteq & \; \{ \II(f) + \II(g) \} E(M_i) \nonumber \\
=& \; \II(f+g) E(M_i). \label{eq:FF2}
\end{align}
If we let $i \to \infty$ in \eqref{eq:FF1} and make use of the fact that
$$\{ \bigcup_{i=1}^{\infty} E(M_i) x: x \in \cH_n \}$$
is a core for $\II(f+g)$, then we obtain $\II(f+g) \subseteq \overline{ \II(f)+\II(g) }$. Finally, \eqref{eq:FF2} implies that $E(M_i) \overline{\II(f) + \II(g)} \subseteq \II(f+g) E(M_i)$ and letting $i \to \infty$ we have $\overline{\II(f)+\II(g)} \subseteq \II(f+g)$.

\smallskip

\noindent {\bf Step 4:} {\it Prove} (iv).

\smallskip

It follows from (i) and (ii) that $\II(f)^*\II(f) = \II(\overline{f})\II(f) \subset \II(\overline{f} f )$. Note that $\II(f)^* \II(f) \in \cL(\cH_n)$ is closed. Thus, we may use Theorem \ref{thm:Ctransform}(ii) to conclude that $\II(f)^* \II(f)$ is self-adjoint. On the other hand, (i) implies that $\II(\overline{f} f )$ is self-adjoint, in which case, we have $\II(f)^* \II(f) = \II(\overline{f} f )$. We can replicated the above argument to obtain $\II(f) \II(f)^* = \II(f \overline{f})$ and hence we have $\II(f)^* \II(f) = \II(f) \II(f)^*$.

\smallskip

\noindent {\bf Step 5:} {\it Prove (v)}.

\smallskip

We begin by noting that (ii) implies that $\cD(\II(f)\II(g)) \subseteq \cD(\II(fg)) \cap \cD(g)$. Thus, in order to prove (v), we need only show the opposite set containment. Suppose $x \in \cD(\II(fg)) \cap \cD(\II(g))$. Then we may use \eqref{eq:E1} to obtain
$\II(fg) E(M_i)x = \II(f) \II(g) E(M_i)x$ and letting $i \to \infty$ and using the fact that $\II(f)$ is closed, we have $\II(g) x \in \cD(\II(f))$. Thus, we have $x \in \cD(\II(f)\II(f))$ as required.
\end{proof}

\begin{thm}
\label{thm:Ffacts}
For any $f,g \in \gF(\Omega, \sA, \CC_{\gI}, E)$, we have the following{\rm :}
\begin{enumerate}
\item[(i)] If $f = g$ $E$-a.e. on $\Omega$, then $\II(f) = \II(g)$.

\smallskip
\item[(ii)] If $f$ is real-valued $E$-a.e. on $\Omega$, then $\II(f) \in \cL(\cH_n)$ is a self-adjoint operator.

\smallskip
\item[(iii)] If $f(\lambda) \geq 0$ $E$-a.e. on $\Omega$, then $\II(f) \in \cL(\cH_n)$ is a positive operator. Moreover, $(\II(\sqrt{f}))^2 = \II(f)$ and hence $\II(f)$ has a positive square root.
\end{enumerate}
\end{thm}

\begin{proof}
Suppose $f = g$ $E$-a.e.. Then $E( \{ \lambda \in \Omega: f(\lambda) \neq g(\lambda) \} ) = 0$ and then it follows from \eqref{eq:DOM} that $\cD(\II(f)) = \cD(\II(g))$. We also have that any bounding sequence $(M_i)_{i=1}^{\infty}$ for $f$ is also a bounding sequence for $g$. Consequently, we may use \eqref{eq:IIf} to deduce that $\II(f) = \II(g)$. Thus, we have proved (i).

Next, suppose $f$ is real-valued $E$-a.e.. Then from \eqref{eq:XY} with $g(\lambda) = 1$ and $y = x$ we obtain
\begin{equation}
\label{eq:SELF}
\langle \II(f)x, x \rangle = \int_{\Omega} {\rm Re}(f(\lambda)) d\langle E(\lambda)x, x \rangle = \int_{\Omega} f(\lambda) d\langle E(\lambda)x, x \rangle \in \gS(\RR_n)
\end{equation}
for all $x \in \cD(\II(f))$. Thus, $\II(f)$ is self-adjoint and we have proved (ii).

Next, suppose $f \geq 0$ $E$-a.e.. Then \eqref{eq:SELF} can be used to show that $\II(f)$ is a positive operator. Finally, in view of the first part of (iii), it is obvious that $\II(\sqrt{f})$ is a positive operator. The fact that $(\II(\sqrt{f}))^2 = \II(f)$ is a direct consequence of Theorem \ref{thm:UNBFUNCCALC}(ii) and (v).
\end{proof}

\begin{defn}[$L_{\infty}(\Omega, \sA,\CC_{\gI}, E)$]
\label{def:L_infty}
Let $f \in \gF(\Omega, \sA, \CC_{\gI}, E)$. We shall let
$$L_{\infty}(\Omega, \sA,\CC_{\gI}, E) := \{ \text{$E$-measurable $f:\Omega \to \CC_{\gI}$: $f$ is bounded on $\Omega$} \}$$
be endowed with the norm
$$\| f \|_{\infty} := \inf_{{\stackrel{M \in \sA}{{\rm with}\; E(M) = 0}}} \sup\{ |f(\lambda)|: \lambda \in \Omega \, \setminus \, M \}.$$
\end{defn}

\begin{lem}
\label{lem:SIbounded}
Let $f \in \gF(\Omega, \sA, \CC_{\gI}, E)$. The operator $\II(f) \in \cB(\cH_n)$ if and only if $f \in L_{\infty}(\Omega, \sA, \CC_{\gI}, E)$. In this case, $\| \II(f) \| = \| f \|_{\infty}$.
\end{lem}

\begin{proof}
Suppose $f \in  L_{\infty}(\Omega, \sA, \CC_{\gI}, E)$. Then $|f(\lambda)| \leq \kappa$ $E$-a.e. for some $\kappa > 0$ and
$$
\| \II(f)x \|^2 = \int_{\Omega} |f(\lambda)|^2 d( {\rm Re}\langle E(\lambda)x, x \rangle ) \leq \| f \|_{\infty} \| x \|^2.
$$
Thus, $\II(f) \in \cB(\cH_n)$.

Next, suppose $\II(f) \in \cB(\cH_n)$ and put
$$M_i := \{ \lambda \in \Omega: |f(\lambda)| \geq \| \II(f) \| + 2^{-i} \} \quad \quad {\rm for} \quad i=1,2,\ldots.$$
and $M := \cup_{i=1}^{\infty} M_i$. Note that
$M = \{ \lambda \in \Omega: |f(\lambda)| > \| \II(f) \| \}$. Using Theorem \ref{thm:UNBFUNCCALC}(ii) and \eqref{eq:fNORM}, we obtain
\begin{align*}
\| \II(f) \|^2 \cdot \| E(M_i) x \|^2 \geq& \; \| \II(f) E(M_i) x \|^2 \\
=& \; \| \II(f \chi_{M_i} ) x \|^2 \\
=& \; \int_{\Omega} |f(\lambda) \chi_{M_i}(\lambda) |^2 d( {\rm Re}\langle E(\lambda)x, x \rangle) \\
=& \; \int_{\Omega} |f(\lambda) |^2 d( {\rm Re}\langle E(\lambda)x, x \rangle) \\
\geq& \; ( \| \II(f) \| + 2^{-i} )^2 \| E(M_i)x \|^2 \quad {\rm for} \quad i \in \NN \; {\rm and} \; x \in \cH_n.
\end{align*}
But then $E(M_i) = 0$ for all $i=1,2,\ldots$, in which case $E(M) =0$, i.e., $\| f \|_{\infty} \leq \| \II(f) \|$. Thus, we have the characterisation of elements in $L_{\infty}(\Omega, E)$ and $\| \II(f) \| = \| f \|_{\infty}$.
\end{proof}

\begin{lem}
\label{lem:INVERT}
Let $f \in \gF(\Omega, \sA, \CC_{\gI}, E)$. The right linear operator $\II(f) \in \cL(\cH_n)$ is invertible if and only if $f(\lambda) \neq 0$ $E$-a.e. on $\Omega$. In this case, $\II(f)^{-1} = \II(1/f)$, with the understanding that $\frac{1}{\infty} =0$ and $\frac{1}{0} = \infty$.
\end{lem}

\begin{proof}
Let $N :=\{ \lambda \in \Omega: f(\lambda) = 0 \}$. We may use Theorem \ref{thm:UNBFUNCCALC}(iii) and (v) to obtain
$$\II(f) E(N) = \II(f \chi_{N}) = \II(0) = 0.$$
Consequently, $\II(f)$ is not invertible whenever $E(N) \neq 0$.

Next, suppose $E(N) = 0$ and hence $1/f \in \gF(\Omega,E)$. We may use the fact that $f(1/f) = 1$ $E$-a.e. on $\Omega$ to see that $\cD(\II(f (1/f))) = \cD(\II(g)) = \cH_n$, where $g(\lambda) = 1$. Consequently, Theorem \ref{thm:UNBFUNCCALC}(v) implies that
$$\cD( \II(1/f) \II(f) ) = \cD(\II(f)).$$
Moreover, Theorem \ref{thm:UNBFUNCCALC}(iii) implies that
$$\II(1/f)\II(f) \subseteq \II( (1/f)f) = \II(1) = I.$$
Thus, putting these observations together, we have that $\II(f)$ is invertible and $\II(f)^{-1} \subseteq \II(1/f)$. To see that $\II(1/f) \subseteq \II(f)$, we may simply replace $f$ by $1/f$ in the above proof.
\end{proof}

\begin{defn}[intrinsic function]
\label{def:INTRINSIC}
  Let $\gI \in \mathbb{S}$. Let $\Omega\subseteq \mathbb{R}^{n+1}$ be an axially symmetric open set and let $\mathcal{U} = \{ (u,v)\in\mathbb{R}^2 : u+ \mathbb{S} v\subseteq \Omega\}$.
A function $f:\Omega\to \mathbb{R}_{n}$ is called a left
 slice function, if it is of the form
 \[
 f(\lambda) = f_{0}(u,v) + \gI f_{1}(u,v)\qquad \text{for } \lambda = u + \gI v\in \Omega
 \]
with two functions $f_{0},f_{1}: \mathcal{U}\to \mathbb{R}_{n}$ that satisfy the compatibility conditions
\begin{equation}\label{eq:INTRINSICREAL}
f_{0}(u,-v) = f_{0}(u,v),
\end{equation}
\begin{equation}
\label{eq:INTRINSICIMAGINARY}
 f_{1}(u,-v) = -f_{1}(u,v).
\end{equation}
We say that $f$
 is an {\it intrinsic function}
if in addition $f_{0}$, and $f_{1}$ are real valued.
We denote by  $\mathcal{S}(\Omega)$
the set of intrinsic functions defined on $\Omega$.
\end{defn}
\begin{rem}
\label{NOTATION-INTRINSIC}
For intrinsic functions, when we
 restrict $f$ to  $\Omega_\gI \subseteq\CC_{\gI}$,  we have that $f(\Omega_\gI)$ belongs to $\CC_{\gI}$,
since  $f_{0}$, and $f_{1}$ are real-valued.
To stress this fact we will often use the notation
$\gF(\Omega_\gI, \sB(\Omega_\gI),\CC_{\gI})$ instead of $\gF(\Omega, \sB(\Omega))$
when we work on a complex plane $\CC_{\gI}$, and when we consider the half plane
$\CC_{\gI}^+$ we write $\gF(\Omega_\gI^+, \sB(\Omega_\gI)^+,\CC_{\gI}^+)$.
Observe that condition (\ref{eq:INTRINSICIMAGINARY}) forces
$f_1(u,v)=0$ for every $u+\gI v\in \Omega_\gI^+\cap \mathbb{R}$ whenever $v=0$.
\end{rem}

\begin{thm}
\label{thm:IMPORTANT}
Let $E$ be a spectral measure on $\sB(\Omega_{\gI}^+)$, where $\Omega_{\gI}^+ \subseteq \CC_{\gI}^+$ for $\gI \in \mathbb{S}$, and $f \in \gF(\Omega_{\gI}^+, \sB(\Omega_{\gI}^+),\CC_{\gI}, E)$ be an intrinsic function. Then
\begin{align*}
\sigma_S(\II(f))\cap \CC^+_{\gI} =& \; \{ s \in \CC^+_{\gI}: E(\{ \lambda \in \Omega: |f(\lambda)^2 - 2{\rm Re}(s) f(\lambda) + |s|^2| < \varepsilon \} ) \\
& \;\; \neq 0 \text{ for all $\varepsilon > 0$}\}.
\end{align*}
\end{thm}

\begin{rem}
\label{rem:INTRINSIC}
The assumption that $f \in \gF(\Omega, \sB(\Omega),\CC_{\gI}, E)$ is intrinsic (see Definition \ref{def:INTRINSIC}) is necessary for preserving the fact that $\sigma_S(f(T))$ is axially symmetric (see Remark \ref{rem:AXIALLYSYMMETRIC}).
\end{rem}

\begin{proof}[Proof of Theorem \ref{thm:IMPORTANT}]
Notice that $s \in \rho_S(\II(f))$ if and only if
$( \II(f)^2 - 2 {\rm Re}(s) \II(f) + |s|^2 I )^{-1} \in \cB(\cH_n)$.
Thus, Lemma \ref{lem:SIbounded} and Lemma \ref{lem:INVERT} imply that $s \in \rho_S(\II(f)) \cap \CC_{\gI}^+$ if and only if
$f^2 - 2 {\rm Re}(s)f + |s|^2 \neq 0$ $E$-a.e. on $\Omega$ and $1/(f^2 - 2 {\rm Re}(s)f +|s|^2) \in L_{\infty}(\Omega,\sB(\Omega), \CC_{\gI}, E)$, or, equivalently, there exists $\kappa > 0$ such that
 $$
 E( \{ \lambda \in \Omega: |\ f(\lambda)- 2{\rm Re}(s)f(\lambda) - |s|^2 | \geq \kappa \}) = 0.
  $$
  Thus, $s \in \sigma_S(\II(f)) \cap \CC_{\gI}^+$ if and only if
$$E( \{ \lambda \in \Omega: |f(\lambda)^2 - 2{\rm Re}(s) f(\lambda) + |s|^2| < \varepsilon \} ) \neq 0$$
for all $\varepsilon > 0$.
\end{proof}

\begin{rem}
\label{rem:REMARKIMPORTANT}
We wish to highlight that Theorem \ref{thm:IMPORTANT} will be very useful for showing that an operator-valued integral \eqref{eq:SPECN} appearing in the spectral theorem for a bounded normal operator is on the $S$-spectrum of $T$.
\end{rem}
\begin{thm}
\label{thm:COMMUTINGspectral}
Let $E$ be a spectral measure on $\sB(\Omega)$ and $W \in \cB(\cH_n)$. Then $E(M)W  = W E(M)$ for all $M \in \sB(\Omega)$ if and only if $W \II(f) \subseteq \II(f) W$ for every $f \in \gF(\Omega, \sB(\Omega), \CC_{\gI}, E)$.
\end{thm}

\begin{proof}
Suppose $E(M)W  = W E(M)$ for all $M \in \sB(\Omega)$. Then, as $\II(\chi_{M} ) = E(M)$, we have $W \II(f) \subseteq \II(f)W$ for every $f \in \gF(\Omega, \sB(\Omega), \CC_{\gI}, E)$. Conversely, suppose $W \II(f) \subseteq \II(f) W$ for every $f \in \gF(\Omega, \sB(\Omega), \CC_{\gI}, E)$.  Then
\begin{align*}
{\rm Re}\, \langle E(M)Wx, Wx \rangle =& \; \| E(M)Wx \|^2 \\
=& \; \| W E(M) x \|^2 \\
\leq& \; \| W \|^2 \| E(M)x \|^2 \\
=& \; \| W \|^2 {\rm Re} \langle E(M)x, x \rangle
\end{align*}
for all $M \in \sB(\Omega)$ and $x \in \cH_n$. Thus, $Wx \in \cD(\II(f))$ whenever $x \in \cD(\II(f))$. It is immediate to see that $W \II(f) = \II(f) W$ whenever $f \in \gB_s(\Omega, \sA, \CC_{\gI}, E)$ and that this observation can be extended to all
$f \in \gB(\Omega, \sA, \CC_{\gI}, E)$. Consequently, we have $W \cD(\II(f)) \subseteq \cD(\II(f))$ and hence $W \II(f) \subseteq \II(f) W$.
\end{proof}

\begin{thm}
\label{thm:COMMUT}
Suppose $S \in \cB(\cH_n)$ and $E$ is a spectral measure on $(\Omega, \sA)$. Then $S E(M) = E(M) S$ for every $M \in \sA$ if and only if $S \II(f) \subseteq \II(f) S$ for every $f \in \gF(\Omega, \sB(\Omega), \CC_{\gI}, E)$.
\end{thm}

\begin{proof}
If $S \II(f) \subseteq \II(f) S$ for every $f \in \gF(\Omega, \sA, E)$, then $S E(M) = E(M) S$ for every $M \in \sA$ is obvious. Conversely, suppose $S E(M) = E(M) S$ for every $M \in \sA$. Since
\begin{align*}
{\rm Re}(E_{Sx}(M)) =& \; {\rm Re}(\langle E(M)Sx, E(M)Sx \rangle) = {\rm Re}( \langle S E(M)x, S E(M)x \rangle) \\
 =& \; \| SE(M)x \|^2 \\
 \leq& \; \| S \|^2 \| E(M)x \|^2 \\
 =& \; \| S \|^2 {\rm Re}(\langle E(M)x, E(M)x \rangle,
 \end{align*}
we may use \eqref{eq:DOM} to conclude that $Sx \in \cD(\II(f))$ whenever $x \in \cD(\II(f))$. Since $S E(M) = E(M) S$ for every $M \in \sA$, we have that $S \II(f_0) = S \II(f_0)$ for every simple function $f_0 \in \gB_s$, which can easily be extended, by taking limits, to $S \II(f) = \II(f) S$ for every $f \in \gB$. The fact that $S \II(f) = \II(f) S$ for every $f \in \gF(\Omega, \sA, E)$ follows immediately from \eqref{eq:IIf}.
\end{proof}

\begin{defn}[transformation of a spectral measure]
\label{def:TSM}
Let $E$ be a spectral measure on a Borel $\sigma$-algebra $\sA$ generated by a set $\Omega$, $J$ be an associated imaginary operator with $E$ and $\psi$ be a function such that $\psi: \Omega \to \Omega'$ and $\sA'$ be the $\sigma$-algebra of all subsets $M' \subseteq \Omega'$ such that $\psi^{-1}(M') \in \sA$. We shall let
\begin{equation}
\label{eq:NEWMEASURE}
E'(M') := E(\psi^{-1}(M')) \quad \quad {\rm for} \quad M' \in \sA'.
\end{equation}
\end{defn}

\begin{rem}
\label{rem:TRANSSPEC}
It is very easy to check that $E': \sA' \to \cB(\cH_n)$ is a spectral measure on $\sA'$. Moreover, in view of \eqref{eq:NEWMEASURE} if $J$ is an imaginary operator associated with $E$, then $J$ is also an imaginary operator associated with the $E'$.
\end{rem}

\begin{thm}
\label{thm:CHANGEOFVAR}
Let $\Omega, \Omega', \psi, J, E$ and $E'$ be as in Definition \ref{def:TSM}. Suppose $h \in \gF(\Omega', \sA', F)$. Then $h \circ \psi \in \gF(\Omega, \sA, E)$ and
\begin{align}
& \;\int_{\Omega'} {\rm Re}(h(\lambda')) \, dE'(\lambda') +  \int_{\Omega'} {\rm Im}(h(\lambda')) \, dE'(\lambda')J  \nonumber \\[4pt]
=& \;  \int_{\Omega} {\rm Re}(h(\psi(\lambda))) \, dE(\lambda)  + \int_{\Omega} {\rm Im}(h(\psi(\lambda))) \, dE(\lambda)J  \label{eq:INTTRANS}
\end{align}
\end{thm}

\begin{proof}
The fact that $h \circ \psi \in \gF(\Omega, \sA, E)$ follows immediately from Definition \ref{def:TSM}. Since $\langle E(\cdot) x,x \rangle$ and $\langle F(\cdot) x, x \rangle$ are positive $\RR_n$-valued measures on $\sA$ and $\sA'$ for any $x \in \cH_n$, respectively, we have invoke Theorem \ref{thm:TRANSSCALARMEAS} to obtain
$$\int_{\Omega'} |h(\lambda')|^2 d\langle E'(\lambda')x, x \rangle = \int_{\Omega} |h(\psi(\lambda))|^2 d\langle E(\lambda)x ,x \rangle \quad \quad {\rm for} \quad x\in \cH_n$$
and
\begin{align}
& \;  \int_{\Omega'}  {\rm Re}(h(\lambda') ) d\langle E'(\lambda') y, y \rangle + \int_{\Omega'}  {\rm Im}(h(\lambda') ) d\langle E'(\lambda')J y, y \rangle   \nonumber \\[4pt]
=& \; \int_{\Omega} {\rm Re}(h(\psi(\lambda)) )d\langle E(\lambda)y ,y \rangle
+\int_{\Omega} {\rm Im}(h(\psi(\lambda))) d\langle E(\lambda)J y ,y \rangle \label{eq:SECONDINT}
\end{align}
for all $y \in \cH_n$ such that $h$ is $\langle E'(\cdot) y, y \rangle$-integrable on $\Omega'$. We may use  \eqref{eq:DOM} and \eqref{eq:SECONDINT} to realise that $\cD(\II_{E'}(h)) = \cD(\II_E (h \circ \psi) )$, where $\II_{E'}$ and $\II_E$ denote the spectral integrals with respect to $E'$ and $E$, respectively. Notice that \eqref{eq:XY} with $g(\lambda) = 1$ and \eqref{eq:SECONDINT} can be used to obtain $\langle \II_{E'}(h)y, y \rangle = \langle \II_E(h \circ \psi)y, y\rangle$ for all $y \in \cD(\II_E(h \circ \psi))$ (which coincides with $\II_{E'}(h)$). Finally, the polarisation formula \eqref{eq:POLAR}, to deduce \eqref{eq:INTTRANS}.

\end{proof}

\section{Spectral theorem for a bounded self-adjoint operator}
\label{sec:STBSA}

The main goal of this section is to formulate and prove the spectral theorem for bounded self-adjoint operators on a Clifford module (see Theorem \ref{thm:BSA}).

\begin{thm}[spectral theorem for bounded self-adjoint operators]
\label{thm:BSA}
Let $T \in \cB(\cH_n)$ be self-adjoint. Then there exists a spectral measure $E$ on the Borel $\sigma$-algebra $\mathscr{B}(\sigma_S(T))$ such that
\begin{equation}
\label{eq:SPECSA}
T= \int_{\sigma_S(T)} t \, dE(t).
\end{equation}
The spectral measure $E$ is unique in the sense that if $F$ is a spectral measure on $\mathscr{B}(\RR)$ such that $T = \int_{\RR} t \, dF(t)$, then $E(M \cap \sigma_S(T)) = F(M)$ for all $M \in \mathscr{B}(\RR)$.  Moreover, $W \in \cB(\cH_n)$ commutes with $T$ if and only if $W E(M) = E(M) W$ for every $M \in \sB(\sigma_S(T))$.
\end{thm}

Before we can prove Theorem \ref{thm:BSA}, we need a number of lemmas. In the following one, we wish to stress that the polynomial $p$ appearing in the spectral mapping identity \eqref{eq:SPID} is a polynomial with real coefficients.

\begin{lem}
\label{lem:SPECTRAL_MAPPING}
Let $T \in \cB(\cH_n)$ be self-adjoint. Then
\begin{equation}
\label{eq:SPID}
\sigma_S(p(T)) = p(\sigma_S(T)) \quad \quad {\rm for} \quad p \in \RR[t].
\end{equation}
\end{lem}

\begin{proof}
Suppose $p(t) = \sum_{j=0}^k p_j t^j \in \RR[t]$, with $p_k \neq 0$. We begin by noting that $p(T) := \sum_{j=0}^n p_j T^j \in \cB(\cH_n)$ is self-adjoint. Thus, Lemma \ref{lem:15July1} asserts that $\sigma_S(p(T)) \subseteq \RR$.  If $k = 0$, then \eqref{eq:SPID} holds trivially.

Suppose $k > 0$. We will first show that $p(\sigma_S(T)) \subseteq \sigma_S(p(T))$.  We will only show that $\sigma_S(p(T)) \subseteq p(\sigma_S(T))$ (the proof of the other containment is very standard and does not divert at all from the classical case).  Choose $t_0 \in \sigma_S(p(T))$.  Consider the polynomial $\varphi(t) := p(t) - t_0 \in \RR[t]$. If $\varphi$ has all real zeros, say $t_1, \ldots t_k$ (not necessarily distinct), then $\varphi(t) = c \prod_{j=1}^k (t - t_j)$ for some constant $c \in \RR$. Therefore,
$$\varphi(T) = p(T) - t_0 \, I = c \prod_{j=1}^k (T - t_j \, I).$$
Since $\varphi(T)$ is not invertible, $T - t_{j'} \, I$ must not be invertible for some $j' \in \{ 1, \ldots, k \}$. Thus, $t_{j'} \in \sigma_S(T)$, in which case $p(t_{j'}) = t_0$.

Next, if $\varphi$ has non-real zeros, say $\zeta_1, \bar{\zeta}_1, \ldots, \zeta_{\ell}, \bar{\zeta}_{\ell}$.  Then
$$
\varphi(t) = \psi(t) \prod_{j=1}^{\ell} (t - \zeta_k)( t- \bar{\zeta}_m)
$$
 for some polynomial $\psi \in \RR[t]$ which is constant or has all real zeros. Write $\zeta_j = u_j + v_j \, \gI$, where $v_j \neq 0$ for $j =1 ,\ldots, \ell$.  Since
$$(t - \zeta_j )(t - \bar{\zeta}_j) = (t - u_j)^2 + v_j^2 \quad \quad {\rm for} \quad j=1,\ldots, \ell,$$
we have
$$\varphi(T) = \psi(T) \prod_{j=1}^{\ell} \{ (T - u_j \, I)^2 + v_j^2 \, I \}.$$
Since $T -u_j \, I \in \cB(\cH_n)$ is self-adjoint, we have that $(T - u_j \, I)^2$ is a positive operator, in which case $(T - u_j \, I)^2 + v_j^2 \, I$ is invertible for $j=1,\ldots, \ell$. Consequently, $\varphi(T)$ not being invertible implies that $\psi(T)$ is not invertible. Hence, $\psi$ cannot be a constant polynomial. We may now proceed as in the case when $\varphi$ has all real zeros to obtain the desired conclusion.
\end{proof}

\begin{lem}
\label{lem:POLYINEQ}
Let $T \in \cB(\cH_n)$ be self-adjoint. Then
\begin{equation}
\label{eq:POLYineq}
\| p(T) \| =  \max_{t \in \sigma_S(T) } |p(t)| \quad \quad {\rm for} \quad p \in \RR[t].
\end{equation}
\end{lem}

\begin{proof}
In view of \eqref{eq:C*identity} and \eqref{eq:SRF}, we have
\begin{align*}
\| p(T) \|^2 =& \; \| p(T)^* \, p(T) \|  = \| p(T)^2 \| = r_S(p(T)^2)\\
=& \; \max \{ | \tau|: \tau \in \sigma_S( p(T)^2 ) \} \\
=& \; \max \{ p(t_0)^2: t_0 \in \sigma_S(T) \},
\end{align*}
i.e., \eqref{eq:POLYineq} holds.
\end{proof}

\begin{thm}[functional calculus for bounded self-adjoint operators]
\label{thm:FC}
Let $T \in \cB(\cH_n)$ be self-adjoint. Then corresponding to any $f,g \in \mathscr{C}(\sigma_S(T), \RR)$, there exist self-adjoint operators $f(T),g(T) \in \cB(\cH_n)$ which exhibit the following properties{\rm :}
\begin{enumerate}
\item[(i)] $(fg)(T) = f(T) g(T)$.
\smallskip
\item[(ii)] $(f+g)(T) = f(T) + g(T)$.
\smallskip
\item[(iii)] $\| f(T) \| = \| f \|_{\infty}$.
\smallskip
\item[(iv)] If $f|_{\sigma_S(T)} \geq 0$, then $f(T)$ is a positive operator.
\end{enumerate}
\end{thm}

\begin{proof}
Fix $f \in \mathscr{C}(\sigma_S(T), \RR)$. By the Weierstra\ss \ approximation theorem, $\RR[t]$ is uniformly dense in $\mathscr{C}(\sigma_S, \RR)$ with supremum norm
$$\| f \|_{\infty} :=  \sup_{t \in \sigma_S(T)} \, |f(t)| \quad \quad {\rm for} \quad f \in \mathscr{C}(\sigma_S(T), \RR).$$
Thus, there exists a sequence of real-valued polynomials $(f_i)_{i=1}^{\infty}$ such that $f$ is the uniform limit of $(f_i)_{i=1}^{\infty}$ on $\sigma_S(T)$, i.e.,
$$\lim_{n\to \infty} \, \max_{t_0 \in \sigma_S(T)} |f(t_0) - f_n(t_0) |.$$
Therefore, $(f_i)_{i=1}^{\infty}$ is a Cauchy sequence and \eqref{eq:POLYineq} applied to $f_i$ implies that
$$\lim_{i,j \to \infty} \| f_i(T) - f_j(T) \| = 0.$$
Since $\cB(\cH_n)$ is complete, as must have that $(f_i(T))_{i=1}^{\infty}$ has a limit in the uniform operator topology, which we will denote by $f(T)$. One can easily check that $f(T)$ does not depend on the choice of the Cauchy sequence $(f_i)_{i=1}^{\infty}$.

Assertions (i)-(iii) are a direct consequence of the definition of $f(T)$ for $f \in \mathscr{C}(\sigma_S(T), \RR)$, given above, just as in the classical complex Hilbert space setting. To prove (iv), observe that if $f |_{\sigma_S(T)} \geq 0$, then there exists $g \in \RR[t]$ such that $f_{\sigma_S(T)} = g^2|_{\sigma_S(T)}$. Thus,
\begin{align*}
\langle f(T)x, x \rangle =& \; \langle g(T)^2 x, x \rangle \\
=& \;  \langle g(T)x, g(T)x \rangle \\
\succeq& \;  0 \quad \quad {\rm for} \quad x \in \cH_n.
\end{align*}
Thus, $f(T)$ is a positive operator.
\end{proof}

We are now ready to prove Theorem \ref{thm:BSA}.

\begin{proof}[Proof of Theorem \ref{thm:BSA}]
Fix $T \in \cB(\cH_n)$ and $x \in \cH_n$. We shall utilise the functional calculus given in Theorem \ref{thm:FC} to make sense of $f(T)$ for $f \in \mathscr{C}(\sigma_S(T), \RR)$. Let $\ell_{x}: \mathscr{C}(\sigma_S(T), \RR) \to \gS(\RR_n)$, where $\gS(\RR_n)$ denotes the real vector space of self-adjoint Clifford number in $\RR_n$, be the linear functional given by
$$\ell_{x}(f) =  \langle f(T) x, x \rangle \quad \quad {\rm for} \quad f \in \mathscr{C}(\sigma_S(T), \RR).$$
The fact that $\ell_x(f) \in \gS(\RR_n)$ follows at once from the fact that $f(T) = f(T)^*$ (see Theorem \ref{thm:FC} and $\ell_x(f) = \langle f(T)x,x\rangle$.

We claim that $\ell_x$ is a positive linear functional. If $f|_{\sigma_S(T)} \geq 0$, then $f|_{\sigma_S(T)} = g^2|_{\sigma_S(T)}$ for some $g \in \mathscr{C}(\sigma_S(T), \RR)$. Thus,
\begin{align*}
\ell_x(f) =& \;  \langle g(T)^2 x, x \rangle \\
=& \;  \, \langle g(T)x, g(T) x \rangle \\
\geq& \; 0 \quad\quad {\rm for} \quad x \in \cH_n,
\end{align*}
in which case, $\ell_x: \mathscr{C}(\sigma_S(T), \RR) \to \gS(\RR_n)$ is a positive linear functional. Consequently,  it follows from Theorem \ref{thm:CliffordRRPOS} that there exists a unique positive $\RR_n$-valued Borel measure $\mu_x$ on $\sB(\sigma_S(T))$ such that
\begin{equation}
\label{eq:ELLx}
\ell_x(f) = \int_{\sigma_S(T)} f(t) \, d\mu_x(t) \quad \quad {\rm for} \quad f \in \mathscr{C}(\sigma_S(T), \RR).
\end{equation}

Fix $x, y \in \cH_n$. Utilising the polarisation formula \eqref{eq:POLAR}, one can check that
\begin{equation}
\label{eq:FULLxy}
\langle f(T)x, y \rangle = \int_{\sigma_S(T)} f(t) \, d\mu_{x,y}(t) \quad \quad {\rm for} \quad f \in \mathscr{C}(\sigma_S(T), \RR),
\end{equation}
where $\mu_{x,y}$ denotes the $\RR_n$-valued measure on $\sB(\sigma_S(T))$ given by
\begin{equation}
\label{eq:fullMU}
\mu_{x,y} := \frac{\sum_{\alpha} (\mu_{x+y e_{\alpha}} - \mu_{x-ye_{\alpha}} )e_{\alpha} }{ 4 \, \dim \mathfrak{S}(\RR_n) } \quad \quad {\rm for} \quad x,y \in \cH_n.
\end{equation}
Moreover, one can use \eqref{eq:fullMU} and the uniqueness of the positive Borel measure in \eqref{eq:ELLx} to show that for any $x, y \in \cH_n$, there is one and only one $\RR_n$-valued measure on $\sB(\sigma_S(T))$ such that \eqref{eq:FULLxy} holds.

An immediate consequence of the uniqueness of $\mu_{x,y}$ in \eqref{eq:FULLxy} is for any $x, y,z \in \cH_n$ and $a,b \in \RR_n$, we have
$$
\mu_{xa + yb, z}(M) = \mu_{x,z}(M) \, a + \mu_{y, z}(M) b
$$
and
$$
\mu_{x, ya + zb}(M) = \bar{a} \mu_{x,y}(M) + \bar{b} \mu_{x,z}(M)
$$
for all $M\in \sB(\sigma_S(T))$. We can also use \eqref{eq:FULLxy} to show that
for any $x, y \in \cH_n$,  we have
\begin{align*}
|\mu_{x,y}(M)| \leq& \; |\mu_{x,y}(\sigma_S(T)) | \\
=& \; | \langle x, y \rangle | \\
\leq& \; \| x \| \| y \| \quad \quad {\rm for} \quad M \in \sB(\sigma_S(T)).
\end{align*}
Thus, for any $M \in \sB(\sigma_S(T))$ we may use Theorem \ref{thm:RRforHM} to obtain an operator $E(M) \in \cB(\cH_n)$ such that
\begin{equation}
\label{eq:Espec}
\langle E(M)x, y \rangle := \mu_{x,y}(M).
\end{equation}

We now wish to show that $E: \sB(\sigma_S(T)) \to \cB(\cH_n)$ is a spectral measure (see Definition \ref{def:SPECTRALMEASURE}). To see that $E$ is a spectral measure, we shall make use of Lemma \ref{lem:2SeptMEASURE}. Since $\langle E(\cdot)x, x \rangle = \mu_x$ is a positive Borel measure (and hence $\mu_x$ is countably additive) for any $x \in \cH_n$, it suffices to show that $E(\sigma_S(T)) = I$, $E(M) = E(M)^*$ and $E(M)^2 = E(M)$ for $M \in \sB(\sigma_S(T))$.

Putting $f(t) = 1$ into \eqref{eq:FULLxy}, we obtain $\langle x, y \rangle = \mu_{x,y}(\sigma_S(T)) = \langle E(\sigma_S(T)x,y \rangle$ for all $x,y \in \cH_n$ and hence $E(\sigma_S(T)) = I$. Next, since
$$\overline{ \langle f(T)x, y \rangle } = \int_{\sigma_S(T)} f(t) d\bar{\mu}_{x,y}(t) \quad \quad {\rm for} \quad f \in \sC(\sigma_S(T), \RR).$$
On the other hand,
\begin{align*}
\overline{ \langle f(T)x, y \rangle } =& \; \langle y, f(T)x \rangle \\
=& \; \langle f(T)y, x \rangle \\
=& \; \int_{\sigma_S(T)} f(t) d\mu_{y,x}(t) \quad \quad {\rm for} \quad f \in \sC(\sigma_S(T), \RR).
\end{align*}
Thus, for any $x, y \in \cH_n$, we have
$$
\int_{\sigma_S(T)} f(t) d\bar{\mu}_{x,y}(t) = \int_{\sigma_S(T)} f(t) d\mu_{y,x}(t) \quad \quad {\rm for} \quad f \in \sC(\sigma_S(T), \RR)
$$
and the aforementioned uniqueness of the $\RR_n$-valued measure $\mu_{x,y}$ such that \eqref{eq:FULLxy} forces $\bar{\mu}_{x,y} = \mu_{y,x}$. Consequently, \eqref{eq:Espec} implies that $E(M) = E(M)^*$ for all $M \in \sB(\sigma_S(T))$.

We will now prove that $E(M)^2 = E(M)$ for all $M \in \sB(\sigma_S(T))$. Using the definition of $f(T)$ for $f \in \sC(\sigma_S(T), \RR)$, one can extend the identity in Theorem \ref{thm:FC}(i) to the case when $f, g\in \sC(\sigma_S(T), \RR)$. Thus, for any $x, y \in \cH_n$, we have
$$\langle f(T) g(T)x, y \rangle = \langle (fg)(T)x, y \rangle = \int_{\sigma_S(T)} f(t) g(t) d\mu_{x,y}(t) $$
for all $f,g \in \sC(\sigma_S(T), \RR).$ On the other hand,
$$\langle f(T) g(T)x, y \rangle = \int_{\sigma_S(T)} f(t) d\mu_{g(T)x, y}  \quad \quad {\rm for} \quad f,g \in \sC(\sigma_S(T), \RR).$$
Thus, using the uniqueness of the $\RR_n$-valued measure $\mu_{x,y}$ in \eqref{eq:FULLxy}, we have 	$d\mu_{g(T)x, y } = g\, d\mu_{x,y}$ for all $x, y \in \cH_n$. Consequently,  for any $x,y \in \cH_n$, we have $\langle E(M) g(T)x, y \rangle = \int_M g(t) d\mu_{x,y}$ for all $M \in \sB(\sigma_S(T))$. Putting these observations together, for any $x, y \in \cH_n$, we have
$$\int_{\sigma_S(T)} g(t) d\mu_{x, E(M)y} = \int_{\sigma_S(T)} g(t) \chi_M(t) d\mu_{x,y}(t)$$
for all $g \in \sC(\sigma_S(T), \RR)$. Thus, using the uniqueness of the $\RR_n$-valued measure in \eqref{eq:FULLxy}, we have $d\mu_{x,E(M)y} = \chi_M d\mu_{x,y}$, i.e.,
$$\mu_{x,E(M)y}(N) = \int_N \chi_M(t) d\mu_{x,y}(t) = \mu_{x,y}(M \cap N) \quad \quad {\rm for} \quad N \in \sB(\sigma_S(T)),$$
i.e., $\langle E(N)x, E(M)y \rangle = \langle E(M \cap N) x, y \rangle$. Using the fact that $E(M) = E(M)^*$, we arrive at $\langle E(M)E(N)x, y \rangle = \langle E(M \cap N)x, y \rangle$ for all $M, N \in \sB(\sigma_S(T))$. Setting $N = M$, we get $E(M)^2 = E(M)$ for all $M \in \sB(\sigma_S(T))$.  Thus, $E$ is a spectral measure on $\sigma_S(T)$.

Formula \eqref{eq:SPECSA} follows at once from \eqref{eq:FULLxy} with $f(t) = t$ together with \eqref{eq:Espec}.

Let $F$ be another spectral measure on $\sB(\RR)$ such that $\eqref{eq:SPECSA}$ holds with $F$ in place of $E$. We shall let $\II_F(f)$ denote the spectral integral for
$f \in \gF(\sB(\RR), F)$ (see Theorem \ref{thm:UNBFUNCCALC}). Let $g: \sigma_S(T) \to \RR$ given by $g(t) = t$. Then
$$T = \II_F(g) = \int_{\sB(\RR)} t \, dF(t).$$
Theorem \ref{thm:IMPORTANT} asserts that
\begin{align*}
\sigma_S(\II(g)) =& \; \{ s \in \RR: F( \{ t \in \sB(\RR): |g(t)^2 - 2 {\rm Re}(s)g(t) + |s|^2 | < \varepsilon \} ) \neq 0 \\
& \;\; \; \text{ for all $\varepsilon > 0$}\} \\
=& \; \{ s \in \RR: F( \{ t \in \sB(\RR): |t^2 -2 st + s^2 | < \varepsilon \} ) \neq 0 \\
& \;\; \; \text{ for all $\varepsilon > 0$}\} \\
=& \; \{ s \in \RR: F( \{ t \in \sB(\RR): |t-s|^2 < \varepsilon \} ) \neq 0 \\
& \;\; \; \text{ for all $\varepsilon > 0$}\} \\
=& \; \supp F,
\end{align*}
where we used \eqref{eq:SPRTALT} to obtain the last equality. Thus, as $\sigma_S(T) = \sigma_S(\II_F(g))$, we have $\sigma_S(T) = \supp F$.  On the other hand,  let $\II(g)$ denote the spectral integral of $g \in \gF(\sigma_S(T), E)$. Then, using the argument above, with $F$ replaced by $E$, we obtain $\sigma_S(T) = \supp E$ and hence $\sigma_S(T) =\supp E = \supp F$. Using Theorem \ref{thm:BFUNC}(v), we obtain for any $x \in \cH_n$,
$$\int_{\sigma_S(T)} 	|f(t)|^2 d( {\rm Re} \langle E(t)x, x \rangle ) = \int_{\sigma_S(T)} 	|f(t)|^2 d( {\rm Re} \langle F(t)x, x \rangle )$$
for every $f \in \gB$ and, in particular, for every $f \in \sC(\sigma_S(T), \RR)$. Thus, Theorem \ref{thm:RR} forces $\langle E(M)x, x \rangle = \langle F(M)x, x \rangle$ for all $M \in \sB(\sigma_S(T))$. Consequently, the polarisation formula \eqref{eq:POLAR} forces $E =F|_{\sigma_S(T)}$.

Finally, suppose $WT = WT$ for $W \in \cB(\cH_n)$. Using Theorem \ref{thm:BFUNC}(iv), we have for any $x, y \in \cH_n$,
$$\langle f(T)Wx, y \rangle = \int_{\sigma_S(T)} {\rm Re}(f(t)) d\langle E(t) Wx, y \rangle \quad \quad {\rm for} \quad f \in \gB.$$
In view of Definition \ref{def:BFUNC}, we have that $Wf(T) = f(T) W$ for every \\$f \in \gB$ and hence, for any $x,y \in \cH_n$,
$$\langle f(T)Wx, y \rangle = \langle f(T)x, W^*y \rangle = \int_{\sigma_S(T)} {\rm Re}(f(t)) d\langle E(t)x, W^*y \rangle \quad  {\rm for} \; f \in \gB.$$
Putting these observations together, we have, in particular,
$$
\int_{\sigma_S(T)} f(t) d\langle E(t) Wx, y \rangle = \int_{\sigma_S(T)} f(t) d\langle E(t) x,W^* y \rangle \quad  {\rm for} \; f \in \mathscr{C}(\sigma_S(T), \RR).
$$
Thus, using Corollary \ref{cor:RR}, we have $\langle E(M)Wx,y \rangle = \langle E(M)x, W^*y$, i.e., $\langle E(M)Wx, y \rangle = \langle WE(M)x, y \rangle$ for every $x,y \in \cH_n$ and $M \in \sB(\sigma_S(T))$, i.e., $E(M)W = WE(M)$ for every $M \in \sB(\sigma_S(T))$.

\end{proof}

\begin{rem}
\label{rem:SUPP}
Let $T$ and $E$ be as in Theorem \ref{thm:BSA}. We wish to record that in the proof of Theorem \ref{thm:BSA} we showed that the support of the spectral measure of $A$ is precisely the $S$-spectrum of $T$, i.e.,
$$\supp E = \sigma_S(T).$$
\end{rem}

\begin{cor}
\label{cor:POSSQROOT}

Let $T \in \cB(\cH_n)$ be a positive operator. Then there exists a unique positive operator $A \in \cB(\cH_n)$ such that $A^2 = T$. Moreover, if $(f_j)_{j=1}^{\infty}$ is any sequence of polynomials such that $f_j(t) \to t^{1/2}$ uniformly on $[0,d]$, where $d = \max \sigma_S(T)$, then
\begin{equation}
\label{eq:NORMOP}
\lim_{j \to \infty} \| f_j(T) - T^{1/2} \| = 0. \end{equation}
\end{cor}

\begin{proof} The spectral theorem for a bounded self-adjoint operator (see Theorem \ref{thm:BSA} asserts that there exists a uniquely determined spectral measure $E$ on $\sigma_S(T)$ such that $T = \int_{\sigma_S(T)} t \, dE(t)$. In view of Lemma \ref{lem:15July1pos}, we have $\sigma_S(T) \subseteq [0, d]$, where $d \geq 0$. Notice that have that $f(t) = t^{1/2} \in \gB(\sigma_S(T), \sB(\sigma_S(T)), \CC_{\gI})$ (see Definition \ref{def:BFUNC}), where $\gI \in \mathbb{S}$, and $$A := f(T) = \int_{\sigma_S(T)} t^{1/2} \, dE(t).$$ Thus,  $$\langle Ax, x \rangle = \int_{[0,d]} t^{1/2} d\langle E(t)x, x \rangle \succeq 0 \quad \quad {\rm for} \quad x \in \cH_n$$ and hence $A$ is a positive operator and we may use Theorem \ref{thm:BFUNC}(i) to check that $A^2 = T$. The limit \eqref{eq:NORMOP} is an immediate consequence of Theorem \ref{thm:BFUNC}(vii). We will now prove that there is only one positive operator $W \in \cB(\cH_n)$ such that $W^2 = T$. Suppose $\wt{A} \in \cB(\cH_n)$ is a positive operator such that $\wt{A}^2 = T$.  Then it is easy to see that $\wt{A}T = T \wt{A}$ and hence $\wt{A} g(T) = g(T) \wt{A}$ for any $g \in \RR[t]$. Consequently, we have  $$\wt{A} \left( \lim_{j \to \infty} f_j(T) \right)= \left(\lim_{j \to \infty} f_j(T)\right) \wt{A},$$ where $(f_j)_{j=0}^{\infty}$ is any sequence of polynomials in $\RR[t]$ such that $f_j(t) \to t^{1/2}$ uniformly on $[0, d ]$, i.e.,  \begin{equation} \label{eq:AtildeAcommut} \wt{A} A = A \tilde{A}. \end{equation} For any $x \in \cH_n$, let $y := (A - \wt{A})x$. Thus, using \eqref{eq:AtildeAcommut}, we obtain \begin{equation} \label{eq:INTMED} \langle A y, y\rangle + \langle \wt{A}y, y \rangle = \langle (A+\wt{A})(A-\wt{A})x, y \rangle = \langle (A^2 - \wt{A}^2)x, y \rangle = 0. \end{equation} Therefore, as $\langle Ay, y \rangle \succeq 0$ and $\langle \wt{A} y, y \rangle \succeq 0$, we must have
\begin{equation}
\label{eq:AtildeAy}
\langle Ay,y \rangle = \langle \wt{A}y, y \rangle = 0.
\end{equation}
But then, as $A = D^*D$ for some $D \in \cB(\cH_n)$ and $\wt{A} = \wt{D}^* \wt{D}$ for some $\wt{D} \in \cB(\cH_n)$,  we have that \eqref{eq:AtildeAy} implies that $\langle Dy, Dy \rangle = \langle \wt{D}y, \wt{D}y \rangle =0 \in \RR_n$, i.e.,  $${\rm Re}\, \langle Dy, Dy \rangle = {\rm Re}\, \langle \wt{D}y, \wt{D}y \rangle =0.$$
Thus, $\| Dy \| = \| Dy \| = 0$, i.e., $Dy = \wt{D} y = 0$. But then we have
$$Ay = D^* D y = \wt{D}^* D = \wt{A} y = 0.$$
Therefore, for any $x \in \cH_n$, we have
\begin{align*}
\| (A - \wt{A})x \|^2 =& \; {\rm Re}\, \langle (A-\wt{A})x, (A - \wt{A})x \rangle \\
=& \; {\rm Re}\, \langle (A-\wt{A})^2 x, x \rangle \\
=& \; {\rm Re}\, \langle (A-\wt{A})y, x \rangle \\
=& \; 0
\end{align*}
and hence $(A-\wt{A})x =0$ for every $x \in \cH_n$, i.e., $A = \wt{A}$.
\end{proof}

\begin{cor}
\label{cor:WCOMM}
Suppose $A \in \cB(\cH_n)$ and $B \in \cB(\cH_n)$ are self-adjoint operators with spectral measures $E_A$ and $E_B$, respectively.  Then $AB = BA$ if and only if $E_A(M) E_B(M) = E_B(M) E_A(M)$ for all $M \in \sB(\RR)$.
\end{cor}

\begin{proof}
The last assertion in Theorem \ref{thm:BSA} with $T= A $ and $W = B$ ensures that $AB = BA$ if and only if $B E_A(M) = E_A(M) B$. Applying the last assertion in Theorem \ref{thm:BSA} with $T = B$ and $W = E_B(M)$ we see that
$$AB = BA \Longleftrightarrow B E_A(M) = E_A(M)B \Longleftrightarrow E_A(M) E_B(M) = E_B(M)E_A(M)$$
for any $M \in \sB(\RR)$.
\end{proof}

\section{Polar decomposition for bounded operators}
\label{sec:PDBOs}

Let $T \in \cB(\cH_n)$. The aim of this section it show that there exists a uniquely determined positive operator $P \in\cB(\cH_n)$ and a partial isometry $Q$ such that $T = UQ$. This will be important for the spectral theorem for a normal operator in Section \ref{sec:BN}.

\begin{thm}[polar decomposition for bounded operators]
\label{thm:POLARDECOMP}
Every $T \in \cB(\cH_n)$ admits a factorisation
\begin{equation}
\label{eq:POLARDECOMP}
T = U Q,
\end{equation}
where $Q := |T|$ is uniquely determined, where $|T| := (T^*T)^{1/2}$, and $U: \overline{ {\rm Ran}(Q) } \to \overline{ {\rm Ran}(T)}$ is a partial isometry. In the particular case that $T \in \cB(\cH_n)$ is normal, we can choose $U$ such that $U$ is unitary and we have the following{\rm :}
\begin{equation}
\label{eq:COMMUTINGWTQ}
WT = TW \; {\rm and} \; WT^* = T^* W \Longrightarrow WQ = Q W \; {\rm for} \; W \in \cB(\cH_n)
\end{equation}
and
\begin{equation}
\label{eq:COMMUTINGWTU}
WT = TW \; {\rm and} \; WT^* = T^* W \Longrightarrow WU = U W \; {\rm for} \; W \in \cB(\cH_n).
\end{equation}
Moreover, $T$ is normal if and only if $QU  =U Q$.
\end{thm}

\begin{proof}
Note that $T^*T \in \cB(\cH_n)$ is a positive operator and has a unique positive square root $Q := (T^*T)^{1/2} \in \cB(\cH_n)$ (see Corollary \ref{cor:POSSQROOT}). For any $x \in \cH_n$, we have
\begin{align}
\| T x \|^2 =& \; {\rm Re}\, \langle Tx, Tx \rangle \nonumber \\
=& \; {\rm Re}\, \langle T^* Tx, x \rangle \nonumber \\
=& \; {\rm Re}\, \langle Q^2 x ,x \rangle \nonumber \\
=& \; {\rm Re}\, \langle Qx, Qx \rangle \nonumber \\
=& \; \| Q x \|^2. \label{eq:TQ}
\end{align}
Thus, if we let $x = y -z$, with $y,z \in \cH_n$, then \eqref{eq:TQ} implies that $Ty = Tz$ whenever $Qy = Qz$ for any $y,z \in \cH_n$. Consequently, we may let $U: \Ran Q \to  \Ran T$ denote the operator belonging to $\cB(\cH_n)$ given by $U(Qx) = T x$ for $x \in \cH_n$. Next, extend $U$ to all of $\cH_n$ via
\begin{equation}
\label{eq:PERP}
Uy = \begin{cases}
T x & \quad {\rm if} \quad  y = Qx, \\
0 & \quad {\rm if} \quad y   \in \overline{\Ran \, Q}^{\perp}.
\end{cases}
\end{equation}
Since $\langle Uy, z \rangle = \langle y, U^* z \rangle = 0$ for all $y \in \Ran \, Q$ and $z \in \cH_n$, we have $U^*: \cH_n \to \overline{ \Ran \,Q}^{\perp}$ and hence $\Ran \, U \subseteq \overline{\Ran \,Q}$.

We now wish to show that
\begin{equation}
\label{eq:URAN}
U^*U x = x \quad \quad {\rm for} \quad x \in \overline{\Ran\, Q}.
\end{equation}
Suppose $y, z \in \Ran Q$. Then
$$\langle y, z \rangle = \langle Uy, Uz \rangle = \langle y, U^* U z \rangle$$
and hence $\langle U^* Uy -y, w\rangle = 0$ for all $w \in \Ran Q$. We have already noted that $U^*: \cH_n \to \overline{\Ran Q}$. Hence for all $y \in \overline{\Ran Q}$, we have $U^*U y - y \in \overline{\Ran Q}$. Consequently, we have \eqref{eq:URAN}. Putting all of these observations together, we arrive at the factorisation \eqref{eq:POLARDECOMP}. The asserted uniqueness can be justified precisely the same way as the complex Hilbert space case.

Let us now suppose that $T$ is normal. We first note that $\langle TT^*x, x\rangle = \langle T^*Tx, x \rangle$ implies that $\langle T x, Tx \rangle = \langle T^* x, T*x \rangle$ and hence ${\rm Re}\, \langle T x, Tx \rangle = {\rm Re}\, \langle T^* x, T*x \rangle$, i.e., $\| T x \| = \| T^* x \|$ for all $x \in \cH_n$. Therefore, $\Ker\, T = \Ker\, T^*$ and hence
$$\Ran \,Q^{\perp} = \Ker \,Q = \Ker \,T = \Ker \,T^* = \overline{\Ran \,T}^{\perp}.$$
Thus, we can extend $U: \overline{\Ran Q} \to \overline{\Ran Q}$ from \eqref{eq:PERP} to $U: \cH_n \to \cH_n$ (with a slight abuse of notation we shall denote the extension of $U$ by $U$ as well) via
$$Uy = y\quad  {\rm if} \quad y \in \overline{\Ran Q}^{\perp}.$$
As $\cH_n = \overline{\Ran Q} \oplus \overline{\Ran Q}^{\perp}$, we have that $U: \cH_n \to \cH_n$ is unitary.

The implication \ref{eq:COMMUTINGWTQ} is an immediate consequence of the final assertion of Theorem \ref{thm:BSA} and Corollary \ref{cor:POSSQROOT}. We will now show \eqref{eq:COMMUTINGWTU}. In view of \eqref{eq:COMMUTINGWTQ}, we have
\begin{equation}
\label{eq:NULL}
(WU - UW)Q = 0.
\end{equation}
Thus, as $T \in \cB(\cH_n)$ is normal, we have $\cH_n = {\rm Ker} T \oplus \overline{ \Ran T }$. Thus we may write any $x \in \cH_n$ as $x = y + z $, where $y \in \Ker T = \Ker Q$ and $z \in \overline{ \Ran T}$, we may use the fact that $\Ker Q = \Ker T$ and \eqref{eq:NULL} to obtain
$$(WU - UW)x = (WU - UW)(y+z) = 0,$$
i.e., $WU = UW$.

The final assertion can be justified in the same way as the complex Hilbert space case (bearing in mind that one must use Corollary \ref{cor:POSSQROOT} which guarantees that bounded positive operators have a unique positive square root).
\end{proof}

\section{An additive decomposition for bounded operators and imaginary operators}
\label{sec:ADDandIOs}
In this section, we will show that all bounded operators on a Clifford module admit an additive decomposition which is analogous to the well-known additive decomposition
$$T = \left( \frac{T + T^*}{2} \right) + i \left( \frac{T - T^*}{2i} \right),$$
which holds in the complex Hilbert space case. In the particular case that the bounded operator is normal, this additive decomposition will be useful for proving the spectral theorem for a bounded normal operators in Section \ref{sec:BN}.

\begin{lem}
\label{lem:UNITARY}
Let $T \in \cB(\cH_n)$ be unitary. Then
\begin{equation}
\label{eq:UNITARY}
\sigma_S(T) \subseteq \{  s \in \RR^{n+1}: |s| = 1 \}.
\end{equation}
\end{lem}

\begin{proof}
We claim that $\rho_S(T) \subseteq \RR^{n+1} \setminus \{  s \in \RR^{n+1}: |s| = 1 \}$. Since $TT^* = I$, it is obvious that $\| T \| = 1$. Thus, Theorem \ref{thm:SPECNONEMPTYCOMPACT} ensures that
$$\{  s \in \RR^{n+1}: |s| > 1 \} \subseteq \rho_S(T).$$
Since $\cQ_0(T) = T^2$ is a unitary operator, we have that $\cQ_0(T)$ is invertible, i.e., $0 \in \rho_S(T)$. For any $s = s_0 + s_1 \in \RR^{n+1}$, with $0 < |s| < 1$, we have
\begin{align*}
\cQ_s(T) =& \; T^2 - 2 {\rm Re}(s) T + |s|^2\, I \\
=& \;(|s|^2)  T^2 \{  (T^*)^2 + 2{\rm Re}(s^{-1}) T^* + |s^{-1}|^2 T^2    \} \\
=& \; |s|^2 T^2 \cQ_{s^{-1}}(T^*).
\end{align*}
Thus, as $T^*$ is unitary and $|s^{-1}| > 1$, we have that $s^{-1}$ belongs to $\rho_S(T^*)$, in which case $s \in \RR^{n+1}$, with $0 < |s| < 1$, belongs to $\rho_S(T)$. Putting everything together we have \eqref{eq:UNITARY}.
\end{proof}

\begin{defn}[imaginary operator]
\label{def:IMAGINARY}
We will call an operator $J_0 \in \cB(\cH_n)$ a {\it partial imaginary operator} if $J_0$ is a partial isometry and $J_0^* = - J_0$. We will call $J \in \cB(\cH_n)$ an {\it imaginary operator} if $J$ is unitary and $J^* = -J$.
\end{defn}

\begin{thm}
\label{thm:AJBgeneral}
Corresponding to any operator $T\in \cB(\cH_n)$, there exist a self-adjoint operator $A \in \cB(\cH_n)$, a partial imaginary operator $J_0 \in \cB(\cH_n)$ and a positive operator $B \in \cB(\cH_n)$ such that
\begin{equation}
\label{eq:AJBgeneral}
T = A  +J_0 B,
\end{equation}
where $A := (T+T^*)/2$ and $B := |T-T^*|/2$ are uniquely determined by $T$. Moreover, we may choose $J_0$ in \eqref{eq:AJBgeneral} to be an imaginary operator.
\end{thm}

\begin{proof}
If we let $A := (T+T^*)/2$, then $A \in \cB(\cH_n)$ is obviously self-adjoint. Consider the anti self-adjoint operator  $Y := T - A = (T-T^*)/2 \in \cB(\cH_n)$.
Since $Y$ is normal, we may use Theorem \ref{thm:POLARDECOMP} to find a positive operator $B \in \cB(\cH_n)$ and a partial isometry $J_0$ such that $Y = J_0 B $. Moreover, by Theorem \ref{thm:POLARDECOMP}, we have $BJ_0 = J_0 B$. Since $Y$ is anti self-adjoint, we must have $J_0^* = -J_0$. Thus, $J_0$ is a partial imaginary operator.

The uniqueness of $A$ is obvious. The uniqueness of $B$ follows from Theorem \ref{thm:POLARDECOMP}. Finally, the fact that we may choose $J_0$ to be unitary follows from Theorem \ref{thm:POLARDECOMP} applied to the bounded normal operator $(T-T^*)/2$.
\end{proof}

In the case that $T \in \cB(\cH_n)$ is normal, we have the following refinement of Theorem \ref{thm:AJBgeneral}.

\begin{thm}
\label{thm:AJB}
Corresponding to any normal operator $T\in \cB(\cH_n)$, there exist a self-adjoint operator $A \in \cB(\cH_n)$, a partial imaginary operator $J_0 \in \cB(\cH_n)$ and a positive operator $B \in \cB(\cH_n)$ such that $A$, $J_0$ and $B$ mutually commute and satisfy
\begin{equation}
\label{eq:AJB}
T = A  +J_0 B.
\end{equation}
In this case, $A$ and $B$ are as in Theorem \ref{thm:AJBgeneral} and
\begin{equation}
\label{eq:TJ}
TJ_0 = J_0 T.
\end{equation}
Moreover, we may choose $J_0$ to be an imaginary operator {\rm (}in this case we shall write $J$ in place of $J_0${\rm )}.
\end{thm}

\begin{proof}
In view of Theorem \ref{thm:AJB}, we only have to show that $A:=(T+T^*)/2, B:=|T-T^*|/2$ and $J_0$ mutually commute and that \eqref{eq:TJ} holds. The fact that $B$ and $J_0$ commute follows from the fact that $J_0 B$ is a polar factorisation for the bounded normal operator $Y := (T-T^*)/2$. Since $T$ is normal, we have $A (T-T^*) =  (T-T^*) A$ and $A (T-T^*)^* = (T-T^*)^* A$. Thus, the fact that $A$, $B$ and $J_0$ mutually commute follows at once from \eqref{eq:COMMUTINGWTQ} and \eqref{eq:COMMUTINGWTU}. The final assertion \eqref{eq:TJ} is an immediate consequence of \eqref{eq:AJB} and the fact that $A$, $J$ and $B$ mutually commute.
\end{proof}

\begin{thm}
\label{thm:IOs}
Let $J$ be an imaginary operator on a Clifford module $\cH_n$, fix $\gI \in \mathbb{S}$ and define
\begin{equation}
\label{eq:CI}
\cH_{\pm}(J, \gI) := \{ x \in \cH_n: J\, x = x(\pm \gI) \}.
\end{equation}
Then the following statements hold{\rm :}
\begin{enumerate}
\item[(i)] $\cH_{\pm}(J, \gI)$ are nontrivial, i.e., $\cH_{\pm}(J, \gI) \neq \{ 0 \}$ and
\begin{equation}
\label{eq:TRIVIALINTER}
\cH_{+}(J, \gI) \cap \cH_{-}(J, \gI) = \{ 0 \}.
\end{equation}
\item[(ii)] $\cH_{\pm}(J, \gI)$ are $\CC_{\gI}$ closed right linear subspaces of $\cH_n$, with respect to $\CC_{\gI}$, i.e., if $x, y \in \cH_{\pm}(J, \gI)$ and $\lambda \in \CC_{\gI}$, then $x\,\lambda + y \in \cH_{\pm}(J, \gI)$. Consequently, $\cH_{\pm}(J, \gI)$ may be both be viewed as a complex Hilbert space with respect to the $\CC_{\gI}$-valued inner product given by
\begin{equation}
\label{eq:CC_gI}
\langle x, y \rangle_{\CC_{\gI}} := {\rm Re}(\langle x, y \rangle) - {\rm Re}(\langle x, y \rangle \gI)\gI \in \CC_{\gI} \quad \quad {\rm for} \quad x, y \in \cH_{\pm}(J, \gI).
\end{equation}
\item[(iii)] $\cH_n = \cH_{+}(J, \gI) \oplus \cH_{-}(J, \gI)$.
\smallskip
\item[(iv)] For any orthonormal basis $(\eta_i)_{i \in \cI}$ of the complex Hilbert space $\cH_+(J, \gI)$, we have that, for any choice of $\gJ \in \mathbb{S}$ with $\gI \, \gJ = -\gJ \gI$,  $(\eta_i \, \gJ)_{i \in \cI}$ is an orthonormal basis of $\cH_-(J, \gI)$.

     Moreover, for any orthonormal basis $(\tilde{\eta})_{i \in \cI}$ of the complex Hilbert space $\cH_-(J, \gI)$, we have that, for any choice of $\gJ \in \mathbb{S}$ with $\gI \, \gJ = -\gJ \gI$,  $(\eta_i \, \gJ)_{i \in \cI}$ is an orthonormal basis of $\cH_+(J, \gI)$.
    \smallskip
\item[(v)] For any orthonormal basis $(\eta)_{i \in \cI}$ of the complex Hilbert space $\cH_+(J, \gI)$, $(\eta_i)_{i \in \cI}$ is an orthonormal basis of the Clifford module $\cH_n$.
\smallskip
\item[(vi)] $\sigma_S(J) \cap \CC_{\gI}^+ = \{ \gI \}$.
\smallskip
\item[(vii)] For any orthonormal basis $(\eta_i)_{i \in \cI}$ of the complex Hilbert space $\cH_+(J, \gI)$ and $\gI \in \mathbb{S}$, we have
\begin{equation}
\label{eq:Jexpansion}
Jx = \sum_{i \in \cI} \eta_i \, \gI \, \langle x, \eta_i \rangle \quad \quad {\rm for} \quad x \in \cH_n.
\end{equation}
\end{enumerate}
\end{thm}

\begin{proof}
We will first prove (i). Suppose there exists $x \in \cH_n \setminus \{ 0 \}$ such that $x - Jx \gI \neq 0$. Then
\begin{align*}
J(x - xJx \gI) =& \; Jx + x \gI \\
=& \; (x - Jx \gI) \gI.
\end{align*}
Thus, $y := x - Jx \gI \in \cH_n \setminus \{ 0\}$ and $Jy = y \gI$ and hence $\cH_{+}(J, \gI) \neq \{ 0\}$.

On the other hand, if there exists $x \in \cH_n \setminus \{ 0 \}$ such that $x = Jx \gI$, then choose $\gJ \in \mathbb{S}$ such that $\gI \gJ =  -\gJ \gI$. Thus,
$Jx \gI \gJ = x \gJ$ implies that $-J(x \gJ) \gI = x \gJ$ and hence
$$J(x\gJ) = (x\gJ) \gI.$$
Thus, we again arrive at the conclusion $\cH_{+}(J, \gI) \neq \{ 0\}$. The justification that $\cH_{-}(J, \gI) \neq \{ 0 \}$ can be carried out in a similar fashion.

We will now show \eqref{eq:TRIVIALINTER}. Suppose to the contrary that there exists a nonzero vector $x \in \cH_{+}(J, \gI) \cap \cH_-(J, \gI)$. Then we have $x \gI = x(-\gI)$, i.e., $x \gI = 0$ which forces $x = 0$, a contradiction.

We will now prove (ii). Suppose $x,y \in \cH_+(J, \gI)$ and $\lambda \in \CC_{\gI}$. Then
$$J(x \lambda+ y) = x \gI \lambda+y\gI = (x \lambda+y) \gI.$$
Thus, $\cH_{+}(J, \gI)$ is a $\CC_{\gI}$ right linear subspace of $\cH_n$. The fact that $\cH_+(J, \gI)$ is closed follows at once via the continuity of multiplication. The justification that $\cH_{-}(J, \gI)$ is a closed right-linear subspace, with respect to $\CC_{\gI}$, can be carried out in a similar fashion. Finally, the fact that $\cH_{+}(J, \gI)$ and $\cH_-(J, \gI)$ may be viewed as a complex Hilbert spaces with respect to the $\CC_{\gI}$-valued inner product given by \eqref{eq:CC_gI} follows immediately from $\cH_{\pm}(J, \gI)$ being closed right subspaces of $\cH_n$ and checking that \eqref{eq:CC_gI} is an $\CC_{\gI}$-valued inner product (to this end, it will be helpful to use the easily verified fact that $\langle x, y \rangle \lambda = \pm \lambda \langle x, y \rangle$ for all $\lambda \in \CC_{\gI}$ and $x, y \in \cH_{\pm}(J, \gI)$).

We will now prove (iii). Let $y = (x-Jx\gI)/2$ and $z = (x+Jx \gI)/2$. Then $y \in \cH_+(J, \gI)$ since
$$Jy = \frac{1}{2}(Jx+x\gI) = \frac{1}{2} (x - Jx \gI)\gI.$$
Similarly, $z \in \cH_-(J, \gI)$ since
$$Jz = \frac{1}{2}(Jx - x \gI) = \frac{1}{2} (x + Jx \gI) (-\gI).$$
Thus, as $x = y + z$ and \eqref{eq:TRIVIALINTER} holds we have $\cH_n = \cH_+(J, \gI)\oplus \cH_-(J, \gI)$.

We will now prove (iv). Let $\gJ \in \mathbb{S}$ be such that $\gI \gJ = -\gJ \gI$. Then for any $y \in \cH_+(J, \gI)$, we have
$$J y \, \gJ = (y \gI)\gJ = (\eta_i \gJ)(-\gI).$$
Thus, $y \gJ \in \cH_-(J, \gI)$. One can easily push this observation further  and establish that $\varphi: \cH_+(J,\gI) \to \cH_-(J, \gI)$ given by $\varphi(y) = y \gJ$ is an isomorphism. Consequently, if $(\eta_i)_{i \in \cI}$ is an orthonormal basis of $\cH_+(J, \gI)$, then $(\eta_i \, \gJ)_{i \in \cI}$ is an orthonormal basis of $\cH_-(J, \gI)$.

We will now prove (v). Let $(\eta_i )_{i \in \cI}$ be any orthonormal basis of $\cH_+(J, \cI)$. Then we may apply (iv) and (iii) to see that $(\eta_i)_{i \in \cI}$ is also an orthonormal basis of $\cH_n$.

We will now prove (vi). Recall that $J^* = - J$ and $JJ^* = I$. Thus,  combining \eqref{eq:15Julye1ASA} and \eqref{eq:UNITARY}, we have
$$\sigma_S(J) \cap \CC_{\gI}^+ = \{ \lambda \in \CC_{\gI}^+: {\rm Re}(\lambda) = 0 \} \cap \{ \lambda \in \CC_{\gI}^+: |\lambda|=1 \} = \{ \gI \}.$$

We will now prove (vii). Let $(\eta_i )_{i \in \cI}$ be any orthonormal basis of $\cH_+(J, \cI)$. In view of (v), we have that $(\eta_i)_{i \in \cI}$ is an orthonormal basis of $\cH_n$. Consequently, in view of Lemma \ref{lem:KAPPA}(iv), we have
$$x = \sum_{i \in \cI} \eta_i \langle x, \eta_i \rangle \quad\quad {\rm for}\quad x \in \cH_n$$
and hence
$$
Jx = \sum_{i \in \cI}  J \, \eta_i \langle x, \eta_i \rangle
= \sum_{i \in \cI} \eta_i \, \gI \langle x, \eta_i \rangle\quad\quad {\rm for}\quad x \in \cH_n.$$
\end{proof}

\section{Spectral theorem for a bounded normal operator and some consequences}
\label{sec:BN}

In this section we will be prove one of the main results of this manuscript, namely the spectral theorem for a bounded normal operator on a Clifford module.

Let $\gI \in \mathbb{S}$ be arbitrary. In the following lemma, we shall identify $\RR \times [0, \infty)$ with $\CC^+_{\gI} := \{ \lambda \in \CC^+_{\gI}: {\rm Im}(\lambda) \geq 0\}$ in the natural way. Consequently, a spectral measure on $\sB(\RR \times [0,\infty))$ may be viewed as a spectral measure on $\sB(\CC^+_{\gI})$.

\begin{lem}
\label{lem:INTEGRAL}
Let $E_1$ be a spectral measure on $\sB(\RR)$, $E_2$ be a spectral measure on $\sB([0, \infty))$ and $\gI \in \mathbb{S}$. Suppose $E_1(M_1) E_2(M_2) = E_2(M_2) E_1(M_1)$ for every $M_1 \in \sB(\RR)$ and $M_2 \in \sB([0,\infty))$ and let $E(M_1 \times M_2) := E(M_1)E(M_2)$ be the uniquely determined spectral measure on $\sB(\CC_{\gI}^+)$ (see Theorem \ref{thm:COMMUTING}). For any real-valued function $f \in \gF(\sB(\CC_{\gI}^+, E)$, we have
\begin{equation}
\label{eq:F1}
\int_{\RR} f(t) \, dE_1(t) = \int_{\CC_{\gI}^+ } f( {\rm Re}(\lambda)) dE(\lambda)
\end{equation}
and
\begin{equation}
\label{eq:F2}
\int_{[0,\infty)} g(u) \, dE_2(u) = \int_{\CC_{\gI}^+} g({\rm Im}(\lambda)) dE(\lambda)
\end{equation}
\end{lem}

\begin{proof}
We will verify \eqref{eq:F1} for simple functions. For any $M \in \sB(\RR)$, we have
\begin{align*}
\int_{\CC_{\gI}^+} \chi_M({\rm Re}(\lambda)) \, dE(\lambda) =& \; E(M \times [0, \infty)) \\
=& \; E_1(M) E_2([0, \infty)) = E_1(M) \\
=& \; \int_{\RR} \chi_M( t) dE_1( t )
\end{align*}
Thus, by linearity, \eqref{eq:F1} holds for simple functions. Passing to a limit, we obtain \eqref{eq:F1} for any $f \in \gF(\sB(\CC_{\gI}^+), E)$.

Formula \eqref{eq:F2} can be justified in much the same way as \eqref{eq:F1}.

\end{proof}

We are now ready to state and proof the spectral theorem for a bounded normal operator on a Clifford module.

\begin{thm}[spectral theorem for bounded normal operators]
\label{thm:BN}
Let $T \in \cB(\cH_n)$ be normal. Then corresponding to any choice of $\gI \in \mathbb{S}$, there exists a spectral measure $E$ on the Borel $\sigma$-algebra $\mathscr{B}(\sigma_S(T) \cap \CC^+_{\gI})$ such that
\begin{equation}
\label{eq:SPECN}
T= \int_{\sigma_S(T) \, \cap \, \CC_{\gI}^+} {\rm Re}(\lambda) \, dE(\lambda) + \int_{\sigma_S(T)\, \cap \, \CC_{\gI}^+} {\rm Im}(\lambda) \, dE(\lambda) \, J,
\end{equation}
where $J$ an anti self-adjoint and unitary operator obeying \eqref{eq:AJB}. $E$ is unique in the sense that if $F$ is a spectral measure on $\mathscr{B}(\CC^+_{\gI})$ such that
$$
T = \int_{\CC^+_{\gI}} {\rm Re}(\lambda) \, dF(\lambda) + \int_{\CC^+_{\gI}} {\rm Im}(\lambda) dF(\lambda) J,
$$
 then $E(M \cap \sigma_S(T)\cap \CC^+_{\gI}) = F(M)$ for all $M \in \mathscr{B}(\CC^+_{\gI})$.  Moreover, $W \in \cB(\cH_n)$ commutes with $T$ if and only if $W E(M) = E(M) W$ for every $M \in \sB(\sigma_S(T)\cap \CC^+_{\gI})$.
\end{thm}

\begin{proof}
We begin by using \eqref{eq:AJB} to write $T = A+ JB$, where $A \in \cB(\cH_n)$ is self-adjoint, $B \in \cB(\cH_n)$ is positive and $J$ is anti self-adjoint and unitary and $A, B$ and $J$ all mutually commute. Since $A$ is self-adjoint and $B$ is positive, we may use Theorem \ref{thm:BSA} to obtain spectral measures $E_1$ (resp., $E_2$) on $\sB(\sigma_S(A))$ (resp., $\sB(\sigma_S(B))$) such that
$$A = \int_{\sigma_S(A)} t \, dE_1(t) \quad {\rm and} \quad B = \int_{\sigma_S(B)} t \, dE_2(t).$$
Note that since $A \in \cB(\cH_n)$ is self-adjoint, we have that $\sigma_S(T)$ is a non-empty compact subset of $\RR$ and since $B \in \cB(\cH_n)$ is positive, we have that $\sigma_S(B)$ is a non-empty compact subset of $[0, \infty)$ (see Lemma \ref{lem:15July1}). Moreover, in view of Remark \ref{rem:SUPP}, we have
$$\supp E_1 = \sigma_S(A) \quad {\rm and} \quad \supp E_2 = \sigma_S(B).$$
Since $A$ and $B$ commute, we have that $E_1(M_1) E_2(M_2) = E_1(M_1) E_2(M_2)$ for every $M \in \sB(\sigma)$. Thus, we invoke Theorem \ref{thm:COMMUTING} with $\Omega_1 = \sB(\sigma_S(A))$ and $\Omega_2 = \sB(\sigma_S(B))$ to obtain a uniquely determined spectral measure $E$ on $\sB(\sigma_S(A) \times \sigma_S(B))$ given
by $E(M_1 \times M_2) = E(M_1)E(M_2)$ for $M_1 \in \sB(\sigma_S(A))$ and $M_2 \in \sB(\sigma_S(B))$.

Let $\gI \in \mathbb{S}$ be arbitrary and identify $\RR \times[0, \infty)$ with $\CC^+_{\gI} := \{ \lambda \in \CC^+_{\gI}: {\rm Im}(\lambda) \geq 0\}$ in the natural way. Consequently, we will view $E$ as a spectral measure on the non-empty compact subset $\sigma_S(A) \times \sigma_S(B) \subseteq \sB(\CC^+_{\gI})$. Lemma \ref{lem:INTEGRAL} with $f(\lambda) = {\rm Re}(\lambda)$ and $g(\lambda) = {\rm Im}(\lambda)$ imply that
\begin{equation}
\label{eq:A}
A = \int_{\RR} t \, dE_1(t) = \int_{\CC^+_{\gI} } {\rm Re}(\lambda) dE(\lambda)
\end{equation}
and
\begin{equation}
\label{eq:B}
B = \int_{[0, \infty)} u \, dE_2(u) = \int_{\CC^+_{\gI}} {\rm Im}(\lambda)dE(\lambda),
\end{equation}
respectively. Since $\supp E_1 = \sigma_S(A)$, $\supp E_2 = \sigma_B(A)$ and $\supp E \subseteq \sigma_S(A) \times \sigma_B(A)$, we may rewrite \eqref{eq:A} and \eqref{eq:B} as
\begin{equation}
\label{eq:Anew}
A = \int_{\sigma_S(A)} t \, dE_1(t) = \int_{\sigma_S(A) \times \sigma_S(B) } {\rm Re}(\lambda) dE(\lambda)
\end{equation}
and
\begin{equation}
\label{eq:Bnew}
B = \int_{\sigma_S(B)} u \, dE_2(u) = \int_{\sigma_S(A) \times \sigma_S(B)} {\rm Im}(\lambda)dE(\lambda),
\end{equation}
respectively. Thus , as $T = A + BJ$, we can use \eqref{eq:A} and \eqref{eq:B} to obtain
\begin{equation}
\label{eq:CLOSE}
T = \int_{\CC_{\gI}^+} {\rm Re}(\lambda) dE(\lambda) + \int_{\CC_{\gI}^+} {\rm Im}(\lambda) dE(\lambda) \, J.
\end{equation}
Theorem \ref{thm:IMPORTANT} with $f(\lambda) = \lambda$ implies that
\begin{align}
\sigma_S(T)\cap \CC_{\gI}^+ =& \; \sigma_S(\II(f)) \cap \CC_{\gI}^+ \nonumber \\
=& \;  \{ s \in \CC^+_{\gI}: E(\{ \lambda \in \Omega: |f(\lambda)^2 - 2{\rm Re}(s) f(\lambda) + |s|^2| < \varepsilon \} ) \neq 0 \nonumber\\
& \; \; \text{for all $\varepsilon > 0$}\} \nonumber\\
=& \; \{ s \in \CC^+_{\gI}: E(\{ \lambda \in \Omega: |\lambda^2 - 2{\rm Re}(s) \lambda + |s|^2| < \varepsilon \} ) \neq 0 \nonumber\\
& \; \; \text{for all $\varepsilon > 0$}\} \label{eq:NEWAB} \\
=& \; \{ s \in \CC^+_{\gI}: E(\{ \lambda \in \Omega: |\lambda - s|\cdot |\lambda - \bar{s}|  < \varepsilon \} ) \neq 0 \nonumber\\
& \; \; \text{for all $\varepsilon > 0$}\} \nonumber\\
=& \; \{ s \in \CC^+_{\gI}: E(\{ \lambda \in \Omega: |\lambda - s|^2 < \varepsilon \} ) \neq 0 \label{eq:SPECINC}\\
& \; \; \text{for all $\varepsilon > 0$}\}, \nonumber
\end{align}
since, if we write $\lambda = \lambda_0 + \lambda_1 \gI \in \CC_{\gI}^+$, with $\lambda_0 \in \RR$ and $\lambda_1 \geq 0$, and $s = s_0 + s_1 \gI \in \CC_{\gI}^+$, with $s_0 \in \RR$ and $s_1 \geq0$, then
\begin{align*}
|\lambda - \bar{s}|^2 =& \; (\lambda_0-s_0)^2 + \lambda_1^2 + 2 \lambda_1 s_1 + s_1^2 \\
\geq& \; (\lambda_0-s_0)^2 + \lambda_1^2 - 2 \lambda_1 s_1 + s_1^2 \\
=& \; |\lambda - s|^2.
\end{align*}
Thus, in view of \eqref{eq:SPRTALT}, \eqref{eq:SPECINC} implies that $\sigma_S(T) \cap \CC_{\gI}^+ = \supp E$. Thus, we may rewrite \eqref{eq:CLOSE} to obtain \eqref{eq:SPECN}.

We will now show that $E$ is unique. Suppose $E'$ on $\sB(\CC_{\gI}^+)$ is another spectral measure such that \eqref{eq:SPECN} holds. One can argue as in the proof of Theorem \ref{thm:BSA} to see that $\supp E' = \sigma_S(T)$ which is non-empty and compact by  Theorem
\ref{thm:SPECNONEMPTYCOMPACT}. Consequently, we may use Theorem \ref{thm:BFUNC}(iv) with $y=x$ and $f \in \sC(\sigma_S(T) \cap \CC_{\gI}^+, \RR)$ to obtain
$$\int_{\sigma_S(T) \cap \CC_{\gI}^+} f(\lambda) d\langle E(\lambda)x, x \rangle = \int_{\sigma_S(T) \cap \CC_{\gI}^+} f(\lambda) d\langle E'(\lambda)x, x \rangle.
$$
Thus, one can show that $E = E'$ as in the proof of Theorem \ref{thm:BSA}.

The proof of the final assertion can be completed in the same way as the proof of the final assertion in Theorem \ref{thm:BSA}.
\end{proof}

\begin{rem}
\label{rem:IDENTIFICATION_unbounded}
Let $\gI, \gJ \in \mathbb{S}$ and $\gamma: \CC^+_{\gI} \to \CC^+_{\gJ}$ denote the bijective map given by
\begin{equation}
\label{eq:GAMMA}
\gamma(\lambda_0 + \lambda_1 \gI) = \lambda_0 + \lambda_1 \mathfrak{J},
\end{equation}
where $\lambda_0 \in \RR$ and $\lambda_1 \geq 0$. An immediate consequence of the proof of Theorem \ref{thm:BN}, is that if $E_{\gI}$, where $\gI \in \mathbb{S}$, is a spectral measure for a normal operator $T \in \cB(\cH_n)$ and $E_{\gJ}$, where $\gJ \in \mathbb{S}$, then since $\gamma(\sigma_S(T) \cap \CC_{\gI}^+) = \sigma_S(T) \cap \CC_{\gI}^+$ and $\supp E_{\gI} = \sigma_S(T) \cap \CC_{\gI}^+$, we have
\begin{equation}
\label{eq:EQUIVALENCE}
E_{\gI}(M) = E_{\gJ}(\gamma(M)) \quad \quad {\rm for}\quad M \in\sB(\sigma_S(T) \cap \CC_{\gI}^+).
\end{equation}
In view of the above observations, we are justified in calling a spectral measure $E$ on $\sigma_S(T) \cap \CC_{\gI}^+$ {\it the spectral measure of $T$}.
\end{rem}

\begin{cor}[Borel functional calculus in the bounded case]
\label{cor:BFC}
Let $T \in \cB(\cH_n)$ be normal, $J$ be an imaginary operator satisfying \eqref{eq:AJB} and $E$ be the spectral measure of $T$. Fix $\gI \in \mathbb{S}$ and put $\Omega_{\gI}^+ := \sigma_S(T) \cap\CC_{\gI}^+$. For any $f, g \in \gB(\Omega^+_{\gI}, \sB(\Omega_{\gI}^+), \CC_{\gI}, E)$ {\rm (}see Subsection \ref{sec:SIsbdd}{\rm )}, we have the spectral integrals $\II(f), \II(g) \in \cB(\cH_n)$ are normal operators with the following properties{\rm :}
\begin{enumerate}
\item[(i)] $\II(\bar{f}) = \II(f)^*$.

\smallskip

\item[(ii)] $\II(fg) = \II(f)\II(g)$.

\smallskip

\item[(iii)] $\II(\gI) = J$ and $\II(cf + g) = ( {\rm Re}(c)I +{\rm Im}(c) \, J ) \II(f) + \II(g)$ for all  $c \in \CC_{\gI}$ and $\gI \in \mathbb{S}$.

\smallskip

\item[(iv)] For all $x,y \in \cH_n$, we have
$$
\langle \II(f)x, y \rangle = \int_{\Omega_{\gI}^+} {\rm Re}(f(\lambda)) d\langle E(\lambda)x, y \rangle + \int_{\Omega_{\gI}^+} {\rm Im}(f(\lambda)) d\langle JE(\lambda)x, y \rangle
$$

\smallskip

\item[(v)] $\| \II(f) x \|^2 = \int_{\Omega} | f(\lambda)|^2 d ({\rm Re}\, \langle E(\lambda)x, x \rangle)$ for all $x \in \cH_n$.

\smallskip

\item[(vi)] $\| \II(f) \| \leq \| f \|_{\infty}$. Moreover, $\| \II(f) \| = \| f \|_{\infty}$ if and only if $f \in L_{\infty}(\Omega_{\gI}^+, \sB(\Omega_{\gI}^+, \CC_{\gI}, E )$.

\smallskip

\item[(viii)] For any sequence of functions $(f_j)_{j=1}^{\infty}$, where $f_j \in \gB$ for $j=1,2,\ldots$, which converges pointwise $E$-a.e. on $\Omega$ to $f$ and there exists $\kappa > 0$ such that $|f_j(\lambda)| \leq \kappa$ for all $\lambda \in \Omega$ and $j=1,2,\ldots$, we have
$$s-\lim_{j \to \infty} \II(f_n) = \II(f).$$

\smallskip

\item[(viii)] If we let $\Pi_{\pm}(J, \gI)$ denote the orthogonal projection onto the right complex subspace $\cH_{\pm}(J, \gI)$, respectively {\rm (}see Theorem \ref{thm:IOs}{\rm(ii))}, then
\begin{align*}
\langle \II(f)x, y \rangle =& \; \int_{\Omega_{\gI}^+} d\langle E(\lambda) \Pi_+(J, \gI)x, y \rangle \, f(\lambda)\\
+& \; \int_{\Omega_{\gI}^+} d\langle E(\lambda) \Pi_-(J, \gI)x, y \rangle \overline{f(\lambda)} \quad \quad {\rm for} \quad x,y \in \cH_n,
\end{align*}
where both integrals above are meant in the sense of \eqref{eq:complexRIGHTINTEGRAL}.

\smallskip

\item[(ix)] If $f \in \gB(\Omega_{\gI}^+, \sB(\Omega_{\gI}^+), \CC_{\gI}, E)$ is nonnegative $E$-a.e., then $\II(f) \succeq 0$.

\smallskip

\item[(x)] $\II(f)^{-1} \in \cB(\cH_n)$ if and only if $f(\lambda) \neq 0$ $E$-a.e. and
$$
f^{-1} \in L_{\infty}(\Omega_{\gI}^+, \sB(\Omega_{\gI}^+), \CC_{\gI}, E).
$$
 In this case,
 $$
 \II(f)^{-1} = \II(1/f).
 $$

\smallskip

\item[(xi)] The spectral measure $E_{\II(f)}$ of $\II(f)$ satisfies the identity
\begin{equation}
\label{eq:SPECTRALINTEGRALMEASURE}
E_{\II(f)}(M) = E(f^{-1}(M \cap f(\Omega_{\gI}^+) ).
\end{equation}
\end{enumerate}
\end{cor}

\begin{proof}
Let $T \in \cB(\cH_n)$ be normal and fix $\gI \in \mathbb{S}$. Then by Theorem \ref{thm:BN}, we can find a spectral measure $E$ on $\Omega_{\gI}^+ := \sigma_S(T) \cap \CC_{\gI}^+$ such that \eqref{eq:SPECN} holds. Thus, we may invoke Theorem \ref{thm:BFUNC} with $\Omega = \Omega_{\gI}^+$ and $\sA = \sB(\Omega_{\gI}^+)$ to obtain (i)-(v), the first assertion in (vi) and (vii). The fact that $\II(f)\in \cB(\cH_n)$ is normal is an easy consequence of (i) and (ii). Indeed, as
$$\II(f)\II(f)^* = \II(f \bar{f}) = \II(\bar{f} f ) = \II(f)^* \II(f),$$
we have that $\II(f)$ is normal.

The second assertion in (vi) follows immediately from Lemma \ref{lem:SIbounded}.  We will now prove (vii). We may use Theorem \ref{thm:IOs}(iii) to write $x = x_+ + x_-$ and $y = y_+ + y_-$, where $x_+, y_+ \in \cH_+(J, \gI)$ and $x_-,y_- \in\cH_-(J, \gI)$. We note that $\Pi_{\pm}(J, \gI) x = x_{\pm}$. Then we may use (iv) and \eqref{eq:CI} to obtain
\begin{align*}
\langle \II(f)x, y \rangle =& \;  \int_{\Omega_{\gI}^+} {\rm Re}(f(\lambda)) \, d\langle E(\lambda)x, y \rangle \\
+& \; \int_{\Omega_{\gI}^+} {\rm Im}(f(\lambda)) \, d\langle E(\lambda)J(x_+ + x_-), y \rangle \\
 =& \; \int_{\Omega_{\gI}^+} {\rm Re}(f(\lambda)) \, d\langle E(\lambda)(x_+ + x_-), y \rangle \\
+& \; \int_{\Omega_{\gI}^+} {\rm Im}(f(\lambda)) \, d\langle E(\lambda)(x_+ \gI - x_- \gI), y \rangle \\
 =& \; \int_{\Omega_{\gI}^+}  \, d\langle E(\lambda)x_+, y \rangle f(\lambda)\\
+& \; \int_{\Omega_{\gI}^+}  \, d\langle E(\lambda)x_-, y \rangle \overline{f(\lambda)},
\end{align*}
in which case we have the desired identity.

We will now prove (ix). If we let $y = x$ in (iv) with $f(\lambda) \geq 0$ $E$-a.e., then we may use the fact that $E(M) \succeq 0$ for all $M \in \sB(\Omega_{\gI}^+$ to obtain
$$\langle \II(f)x, x \rangle = \int_{\Omega_{\gI}^+} f(\lambda) d\langle E(\lambda)x, x \rangle \succeq 0.$$

Assertion (x) is an immediate consequence of Lemma \ref{lem:SIbounded} and Lemma \ref{lem:INVERT}. Finally, we will prove (xi). Using Theorem \ref{thm:CHANGEOFVAR} with $\Omega = \Omega_{\gI}^+$, $\Omega' = f(\Omega_{\gI}^+) \setminus \{ \infty \}$, $\psi(\lambda) = f(\lambda)$, $h(\lambda') = \lambda'$ and  $F$ be the spectral measure given by \eqref{eq:SPECTRALINTEGRALMEASURE}, we obtain
\begin{align*}
& \; \int_{\CC_{\gI}^+} {\rm Re}(\lambda') dF(\lambda') + \int_{\CC_{\gI}^+} {\rm Im}(\lambda') dF(\lambda') J\\
=& \; \int_{f(\Omega_{\gI}^+) \setminus \{ \infty\} } {\rm Re}(\lambda') dF(\lambda') + \int_{f(\Omega_{\gI}^+)\setminus \{ \infty \}} {\rm Im}(\lambda') dF(\lambda') J\\
=& \; \int_{f(\Omega_{\gI}^+) \setminus \{ \infty\} } {\rm Re}(f(\lambda)) dE(\lambda) + \int_{f(\Omega_{\gI}^+)\setminus \{ \infty \}} {\rm Im}(f(\lambda) ) dE(\lambda) J\\
=& \; \II(f).
\end{align*}
Thus, the asserted uniqueness in Theorem \ref{thm:BN} ensures that the spectral measure for $\II(f)$ is given by \eqref{eq:SPECTRALINTEGRALMEASURE}.
\end{proof}

\begin{cor}[subclasses of bounded normal operators]
\label{cor:SUBCLASSES}
Let $T \in \cB(\cH_n)$ be normal with spectral measure $E$ and $J$ be an imaginary operator associated with $T$ {\rm (}see Theorem \ref{thm:AJB}{\rm )}. We have the following{\rm :}
\begin{enumerate}
\item[(i)] $T$ is self-adjoint if and only if $\sigma_S(T) \subseteq \RR$. In this case,
$$T = \int_{\RR} t \, dE(t).$$
\item[(ii)] $T$ is positive if and only if $\sigma_S(T) \subseteq [0, \infty)$.
In this case,
$$T = \int_{[0,\infty)} t \, dE(t).$$
\item[(iii)] $T$ is anti self-adjoint if and only if $\sigma_S(T) \subseteq \{s \in \RR^{n+1}: {\rm Re}(s) = 0 \}$. In this case,
$$T = \int_{\RR}  t \, dE(t) J.$$
\item[(v)] $T$ is unitary if and only if $\sigma_S(T) \subseteq \{s \in \RR^{n+1}: |s| = 1 \}$. In this case,
$$T = \int_{|\lambda|=1} {\rm Re}(\lambda)  \, dE(\lambda) + \int_{|\lambda|=1} {\rm Im}(\lambda)  \, dE(\lambda).$$
\item[(vi)] $T$ is imaginary if and only if $\sigma_S(T) \cap \CC_{\gI}^+ = \{ \gI \}$. In this case,
$$Tx  = \sum_{i \in \cI} \eta_i \, \gI \langle x, \eta_i \rangle \quad \quad {\rm for} \quad x \in \cH_n,$$
where $(\eta_i )_{i \in \cI}$ is any orthonormal basis of $\cH_+(T, \gI)$ {\rm (}see \eqref{eq:CI}{\rm )}.
\end{enumerate}
\end{cor}

\begin{proof}
Let us prove the claims.
We prove (i).
If $T$ is self-adjoint, then Lemma \ref{lem:15July1} asserts that $\sigma_S(T) \subseteq \RR$. The integral representation $T = \int_{\RR} t\, dE(t)$ appeared in Theorem \ref{thm:BSA}. Conservely, suppose $T \in \cB(\cH_n)$ is normal and $\sigma_S(T) \subseteq \RR$. Then the fact that $T= T^*$ follows at once from Corollary \ref{cor:BFC} with $f(\lambda) = \lambda$.  Assertions (ii) (with the help of Lemma
\ref{lem:15July1pos}), (iii) (with the help of Lemma \ref{lem:UNITARY}) and (iv) (with the help of Lemma \ref{lem:15July1ASA}) can be proved in much the same manner using Theorem \ref{thm:BN} and Corollary \ref{cor:BFC}.

Finally, Assertion (vi) is just \eqref{eq:Jexpansion}.
\end{proof}

\section{Spectral theorem for an unbounded normal operator and some consequences}
\label{sec:UBN}

Before we can formulate and prove the spectral theorem for an unbounded normal operator on a Clifford module, we will need to define the bounded transform of a densely defined closed operator and also a lemma which highlights various properties of the aforementioned bounded transform.

\begin{defn}
\label{def:Ztransform}
Suppose $T \in \cL(\cH_n)$ is a densely defined closed operator. Let $C_T$ be as
 in the statement of Theorem \ref{thm:Ctransform}, i.e., $C_T := (I+T^*T)^{-1}$. In view of Theorem \ref{thm:Ctransform}(ii), $C_T \in \cB(\cH_n)$ and $C_T$ is positive. In view of Corollary \ref{cor:POSSQROOT}, $C_T$ has a unique positive square root $C_T^{1/2} \in \cB(\cH_n)$. If we let
\begin{equation}
\label{eq:Z_Ttransform}
Z_T := T C_T^{1/2},
\end{equation}
then $Z_T$ will be called the {\it bounded transform of $T$} (we will justify this nomenclature in Lemma \ref{lem:Z_Ttransform}).
\end{defn}

\begin{lem}
\label{lem:Z_Ttransform}
Let $T\in \cL(\cH_n)$ be densely defined. Then the following statements hold{\rm :}
\begin{enumerate}
\item[(i)] $Z_T \in \cB(\cH_n)$. Moreover,
\begin{equation}
\label{eq:CONTRACTION}
\| Z_T \| \leq 1
\end{equation}
and
\begin{equation}
\label{eq:IDENTITY}
C_T = (I+T^*T)^{-1} = I - Z_T^* Z_T.
\end{equation}
\smallskip
\item[(ii)] $(Z_T)^* = Z_{T^*}$ and hence $Z_T$ is self-adjoint whenever $T$ is self-adjoint.
\smallskip
\item[(iii)] $Z_T$ is normal whenever $T$ is normal.
\end{enumerate}
\end{lem}

\begin{proof} The proof in the classical complex Hilbert space (see, e.g., Lemma 5.8 in \cite{Schmuedgen}) can be carried over into the Clifford module setting. For completeness, we will provide the proof. The proof is broken into steps.

\bigskip

\noindent {\bf Step 1:} {\it Prove {\rm (i)}}. \\

\bigskip

First note that
\begin{equation}
\label{eq:Sept3yt1}
\{\text{$Cx$: $x \in \cH$}\} = \cD(I + T^* T ) = \cD(T^* T)
\end{equation}
and hence if $x \in \cH_n$, then
\begin{align*}
\| T C_T^{1/2} C_T^{1/2} x \|^2 =& \; \langle T^* T C_T x, C_T x \rangle \\
\leq& \; \langle (I + T^* T) C_T x, C_T x \rangle \\
=& \; \langle C_T^{-1} C_T x , C_T x \rangle \\
=& \; \langle x, C_Tx \rangle \\
=& \; \| C_T^{1/2} x \|^2.
\end{align*}
Thus, if $y \in \{\text{$C_T^{1/2}x$: $x \in \cH_n$}\}$, then
\begin{equation}
\label{eq:Sept3ttu1}
\| Z_T y \| = \| T C_T^{1/2} y \| \leq \| y \|.
\end{equation}
Since ${\rm Ker}(C_T) = \{ 0 \}$, we have that ${\rm Ker}(C_T^{1/2}) = \{ 0 \}$ and hence $\{\text{$C_T^{1/2} x$: $x \in \cH$}\}$ is a dense subset of $\cH_n$. By assumption, $T$ is a closed operator and since $C_T^{1/2} \in \cB(\cH_n)$, we have that $Z_T$ is closed as well.
Thus, we have
$$
\{\text{$C_T^{1/2}x$: $x \in \cH_n$}\} \subseteq \cD(T), \ \ \ \  \cD(Z_T) = \cH_n
$$
 and, in view of \eqref{eq:Sept3ttu1}, we have $\| Z_T \| \leq 1$.

Next, it follows from \eqref{eq:Sept3ttu1} and $C^{1/2} T^* \subseteq Z_T^*$ that
\begin{align*}
(I - C_T) C_T^{1/2} =& \; C_T^{1/2} (I + T^* T) C_T - C_T^{1/2} C_T \\
=& \; C_T^{1/2} T^* T C_T^{1/2} C_T^{1/2} \\
\subseteq & \; Z_T^* Z_T C_T^{1/2}.
\end{align*}
Thus, $Z_T^* Z_T C_T^{1/2} = (I - C_T )C_T^{1/2}$ and, as $\{\text{$C_T^{1/2} x$: $x \in \cH$}\}$ is a dense subset of $\cH_n$, we get \eqref{eq:IDENTITY}.

\smallskip

\noindent {\bf Step 2:} {\it Prove {\rm (ii)}}. \\

\smallskip

Using \eqref{eq:IDENTITY} we get that $C_{T^*} = (I + T T^*)^{-1}$. If $x \in \cD(T^*)$, then let $y = C_{T^*} x$. Therefore,
$$x = (I + T T^* )y $$
and
$$T^* x = T^* (I + TT^*)y = (I + T^* T) T^* y.$$
Thus, $C_{T^*} x \in \cD(T^*)$ and hence
\begin{equation}
\label{eq:Sept3rtz1}
C_T T^* x = T^* y = T^* C_{T^*} x.
\end{equation}
It follows easily from \eqref{eq:Sept3rtz1} and \eqref{eq:IDENTITY} that $p(C_{T^*})x \in \cD(T^*)$ and
$$p (C_T ) T^* x = T^* p(C_{T^*} ) x$$
for any real polynomial $p \in \RR[t]$. By the Weierstra{\ss} approximation theorem, there exists a sequence of real polynomials $( p_n )_{n=0}^{\infty}$ which converge uniformly in supremum norm to the function $t \mapsto t^{1/2}$ on $[0, 1]$. We may use Corollary \ref{cor:POSSQROOT} to obtain
$$\lim_{j \to \infty} \| p_j(C_T) - C_T^{1/2} \| = \lim_{j \to \infty} \| p_j( C_T) - C_T^{1/2} \| = 0.$$

Since $T$ is a closed operator, $T^*$ is also a closed operator. Thus, we have
\begin{align*}
C_T^{1/2} T^* x =& \; \lim_{j \to \infty} p_j(C_T) T^* x = \lim_{j\to \infty} T^* p_n (C_{T^*}) x \\
=& \; T^* (C_{T^*})^{1/2} x \quad {\rm for} \quad x \in \cD(T^*).
\end{align*}
As $C_T^{1/2} T^* \subseteq (TC_T^{1/2})^* = Z_{T^*}$, we get that
$$Z_{T^*} x = C_T^{1/2} T^* x = T^* (C_{T^*})^{1/2} x = (Z_{T})^* x$$
for $x \in \cD(T^*)$. Finally, since $\cD(T^*)$ is dense in $\cH$ and $Z_T \in \cB(\cH_n)$, we have that $Z_{T^*} x = (Z_T)^* x$ for $x \in \cH_n$, i.e., $Z_{T^*} = (Z_T)^*$.

\smallskip

\noindent {\bf Step 3:} {\it Prove {\rm (iii)}.} \\

\smallskip

Using the first of (ii) on $T$ and $T^*$ and the assumption $TT^* = T^* T$, we have
\begin{align*}
I - (Z_T)^* Z_T =& \;  (I + T^* T )^{-1} = (I + T T^* )^{-1} = I- (Z_{T^*})^* Z_{T^*} \\
=& \; I - Z_T (Z_{T})^*
\end{align*}
in which case it is clear that $Z_T$ is normal.
\end{proof}

Suppose $T \in \cL(\cH_n)$ is an unbounded normal operator and consider the bounded transform $Z_T \in \cB(\cH_n)$ of $T$ given by \eqref{eq:Z_Ttransform}. In view of Lemma \ref{lem:Z_Ttransform}, we have $Z_T$ is a bounded normal operator. Thus, we may use Theorem \ref{thm:AJB} to find a self-adjoint operator $A_{Z_T} \in \cB(\cH_n)$, a positive operator $B_{Z_T} \in \cB(\cH_n)$ and a imaginary operator $J$ such that $A, B$ and $J$ mutually commute and obey
\begin{equation}
\label{eq:ZAJB}
Z_T = A_{Z_T} + JB_{Z_T}.
\end{equation}

\begin{defn}[strongly commuting operators]
\label{def:STRONGLYCOMMUTING}
Suppose $T, \widetilde{T} \in \cL(\cH_n)$ are unbounded normal operators . We will say that $T$ and $\widetilde{T}$ {\it strongly commute} if the bounded transforms of $T$ and $\widetilde{T}$, i.e., $Z_T$ and $ Z_{ \widetilde{T}}$, respectively, commute.
\end{defn}

We are now ready to formulate and prove the spectral theorem for an unbounded normal operator.

\begin{thm}
\label{thm:UNBOUNDEDNORMAL}
Let $T \in \cL(\cH_n)$ be an unbounded normal operator. Then corresponding to any choice of $\gI \in \mathbb{S}$ and an imaginary operator $J$ in \eqref{eq:ZAJB}, there exists a spectral measure $E$ on the Borel $\sigma$-algebra $\mathscr{B}(\sigma_S(T) \cap \CC^+_{\gI})$ such that
\begin{equation}
\label{eq:SPECNZ}
T= \int_{\sigma_S(T) \, \cap \, \CC_{\gI}^+} {\rm Re}(\lambda) \, dE(\lambda) + \int_{\sigma_S(T)\, \cap \, \CC_{\gI}^+} {\rm Im}(\lambda) \, dE(\lambda) \, J.
\end{equation}

$E$ is unique in the sense that if $F$ is a spectral measure on $\mathscr{B}(\CC^+_{\gI})$ such that $T = \int_{\CC^+_{\gI}} {\rm Re}(\lambda) \, dF(\lambda) + \int_{\CC^+_{\gI}} {\rm Im}(\lambda) dF(\lambda) J$, then $E(M \cap \sigma_S(T)\cap \CC^+_{\gI}) = F(M)$ for all $M \in \mathscr{B}(\CC^+_{\gI})$.
\end{thm}

\begin{rem}
\label{rem:SPECTRUMNONEMPTY}
Before proceeding to the proof of Theorem \ref{thm:UNBOUNDEDNORMAL}, we wish to point out that a consequence of Theorem \ref{thm:UNBOUNDEDNORMAL} is that for any unbounded normal operator $T \in \cL(\cH_n)$, we have
\begin{equation}
\label{eq:NONEMPTY}
\sigma_S(T) \neq \emptyset.
\end{equation}
Indeed, for any choice of $\gI \in \mathbb{S}$, Theorem \ref{thm:IMPORTANT} with $f(\lambda) = \lambda$ implies that $\supp E = \sigma_S(T) \cap \CC_{\gI}^+$. Since $\supp E \neq \emptyset$, we have \eqref{eq:NONEMPTY}.
\end{rem}

\begin{proof}[Proof of Theorem \ref{thm:UNBOUNDEDNORMAL}]
The proof is broken into steps.

\bigskip

\noindent {\bf Step 1:} {\it Show that there exists a spectral measure $E$ on $\sigma_S(T) \cap \CC^+_{\gI}$ such that \eqref{eq:SPECN} holds}.  \\

\smallskip

In view of Lemma \ref{lem:Z_Ttransform}, we have $Z_T$ is a bounded normal operator. Moreover, $\| Z_T \| \leq 1$. Let $J$ be the imaginary operator appearing in \eqref{eq:ZAJB}. Since $Z_T \in \cB(\cH_n)$ is normal, we can use Theorem \ref{thm:BSA} to obtain a uniquely determined spectral measure $E_{Z_T}$ on $\sB(\sigma_S(Z_T)) \subseteq \overline{ \DD^+_{\gI}}$,where
$$\DD_{\gI} := \{ \lambda \in \CC_{\gI}: |\lambda| <1 \} \quad \quad {\rm for}\quad \gI \in \mathbb{S},$$
such that
\begin{equation}
\label{eq:ZTSR}
Z_T = \int_{\DD_{\gI}} {\rm Re}(\lambda) dE_{Z_T}(\lambda) + \int_{\DD_{\gI}} {\rm Im}(\lambda) dE_{Z_T}(\lambda) J.
\end{equation}

We may use \eqref{eq:ZTSR} and \eqref{eq:XY} to see that, for any $x \in \cH_n$ and $M \in \sB(\CC^+_{\gI})$, we have
\begin{equation}
\label{eq:STEPFURTHER}
\langle (I -(Z_T)^* Z_T) E_{Z_T}(M), E_{Z_T}(M)x \rangle = \int_M (1 - |\lambda|^2 ) d\langle E_{Z_T}(\lambda)x , x \rangle.
\end{equation}
A simple consequence of Theorem \ref{thm:Ctransform} is $C_T^{-1} = I - (Z_T)^* Z_T$ is invertible and positive. Thus,  $\Ker(I - (Z_T)^*(Z_T) ) = \{ 0 \}$ and $I - (Z_T)^*Z_T$ is positive. Thus, \eqref{eq:STEPFURTHER} implies that $\sigma_S(Z_T) \cap \CC_{\gI}^+ = \supp E_{Z_T} \subseteq \overline{\DD^+_{\gI}}$ and $E_{Z_T}(\TT_{\gI}^+) = 0$.
Consequently,
$$
E_{Z_T}(\DD^+_{\gI} )= E_{Z_T}(\overline{\DD^+_{\gI}} \, \setminus \, \TT_{\gI}^+ ) = I,
$$
 where $\TT_{\gI} := \{ \lambda \in \CC_{\gI}: |\lambda| = 1 \}$.

Let $\psi(\lambda) = \lambda(1-|\lambda|)^{-1/2}$. Notice that $\psi$ is an intrinsic function (see Defintion \ref{def:INTRINSIC}). We claim that $\II(\psi) = Z_T (C_T^{1/2})^{-1}$, where $\II$ denotes the spectral integral with respect to the spectral measure $E_{Z_T}$. Indeed, since $E_{Z_T}(\DD_{\gI}^+) = I$ and $\psi$ is clearly finite $E_{Z_T}$-a.e. on $\CC_{\gI}^+$, \eqref{eq:DOM} asserts that $\cD(\II(\psi)) = \cD(\II( (1-\lambda)^{-1/2} ))$. Thus, we may make us of Theorem \ref{thm:UNBFUNCCALC}(ii) and (iv) to obtain
$$\II(\psi) = \II(\lambda) \II( (1-|\lambda|^2)^{-1/2} ).$$ Next, in view of Lemma \ref{lem:invert}, we have that
$$\II( (1-|\lambda|^2)^{-1/2} ) = \II( (1-|\lambda|^2)^{1/2})^{-1}.$$
Finally, since $Z_T = \II(\lambda)$ and the square root of a positive bounded operator is unique (see Corollary \ref{cor:POSSQROOT}), we have that
$$C_T^{1/2} = (I - (Z_T)^* Z_T )^{1/2} = \II( (1- |\lambda|^2)^{1/2} ).$$
Therefore, we have proved the claim $\II(\psi) = Z_T (C_T^{1/2})^{-1/2}$.

In what follows, with a slight abuse of notation, we will view $\psi$ as $\psi|_{\DD_{\gI}^+}$. Let $E(M) := E_{Z_T}(\psi^{-1}(M))$ for $M \in \sB(\CC_{\gI}^+)$. In view of Remark \ref{rem:TRANSSPEC}, $E$ is a spectral measure on $\sB(\CC_{\gI}^+)$ and $J$ is an imaginary associated operator with the spectral measure $E$. Using the fact that $T = Z_T \{(C_T)^{1/2}\}^{-1}$, $\II(\psi) = Z_T \{(C_T)^{1/2}\}^{-1}$ and Theorem \ref{thm:CHANGEOFVAR}, we obtain
\begin{align}
T =& \; \II(\psi) = \int_{\DD_{\gI}^+} {\rm Re}(\psi(\lambda)) dE_{Z_T}(\lambda) + \int_{\DD_{\gI}^+} {\rm Im}(\psi(\lambda)) dE_{Z_T}(\lambda)J \nonumber \\
=& \; \int_{\DD_{\gI}^+} \frac{ 1 }{\sqrt{1-|\lambda|^2}} \, dE_{Z_T}(\lambda) \nonumber \\
=& \; \int_{\CC_{\gI}^+} {\rm Re}(\lambda) dE(\lambda) + \int_{\CC_{\gI}^+} {\rm Im}(\lambda) dE(\lambda) J. \label{eq:ALMOSTTHERE}
\end{align}
Finally, using Theorem \ref{thm:IMPORTANT}, we have that $\supp E = \sigma_S(T) \cap \CC_{\gI}^+$, i.e.,  \eqref{eq:ALMOSTTHERE} is \eqref{eq:SPECNZ}.

\smallskip

\noindent {\bf Step 2:} {\it Show that $E$ is unique}.  \\

\smallskip

Suppose $F$ is another spectral measure on $\sB(\CC_{\gI}^+)$ such that the imaginary operator $J$ is associated with $F$ and
$$T= \int_{\CC_{\gI}^+} {\rm Re}(\lambda) \, dF(\lambda) + \int_{\CC_{\gI}^+} {\rm Im}(\lambda) \, dF(\lambda) \, J.$$
Let $F'(M) := F(\psi(M))$ for $M \in \sB(\DD_{\gI}^+)$. In view of Remark \ref{rem:TRANSSPEC}, $F'$ is a spectral measure on $\sB(\DD_{\gI}^+)$ and the imaginary operator $J$ is associated with $F'$. Moreover,  since $E_{Z_T}(\DD_{\gI}^+)$, we have
\begin{align*}
Z_T =& \;  \int_{\DD_{\gI}^+} {\rm Re}(\lambda) dF'(\lambda) +
\int_{\DD_{\gI}^+} {\rm Im}(\lambda) dF'(\lambda) \, J \\
=& \;  \int_{\DD_{\gI}^+} {\rm Re}(\lambda) dE_{Z_T}(\lambda) +
\int_{\DD_{\gI}^+} {\rm Im}(\lambda) dE_{Z_T}(\lambda) \, J.
\end{align*}
But then the Borel functional calculus applied to $Z_T \in \cB(\cH_n)$ (see Corollary \ref{cor:BFC}) asserts that for any real-valued polynomial $p \in \RR[\lambda]$ and $x \in \cH_n$, we have
\begin{align*}
\langle p(\lambda) x, x \rangle =& \;  \int_{\DD_{\gI}^+} p(\lambda) dF'(\lambda)x, x \rangle \\
=& \; \int_{\DD_{\gI}^+} p(\lambda) dE_{Z_T}(\lambda)x, x \rangle.
\end{align*}
Since $\RR[\lambda]$ is dense in the Banach space of real-valued continuous functions on the compact set $\overline{\DD_{\gI}^+}$ (denoted by $\sC(\overline{\DD_{\gI}^+}, \RR)$), we can use Corollary \ref{cor:RR} to see that $\langle F'(M)x, x \rangle = \langle E_{Z_T}(M)x, x \rangle$ for any $M \in \sB(\DD_{\gI}^+)$, But then the polarisation formula \eqref{eq:POLAR} enables us to deduce that $F' = E_{Z_T}$. Finally since $F(M) = F'(\psi^{-1}(M))$ for $M \in \sB(\CC_{\gI}^+)$, we have $F = E$ on $\sB(\CC_{\gI}^+)$ as required.
\end{proof}

\begin{rem}
\label{rem:IDENTIFICATION}
Let $\gI, \gJ \in \mathbb{S}$ and $\gamma: \CC^+_{\gI} \to \CC^+_{\gJ}$ denote the bijective map given by
\begin{equation}
\label{eq:unGAMMA}
\gamma(\lambda_0 + \lambda_1 \gI) = \lambda_0 + \lambda_1 \mathfrak{J},
\end{equation}
where $\lambda_0 \in \RR$ and $\lambda_1 \geq 0$. An immediate consequence of the proof of Theorem \ref{thm:UNBOUNDEDNORMAL}, is that if $E_{\gI}$, where $\gI \in \mathbb{S}$, is a spectral measure for an unbounded normal operator $T \in \cL(\cH_n)$ and $E_{\gJ}$, where $\gJ \in \mathbb{S}$, then since $\gamma(\sigma_S(T) \cap \CC_{\gI}^+) = \sigma_S(T) \cap \CC_{\gI}^+$ and $\supp E_{\gI} = \sigma_S(T) \cap \CC_{\gI}^+$, we have
\begin{equation}
\label{eq:EQUIVALENCEun}
E_{\gI}(M) = E_{\gJ}(\gamma(M)) \quad \quad {\rm for}\quad M \in\sB(\sigma_S(T) \cap \CC_{\gI}^+).
\end{equation}
In view of the above observations, we are justified in calling a spectral measure $E$ on $\sigma_S(T) \cap \CC_{\gI}^+$ {\it the spectral measure of $T$}.
\end{rem}

\begin{cor}[Borel functional calculus in the unbounded case]
\label{cor:BFCUNBOUNDED}
Let $T, E$ and $J$ be as in Theorem \ref{thm:UNBOUNDEDNORMAL} and $\Omega_{\gI}^+ := \sigma_S(T) \cap \CC_{\gI}^+$ for $\gI \in \mathbb{S}$. For any $f, g \in \gF(\Omega_{\gI}^+, \sB(\Omega_{\gI}^+), E)$ and $c \in \CC_{\gI}$, we have the spectral integrals $\II(f), \II(g) \in \cL(\cH_n)$ have the following properties{\rm :}
\begin{enumerate}
\item[(i)] $\II(\bar{f}) = \II(f)^*$.

\smallskip
\item[(ii)] $\II(fg) = \overline{\II(f)\II(g)}$.

\smallskip
\item[(iii)] $\II(fc +g) = \overline{\II(f) ({\rm Re}(c)I + {\rm Im}(c)J )+ \II(g)}$ for all $c \in \CC_{\gI}$.

\smallskip
\item[(iv)] $\II(f)$ is a closed normal operator belonging to $\cL(\cH_n)$.

\smallskip
\item[(v)] $\cD(\II(f) \II(g)) = \cD( \II(fg) ) \cap \cD(\II(g))$.

\item[(vi)] If we let $\Pi_{\pm}(J, \gI)$ denote the orthogonal projection onto the right complex subspace $\cH_{\pm}(J, \gI)$, respectively {\rm (}see Theorem \ref{thm:IOs}{\rm(ii))}, then
\begin{align*}
\langle \II(f)x, y \rangle =& \; \int_{\Omega_{\gI}^+} d\langle E(\lambda) \Pi_+(J, \gI)x, y \rangle \, f(\lambda)\\
+& \; \int_{\Omega_{\gI}^+} d\langle E(\lambda) \Pi_-(J, \gI)x, y \rangle \overline{f(\lambda)} \quad \quad {\rm for} \quad x \in \cD(T),
\end{align*}
where both integrals above are meant in the sense of \eqref{eq:complexRIGHTINTEGRAL}.

\smallskip

\item[(vii)] If $f \in \gF(\Omega_{\gI}^+, \sB(\Omega_{\gI}^+), \CC_{\gI}, E)$ is nonnegative $E$-a.e., then $\II(f) \succeq 0$.

\smallskip

\item[(vii)] $\II(f)^{-1} \in \cL(\cH_n)$ if and only if $f(\lambda) \neq 0$ $E$-a.e. In this case, $\II(f)^{-1} = \II(1/f)$.

\smallskip

\item[(viii)] The spectral measure $E_{\II(f)}$ of $\II(f)$ satisfies the identity
\begin{equation}
\label{eq:SPECTRALINTEGRALMEASURE2}
E_{\II(f)}(M) = E(f^{-1}(M \cap f(\Omega_{\gI}^+) ).
\end{equation}

\end{enumerate}
\end{cor}

\begin{proof}
Let $T \in \cL(\cH_n)$ be normal and fix $\gI \in \mathbb{S}$. Then by Theorem \ref{thm:UNBOUNDEDNORMAL}, we can find a spectral measure $E$ on $\sigma_S(T) \cap \CC_{\gI}^+$ such that \eqref{eq:SPECNZ} holds. Thus, we may invoke Theorem \ref{thm:UNBFUNCCALC} with $\Omega = \sigma_S(T) \cap \CC_{\gI}^+$ and $\sA = \sB(\sigma_S(T) \cap \CC_{\gI}^+)$ to obtain (i)-(v).

Assertions (vi)-(viii) are straight forward adaptations of \\Theorem \ref{thm:BFUNC}(viii)-(xi) to the unbounded setting. The corresponding proofs can be completed in much the same way as the bounded setting.
\end{proof}

\begin{cor}[subclasses of bounded normal operators]
\label{cor:UNBOUNDEDSUBCLASSES}
Let $T, E$ and $J$ be as in Theorem \ref{thm:UNBOUNDEDNORMAL} and $\Omega_{\gI}^+ := \sigma_S(T) \cap \CC_{\gI}^+$ for $\gI \in \mathbb{S}$. We have the following{\rm :}
\begin{enumerate}
\item[(i)] $T$ is self-adjoint if and only if $\sigma_S(T) \subseteq \RR$. In this case,
$$T = \int_{\RR} t \, dE(t).$$
\item[(ii)] $T$ is positive if and only if $\sigma_S(T) \subseteq [0, \infty)$. In this case,
$$T = \int_{[0,\infty)} t \, dE(t).$$
\item[(iii)] $T$ is anti self-adjoint if and only if $\sigma_S(T) \subseteq \{s \in \RR^{n+1}: {\rm Re}(s) = 0 \}$. In this case,
$$T = \int_{\RR}  t \, dE(t) J.$$
\end{enumerate}
\end{cor}

\begin{proof}
The proofs of (i)-(iii) can be completed in much the same way of the proofs of Corollary \ref{cor:SUBCLASSES}(i)-(iii) with Theorem \ref{thm:UNBOUNDEDNORMAL} in place of Theorem \ref{thm:BN}.
\end{proof}

The following corollary is an unbounded analogue of the Teichm\"uller decomposition.

\begin{cor}
\label{cor:AJBunbounded}
Corresponding to any normal operator $T\in \cL(\cH_n)$, there exist a self-adjoint operator $A \in \cL(\cH_n)$, an imaginary operator $J \in \cB(\cH_n)$ and a positive operator $B \in \cL(\cH_n)$ such that $A$, $J$ and $B$ strongly commute and satisfy
\begin{equation}
\label{eq:AJBunbounded}
T = A  +J B.
\end{equation}
In this case, $A = \int_{\CC_{\gI}^+} {\rm Re}(\lambda) \, dE(\lambda)$ and $B = \int_{\CC_{\gI}^+} {\rm Im}(\lambda) \, dE(\lambda)$, where $E$ is the spectral measure for $T$, and
\begin{equation}
\label{eq:TJunbounded}
\text{$T$ and $J$ strongly commute.}
\end{equation}
\end{cor}

\begin{proof}
The additive decomposition \eqref{eq:AJBunbounded} follows at once from \eqref{eq:SPECNZ}. The fact that $A$, $B$ and $J$ strongly commute follows from the fact that $Z_A = A_{Z_T}$, $Z_B = B_{Z_T}$, where $A_{Z_T}$ and $B_{Z_T}$ are as in \eqref{eq:ZAJB}, and the fact that $A_{Z_T}, B_{Z_T}$ and $J$ mutually commute. Finally, \eqref{eq:TJunbounded} is an immediate consequence of \eqref{eq:AJBunbounded} and the fact that $A, B$ and $J$ strongly commute
\end{proof}

\begin{thm}[spectral mapping theorem for an unbounded normal operator]
Let $T \in \cL(\cH_n)$ be normal with spectral measure $E$ and $\Omega_{\gI}^+ := \sigma_S(T) \cap \CC^+_{\gI}$, where $\gI \in \mathbb{S}$. For any intrinsic continuous function $f \in \gB(\Omega_{\gI}^+, \sB(\Omega_{\gI}^+), \CC_{\gI}, E)$, we have
\begin{equation}
\label{eq:SPECMAPPING}
\sigma_S(\II(f)) \cap \CC_{\gI} = \overline{f(\sigma_S(T) \cap \CC_{\gI})},
\end{equation}
If $T\in \cB(\cH_n)$ is normal, then \eqref{eq:SPECMAPPING} becomes
\begin{equation}
\label{eq:SPECMAPPING_compact}
\sigma_S(\II(f)) \cap \CC_{\gI} = f(\sigma_S(T) \cap \CC_{\gI}).
\end{equation}
\end{thm}

\begin{proof}
We begin by noting that image of $f$ is symmetric about the real axis in $\CC_{\gI}$ (see \eqref{eq:INTRINSICREAL} and \eqref{eq:INTRINSICIMAGINARY}). We will first show that
 $$
\overline{f(\sigma_S(T) \cap \CC_{{\gI}})} \subseteq \sigma_S(\II(f)) \cap \CC_{\gI}.
$$
 Suppose $w_0 \in \overline{f(\sigma_S(T) \cap \CC_{\gI} )}$ and $w_0 \in \CC_{\gI}^+$. For any $\varepsilon > 0$, we can find $\lambda_0 \in \sigma_S(T) \cap \CC_{\gI}$ such that $f(\lambda_0) \in \CC_{\gI}^+$ and
\begin{equation}
\label{eq:SQUEEZE}
|f(\lambda_0) - w_0 | < \frac{\varepsilon}{2}.
\end{equation}
Since $f$ is continuous, we have the existence of $\delta > 0$ such that
\begin{align*}
U_{\delta} :=& \; \{ \lambda \in \CC_{\gI}: |\lambda - \lambda_0 | < \delta \} \\
\subseteq& \; \{ \lambda \in \CC_{\gI}: |f(\lambda) - f(\lambda_0)| < \frac{\varepsilon}{2} \} \\
\subseteq& \; \{ \lambda \in \CC_{\gI}: |f(\lambda) - w_0| < \varepsilon \} \\
=: & \; V_{\varepsilon}.
\end{align*}
Since $w_0$ and $f(\lambda_0)$ both belong to $\CC_{\gI}^+$, we may proceed as in the verification of \eqref{eq:SPECINC} and show that
$$V_{\varepsilon} \cap \CC_{\gI}^+ = \{  \lambda \in \CC^+_{\gI}: |f(\lambda)^2 - 2{\rm Re}(w_0) f(\lambda) +| w_0|^2| < \varepsilon^2 \}.$$
Thus, as $\lambda_0 \in \sigma_S(T) \cap \CC_{\gI}^+$ and $\lambda_0 \in U_{\delta}$, we have that $E(U_{\delta}) \neq \emptyset$, in which case $E(V_{\varepsilon} \cap \CC_{\gI}^+) \neq \emptyset$. But then Theorem \ref{thm:IMPORTANT} ensures that $w_0 \in \sigma_S(\II(f)) \cap \CC_{\gI}^+$.

If $w_0$ and $f(\lambda_0)$ both belong to $\CC_{\gI}^-$, then we can use the fact that the image of $f$ is symmetric about the real axis and the fact that the $S$-spectrum is axially symmetric (see Remark \ref{rem:AXIALLYSYMMETRIC}) and repeat the argument above with $\bar{w}_0$ and $\bar{\lambda}_0$ in place of $w_0$ and $\lambda_0$, respectively, to arrive at the same conclusion.

We will now show that $\sigma_S(\II(f)) \cap \CC_{\gI} \subseteq\overline{f(\sigma_S(T) \cap \CC_{\gI})}$.  The inclusion $\sigma_S(\II(f)) \cap \CC_{\gI} \subseteq\overline{f(\sigma_S(T) \cap \CC_{\gI})}$ follows immediately from comparing the formula for $\sigma_S(\II(f) \cap \CC_{\gI}$
coming from Theorem \ref{thm:IMPORTANT} and the formula for $\overline{f(\sigma_S(T) \cap \CC_{\gI})}$ derived from \eqref{eq:NEWAB}.

Finally, \eqref{eq:SPECMAPPING_compact} drops out easily from the fact that if $T\in \cB(\cH_n)$, then $\sigma_S(T) \cap \CC_{\gI}^+$ is compact.
\end{proof}

\appendix

\section{Background on the $S$-spectrum and connections with functional calculi and function theory}
\label{app:Spectrum}

In 1936, Birkhoff and von Neumann,  in their paper \cite{BF}  on the logic of quantum mechanics, showed that an set-theoretic abstraction of quantum mechanics can be formulated on Hilbert spaces over the reals, complex numbers and quaternions. Consequently, there was a strong motivation to prove the spectral theorem in the quaternionic setting (i.e., a Clifford module over $\RR_2$) and since that time
 several attempts have appeared in the literature. The main contributions are due to O. Teichm{\"u}ller \cite{Teichmueller} in 1936
 and to K. Viswanath \cite{Viswanath} in 1971.
However, both authors do not make clear the notion of spectrum that is in use for quaternionic linear operators. Nevertheless, there are useful results on quaternionic operator theory in \cite{Teichmueller} and \cite{Viswanath}.

The major breakthrough came in 2006
when I. Sabadini and the first author discovered the $S$-spectrum  $\sigma_S(T)$ and the $S$-functional calculus for
a quaternionic linear operator $T$. A prime motivation for this investigation was to give quaternionic quantum mechanics a rigorous mathematical
foundation. The strategy for the identification of the $S$-spectrum
was purely based on  hyperholomorphic analysis methods and not on physical arguments (see the introduction of the book  \cite{6CKG} for a detailed explanation).
The definition of $S$-spectrum $\sigma_S(T)$ for an linear operator $T$ on a quaternionic Banach space $V$ is somewhat counter intuitive because it involves the second order operator
$$
Q_s(T):=T^2-2s_0T+|s|^2 \, I
$$
and is given by
$$
\sigma_S(T)=\{s\in \mathbb{H}\ : \ Q_s(T)\ \ {\rm is \ not \ invertible} \ \ \mathcal{B}(V)\},
$$
where $\mathbb{H}$ is the algebra of quaternions, $s_0$ the real part of $s\in \mathbb{H}$,
 $|s|^2$ is the modulus squared and $\mathcal{B}(V)$ is the space of all bounded linear operators.

Before 2006 in the literature there were two different notions of spectrum in the quaternionic setting (as well as in the Clifford setting):
 the left spectrum $\sigma_L(T)$ and the right spectrum $\sigma_R(T)$ and both definitions mimic the eigenvalue problem for
complex operators.
We point out that just in the finite dimensional case the quaternionic spectral theorem was precisely proved using the notion of right spectrum
by  D. R. Farenick, B. A. F. Pidkowich in \cite{fpp} that was published in 2003.
We also want to point out that in the literature on
 quaternionic quantum mechanics physicists used the right spectrum to
 describe the bounded states where there are just the eigenvalues,
 see the book of S. Adler \cite{adler} and for more recent advanced see the paper of J. Gantner \cite{JONAQS} on the equivalence of complex and quaternionic quantum mechanics.
Only in 2015 the authors with D. Alpay proved the spectral theorem for quaternionic normal operators based on the $S$-spectrum $\sigma_S(T)$ in
\cite{6SpecThm1} which was published in 2016.
Fairly recently, there has been a renewed quaternionic quantum mechanics, which utilises the notion of $S$-spectrum (see, e.g.,  \cite{santar2,santar3,santar4,santar1}).


\begin{thebibliography}{99}


\bibitem{adler}
 S. Adler,
  {\em Quaternionic Quantum Mechanics and Quaternionic Quantum Fields},
   Volume 88 of {\em International Series of Monographs on Physics}. Oxford University Press, New York.
1995.

\bibitem{ABCS1}
  D. Alpay, V. Bolotnikov, F. Colombo, I. Sabadini,
   {\em Self-mappings of the quaternionic unit ball: multiplier properties, the Schwarz-Pick inequality, and the Nevanlinna-Pick interpolation problem},
   Indiana Univ. Math. J., {\bf 64} (2015),  151--180.

\bibitem{6newresol}
D. Alpay, F. Colombo, J. Gantner, S. Sabadini,
 {\em A new resolvent equation for the S-functional calculus},
 J. Geom. Anal., {\bf 25} (2015), no. 3, 1939--1968.


\bibitem{6SpecThm1}
 D. Alpay, F. Colombo, D.P. Kimsey,
 {\em The spectral theorem for quaternionic unbounded normal operators based on the $S$-spectrum},
 J. Math. Phys., {\bf 57} (2016), no. 2, 023503, 27 pp.






\bibitem{6ACSBOOK}
D. {Alpay}, F. {Colombo},  I. {Sabadini},
 {\em Slice Hyperholomorphic Schur Analysis},
 Operator Theory: Advances and Applications, 256. Birkh\"auser/Springer, Cham, 2016. xii+362 pp.


\bibitem{6COFBook}
D. {Alpay}, F. {Colombo},  I. {Sabadini},
 {\em Quaternionic de Branges spaces and characteristic operator function},
 SpringerBriefs in Mathematics, Springer, Cham, 2020/21.




\bibitem{6hinfty}
D. Alpay, F. Colombo,  T. Qian, I. Sabadini,
{\em The $H^\infty$ functional calculus based on the S-spectrum for quaternionic operators and for n-tuples of noncommuting operators},
 J. Funct. Anal., {\bf 271} (2016), no. 6, 1544--1584.

\bibitem{Berberian}
S. K. Berberian,
{\em Notes on spectral theory}, Van Nostrand Mathematical Studies, No. 5, D. Van Nostrand Co., Inc., Princeton, N.J.-Toronto, Ont.-London, 1966.


 \bibitem{BF}
 G. Birkhoff, J. von Neumann, {\em The logic of quantum mechanics}, Ann. of
Math., {\bf 37} (1936), 823--843.






\bibitem{6DIRAC1}
   F. Brackx, R. Delanghe, F. Sommen,
   {\em  Clifford analysis}, Research Notes in Mathematics, 76. Pitman (Advanced Publishing Program), Boston, MA, 1982. x+308 pp.

\bibitem{Bulla}
 E. Bulla, D.  Constales, R. S. Krausshar, J. Ryan,
 {\em  Dirac type operators for Spin manifolds associated to congruence subgroups of generalized modular groups},
  J. Reine Angew. Math., {\bf 643} (2010), 1--19.

\bibitem{6CCKS}
 P. Cerejeiras, F. Colombo, U. K\"ahler, I.  Sabadini, {\em Perturbation of normal quaternionic operators},
  Trans. Amer. Math. Soc., { \bf 372} (2019), no. 5, 3257--3281.


\bibitem{Cohn} D. L. Cohn, {\em  Measure theory}. Second edition. Birkh\"auser Advanced Texts: Basler Lehrb\"ucher. [Birkh\"auser Advanced Texts: Basel Textbooks] Birkh\"auser/Springer, New York, 2013. xxi+457 pp.

\bibitem{6CG}
F. Colombo, J. Gantner,
{\em Quaternionic closed operators, fractional powers and fractional diffusion processes},
 Operator Theory: Advances and Applications, 274. Birkh\"auser/Springer, Cham, 2019. viii+322 pp.

\bibitem{frc1}
F. Colombo, D. Deniz-González, S. Pinton, {\em  The noncommutative fractional Fourier law in bounded and unbounded domains},
 Complex Anal. Oper. Theory, {\bf 15} (2021), no. 7, Paper No. 114, 27 pp.



\bibitem{6CKG}
F. Colombo, J. Gantner, D.P. Kimsey,
{\em Spectral theory on the S-spectrum for quaternionic operators},
 Operator Theory: Advances and Applications, 270. Birkh\"auser/Springer, Cham, 2018. ix+356 pp.

\bibitem{CGKSfunc}
F. Colombo, J. Gantner, D.P. Kimsey, I. Sabadini,
{\em  Universality property of the $S$-functional
calculus, noncommuting matrix variables and Clifford operators},
Preprint 2020.



\bibitem{CKPS}
F. Colombo, D.P. Kimsey, S. Pinton, I. Sabadini,
{\em Slice monogenic functions of a Clifford variable via the $S$-functional calculus},
 Proc. Amer. Math. Soc. Ser. B, {\bf 8} (2021), 281--296.

\bibitem{frc2}
F. Colombo, M. Peloso, S. Pinton,
{\em  The structure of the fractional powers of the noncommutative Fourier law},
 Math. Methods Appl. Sci., {\bf 42} (2019), no. 18, 6259--6276.


\bibitem{6DIRAC2}
F. Colombo, I. Sabadini, F. Sommen, D.C. Struppa,
 {\em Analysis of Dirac systems and computational algebra},
  Progress in Mathematical Physics, 39. Birkh\"auser Boston, Inc., Boston, MA, 2004. xiv+332 pp.

\bibitem{CSSf}
F. Colombo,  I. Sabadini, D.C. Struppa,
{\em Slice monogenic functions},
 Israel J. Math., {\bf 171} (2009), 385--403.

\bibitem{CSSd}
  F. Colombo,  I. Sabadini, D.C. Struppa,
 {\em An extension theorem for slice monogenic functions and some of its consequences},
  Israel J. Math., {\bf 177} (2010), 369--389.

\bibitem{CSSe}
   F. Colombo,  I. Sabadini, D.C. Struppa,
   {\em Duality theorems for slice hyperholomorphic functions},
   J. Reine Angew. Math., {\bf 645} (2010), 85--105.


\bibitem{jfaStruppa}
 F. Colombo, I. Sabadini, D.C. Struppa, {\em A new functional calculus for noncommuting operators}, J. Funct. Anal., {\bf 254} (2008),  2255--2274.

\bibitem{bookSCE}
F. Colombo, I. Sabadini,  D.C. Struppa,
{\em Michele Sce's Works in Hypercomplex Analysis 1955-1973.
A translation with commentaries},
 Birkh\"auser, 2020,  Hardcover ISBN
978-3-030-50215-7.


\bibitem{6EntireBook}
 F. Colombo, I. Sabadini, D. C. Struppa, {\em  Entire slice regular functions},
  SpringerBriefs in Mathematics. Springer, Cham, 2016. v+118 pp.



\bibitem{CSS} F. Colombo, I. Sabadini, D.C. Struppa,
{\em Noncommutative functional calculus. Theory and applications of slice hyperholomorphic functions},
     Progress in Mathematics, 289. Birkh\"auser/Springer Basel AG, Basel, 2011. vi+221 pp.

\bibitem{Conway}
J. B. Conway, {\em A course in functional analysis}, second ed., Graduate Texts in Mathematics, vol. 96, Springer-Verlag, New York, 1990.

\bibitem{6DIRAC3}
 R. Delanghe, F. Sommen, V.  Soucek,
 {\em Clifford algebra and spinor-valued functions. A function theory for the Dirac operator},
 Related REDUCE software by F. Brackx and D. Constales. With 1 IBM-PC floppy disk (3.5 inch).
 Mathematics and its Applications, 53. Kluwer Academic Publishers Group, Dordrecht, 1992. xviii+485 pp.

\bibitem{fpp}
D. R. Farenick, B. A. F. Pidkowich, {\em The spectral theorem in quaternions},
Linear Algebra Appl.,
{\bf 371} (2003), 75--102.

\bibitem{songerg}
 M. Fischmann,  C. Krattenthaler, P. Somberg, {\em  On conformal powers of the Dirac operator on Einstein manifolds}, Math. Z., {\bf 280} (2015), 825--839.


\bibitem{bookTF}
T. Friedrich, {\em Dirac operators in Riemannian geometry}.
Translated from the 1997 German original by Andreas Nestke. Graduate Studies in Mathematics, 25.
American Mathematical Society, Providence, RI, 2000. xvi+195 pp.




\bibitem{6GSBook}
S. Gal, I. Sabadini, {\em  Quaternionic approximation. With application to slice regular functions},
 Frontiers in Mathematics. Birkh\"auser/Springer, Cham, 2019. x+221 pp.


\bibitem{JONAQS}
J. Gantner, {\em On the equivalence of complex and quaternionic quantum mechanics},
 Quantum Stud. Math. Found.,{\bf 5} (2018), no. 2, 357--390.

\bibitem{6JONAME}
 J. Gantner, {\em
 Operator Theory on One-Sided Quaternionic Linear Spaces: Intrinsic S-Functional Calculus and Spectral Operators},
  Mem. Amer. Math. Soc., 267 (2020), no. 1297, 0.

  \bibitem{GSSBOOK}
 G. Gentili, C. Stoppato, D. C.  Struppa, {\em Regular functions of a quaternionic variable}.
  Springer Monographs in Mathematics. Springer, Heidelberg, 2013.

\bibitem{GhiloniRecupero}  R. Ghiloni,  V. Recupero, {\em Semigroups over real alternative *-algebras: generation theorems and spherical sectorial operators}, Trans. Amer. Math. Soc. {\bf 368} (2016), no. 4, 2645--2678.

\bibitem{GR} D. M. Giarrusso, J. D. Romano, {\em A polarization formula for Clifford modules}, Linear and Multilinear Algebra, {\bf 35} (1993), no. 3-4, 191–194. 4.

\bibitem{6DIRAC4}
 J. E. Gilbert,  M. A. M. Murray,
 {\em  Clifford algebras and Dirac operators in harmonic analysis},
  Cambridge Studies in Advanced Mathematics, 26. Cambridge University Press, Cambridge, 1991. viii+334 pp.


\bibitem{6DIRAC5}
 K. G\"urlebeck, K. Habetha,  W. Spr\"oßig,
  {\em Application of holomorphic functions in two and higher dimensions}, Birkh\"aäuser/Springer, [Cham], 2016. xv+390 pp.


\bibitem{6GURLYSPROSS}
 K. G\"urlebeck, W. Spr\"oßig, {\em Quaternionic Analysis and Elliptic Boundary
Value Problems}, International Series of Numerical Mathematics, 89.
Birkh\"auser Verlag, Basel, 1990, 253pp.






\bibitem{Halmos} P. R. Halmos, {\em Measure Theory}, D. Van Nostrand Company, Inc., New York, N. Y., 1950.

\bibitem{Homma}
  Y. Homma, {\em  The higher spin Dirac operators on 3-dimensional manifolds},
   Tokyo J. Math., {\bf 24} (2001), no. 2, 579--596.

\bibitem{6jefferies} B. Jefferies, {\em Spectral properties of noncommuting operators},
Lecture Notes in Mathematics, 1843, Springer-Verlag, Berlin, 2004.

\bibitem{6JM}
 B. Jefferies, A. McIntosh,  {\em The Weyl calculus and Clifford analysis},
Bull. Austral. Math. Soc., {\bf 57} (1998), 329--341.

\bibitem{6JMP}
B. Jefferies, A. McIntosh, J. Picton-Warlow,  {\em The monogenic functional
calculus}, Studia Math., {\bf 136} (1999), 99--119.


\bibitem{Kawai}
  S. Kawai, {\em  Spectrum of the Dirac operator on manifold with asymptotically flat end}, J. Geom. Phys. {\bf 110} (2016), 195--212.

\bibitem{KimFRI}
 E. C., Kim,  T. Friedrich,
  {\em The Einstein-Dirac equation on Riemannian spin manifolds},
   J. Geom. Phys., {\bf 33} (2000), 128--172.

 \bibitem{Kirchberg}
  K.-D. Kirchberg,
  {\em  Lower bounds for the first eigenvalue of the Dirac operator on compact Riemannian manifolds}, Differential Geom. Appl., {\bf 29} (2011), 374--387.


\bibitem{SpinGeometry} H. B. Lawson,  M.L. Michelsohn, {\em Spin Geometry}, Princeton Mathematical Series. 38. Princeton University Press, 2016.


\bibitem{Lax} P. D. Lax, {\em Functional analysis}, Pure and Applied Mathematics (New York), Wiley- Interscience [John Wiley \& Sons], New York, 2002.

\bibitem{Leitner}
 F. Leitner, {\em The first eigenvalue of the Kohn-Dirac operator on CR manifolds},
  Differential Geom. Appl. {\bf 61} (2018), 97--132.

\bibitem{6MQ}
C. Li, A. McIntosh, T. Qian,  {\em Clifford algebras, Fourier transforms and singular convolution operators on Lipschitz surfaces}, Rev. Mat. Iberoamericana, {\bf 10} (1994), 665--721.








\bibitem{6MP}
 A. McIntosh, A. Pryde,  {\em A functional calculus for several commuting
operators}, Indiana U. Math. J., {\bf 36} (1987), 421--439.

\bibitem{Mitrea}
 M. Mitrea, {\em Clifford wavelets, singular integrals, and Hardy spaces},
  Lecture Notes in Mathematics, 1575. Springer-Verlag, Berlin, 1994. xii+116 pp.

\bibitem{Moroianu}
  A. Moroianu, S. Moroianu, {\em  The Dirac spectrum on manifolds with gradient conformal vector fields}, J. Funct. Anal., {\bf 253} (2007), 207--219.




      \bibitem{santar2}
      Muraleetharan, B.; I. Sabadini, K. Thirulogasanthar, {\em
       S-spectrum and the quaternionic Cayley transform of an operator},
        J. Geom. Phys., {\bf 124} (2018), 442--455.

      \bibitem{santar3}
      B. Muraleetharan,  K. Thirulogasanthar,
      {\em Fredholm operators and essential S-spectrum in the quaternionic setting},
       J. Math. Phys., {\bf 59} (2018), no. 10, 103506, 27 pp.

      \bibitem{santar4}
      B. Muraleetharan, K. Thirulogasanthar, {\em  Weyl and Browder S-spectra in a right quaternionic Hilbert space}, J. Geom. Phys., {\bf 135} (2019), 7--20.

 \bibitem{Nakad}
  R. Nakad, J.  Roth,{ \em Lower bounds for the eigenvalues of the Spinc Dirac operator on manifolds with boundary}, C. R. Math. Acad. Sci. Paris 354 (2016), 425--431.

\bibitem{Nakad2}
  R. Nakad, J.  Roth, {\em Lower bounds for the eigenvalues of the Spinc Dirac operator on submanifolds}, Arch. Math. (Basel), {\bf 104} (2015), 451--461.

\bibitem{Pedersen} G. K. Pedersen, {\em Analysis now}, Graduate Texts in Mathematics, vol. 118, Springer- Verlag, New York, 1989.


 \bibitem{6qian1}    T. Qian, {\em Singular integrals on star-shaped Lipschitz surfaces in the quaternionic space}, Math. Ann., {\bf 310} (1998), 601--630.



\bibitem{BOOKTAO}
T. Qian,  P. Li,
{\em Singular integrals and Fourier theory on Lipschitz boundaries},
 Science Press Beijing,
 Beijing; Springer, Singapore, 2019. xv+306 pp.





\bibitem{6DIRAC6}
R. Rocha-Chavez, M. Shapiro, F. Sommen,
  {\em Integral theorems for functions and differential forms},
  in Cm. Chapman \& Hall/CRC Research Notes in Mathematics, 428. Chapman \& Hall/CRC, Boca Raton, FL, 2002. x+204 pp.




\bibitem{Schmuedgen} K. Schm\"udgen, { \em Unbounded self-adjoint operators on Hilbert space}, Graduate Texts in Mathematics, vol. 265, Springer, Dordrecht, 2012.


\bibitem{Teichmueller}
O. Teichm{\"u}ller, {\em Operatoren im {W}achsschen {R}aum} (German), J. Reine
  Angew. Math. {\bf 174} (1936), 73--124.

\bibitem{santar1}
       K. Thirulogasanthar, B. Muraleetharan, {\em Squeezed states in the quaternionic setting}, Math. Phys. Anal. Geom., {\bf 23} (2020), no. 1, Paper No. 8, 26 pp.




\bibitem{Viswanath}
K.~Viswanath, {\em Normal operations on quaternionic Hilbert spaces}, Trans.
  Amer. Math. Soc. {\bf162} (1971), 337--350.


\end{thebibliography}
\end{document}